\documentclass[10pt,reqno,a4paper]{amsart}
          
\usepackage{amsmath,amsrefs,amssymb,tensor, eufrak, dsfont}
\usepackage{graphicx}
\usepackage{xcolor}

\usepackage[colorlinks=true,linkcolor=red,citecolor=blue,urlcolor=teal,bookmarks=false,hypertexnames=true,pdfborder={0 0 0}]{hyperref}
\usepackage{cleveref}

\numberwithin{equation}{section}

\DeclareMathOperator{\Scal}{S}
\DeclareMathOperator{\Ricci}{Ric}
\DeclareMathOperator{\Riemann}{R}
\DeclareMathOperator{\Weyl}{W}
\DeclareMathOperator{\Schouten}{P}

\DeclareMathOperator{\Sym}{Sym}

\DeclareMathOperator{\Bach}{B}

\DeclareMathOperator{\bigO}{O}
\DeclareMathOperator{\smallo}{o}

\newcommand{\ua}{u_{\alpha}}
\newcommand{\Ba}{B_{\alpha}}
\newcommand{\va}{v_{\alpha}}
\newcommand{\hua}{\hat{u}_{\alpha}}

\newcommand{\hva}{\hat{v}_{\alpha}}
\newcommand{\tva}{\tilde{v}_{\alpha}}
\newcommand{\ga}{g_\alpha}
\newcommand{\hga}{\hat{g}_\alpha}
\newcommand{\tga}{\tilde{g}_\alpha}
\newcommand{\xa}{x_{\alpha}}
\newcommand{\ya}{y_{\alpha}}
\newcommand{\za}{z_{\alpha}}
\newcommand{\rhoa}{\rho_{\alpha}}
\newcommand{\ra}{r_{\alpha}}
\newcommand{\da}{d_{\alpha}}
\newcommand{\ma}{\mu_{\alpha}}
\newcommand{\na}{\nu_{\alpha}}

\newcommand{\ve}{\varepsilon}
\newcommand{\vp}{\varphi}

\newcommand{\R}{\mathbb{R}}

\newcommand{\N}{\mathbb{N}}

\renewcommand{\[}{\left[}
\renewcommand{\]}{\right]}
\renewcommand{\(}{\left(}
\renewcommand{\)}{\right)}

\newcommand{\ds}{\displaystyle}

\newtheorem{theorem}{Theorem}[section]
\newtheorem{proposition}[theorem]{Proposition}
\newtheorem{lemma}[theorem]{Lemma}
\newtheorem{claim}[theorem]{Claim}
\newtheorem{remark}[theorem]{Remark}

\begin{document}

\title[Compactness for GJMS]{Compactness of conformal metrics with constant $Q$-curvature of higher order}

\author{Saikat Mazumdar}

\address{Saikat Mazumdar, Department of Mathematics, Indian Institute of Technology, Bombay, Powai, Mumbai, Maharashtra 400076, India}
\email{saikat.mazumdar@iitb.ac.in, saikat@math.iitb.ac.in}

\author{Bruno Premoselli}

\address{Bruno Premoselli, D\'epartement De Math\'ematiques, Universit\'e Libre de Bruxelles}
\email{bruno.premoselli@ulb.ac.be}


\begin{abstract}
Let $k\ge1$ be a positive integer and let $P_g$ be the GJMS operator $P_{g}$  of order $2k$ on a closed Riemannian manifold $(M,g)$ of dimension $n>2k$. We investigate the compactness of the set of metrics conformal to $g$ with prescribed constant positive $Q$-curvature of order $2k$ --- or, equivalently, of the set of positive solutions for the $2k$-th order $Q$-curvature equation. 
Under a natural positivity-preserving condition on $P_{g}$ we establish compactness, for an arbitrary $1 \le k < \frac{n}{2}$, under the following assumptions: 
\begin{itemize}
\item
$(M,g)$ is locally conformally flat and $P_g$ has positive mass in $M$. 
\item
$2k+1 \le n \le 2k+5$ and $P_g$ has positive mass in $M$, whenever the mass is defined. 
\item 
$n \ge 2k+4$ and the Weyl tensor $\Weyl_g$ never vanishes in $M$.
\end{itemize}
For an arbitrary $1 \le k < \frac{n}{2}$ the expression of $P_g$ is not explicit, which is an obstacle to proving compactness. We overcome this by relying on Juhl's recursive formulae for $P_g$ to perform a refined blow-up analysis for solutions of the $Q$-curvature equation and to prove a Weyl vanishing result for $P_g$. 
Our result also hints that the threshold dimension for compactness for the $2k$-th order $Q$-curvature equation diverges as $k \to + \infty$. 
\end{abstract}

\maketitle

\section{Introduction}

Let $k\geq 1$ be an integer and let $(M,g)$ be a compact Riemannian manifold of dimension $n\geq 2k+1$ (i.e. closed, smooth and without boundary).  We will denote by $P_{g}$ the conformally covariant GJMS operator of order $2k$ introduced by Graham, Jenne, Mason and Sparling in the seminal work \cite{GJMS} and by $Q_g$ the associated $Q$-curvature of order $2k$ in $(M,g)$. The purpose of this paper is to prove a compactness result for the set of conformal metrics to $g$ with constant $Q$-curvature, for any order $k \ge 1$. We first recall a few properties of $P_g$. Its construction is based on the Fefferman--Graham ambient metric \cite{FeffermanGraham, FeffermanGraham2}. It is an elliptic self-adjoint differential operator with leading order term $\Delta_{g}^{k}$, where $\Delta_g = - \text{div}_g(\nabla \cdot)$ is the Laplace-Beltrami operator, and is conformally covariant, that is,  if $u \in C^\infty(M)$, $u>0$ and $\hat{g}:= u^{\frac{4}{n-2k}}g$, then 
\begin{equation} \label{conf.inv.Pg} 
P_{\hat g} (f)= u^{- \frac{n+2k}{n-2k}} P_g(u f)~\hbox{ for all } f \in C^\infty(M).
\end{equation}
When $k=1$, $P_g$ is the conformal Laplacian
$$ P_g = \Delta_g + \frac{n-2}{4(n-1)} S_g,$$
where $S_g$ denotes the scalar curvature of $(M,g)$, while for $k=2$, $P_g$ is the Paneitz operator \cite{Branson1, Paneitz}. 
Explicit formulas for $P_{g}$ on any manifold $(M,g)$ are only known for low values of $k$, see e.g. Branson \cite{Branson1,Branson2,Branson3}, Gover--Peterson \cite{GoverPeterson} or Paneitz \cite{Paneitz}, and inductive algebraic formulas expressing $P_{g}$ as a linear combination of compositions of second-order differential operator were obtained by Juhl \cite{JuhlGJMS} (see also Fefferman--Graham \cite{FeffermanGraham2}). When $k \ge 5$, and on any manifold $(M,g)$, there is no explicit expression of $P_g$ yet, since the complexity of its coefficients increases drastically with $k$. In the special case of Einstein manifolds (see Gover \cite{Gover} or again \cite{FeffermanGraham}), $P_{g}$ factorizes  as a product of second-order operators:
\begin{align*}
P_{g}= \prod_{i=1}^{k}\(\Delta_{g}+\frac{(n+2i-2)(n-2i)}{4n(n-1)}S_{g}\).
\end{align*}
The $Q$-curvature $Q_{g}$ associated to $P_{g}$ is defined via the zeroth-order terms of $P_{g}$ as $Q_{g}:=\frac{2}{n-2k}P_{g}(1)$.  When $k=1$, $Q_{g}=S_{g}$ (up to a positive constant), and for $k=2$, $Q_{g}$ was introducted by Branson and \O rsted \cite{BransonOrsted} and generalized by Branson \cite{Branson3}. For more details regarding the $Q$-curvature and its significance, we refer to Juhl \cite{JuhlBook}. See also Case--Gover \cite{CaseGover} for a very recent survey on GJMS operators. 
\smallskip

The question of the existence of a metric conformal to $g$ with constant scalar curvature has attracted a lot of attention in recent years and was strongly motivated by the original resolution of the Yamabe problem. In analytical terms, if $u$ is a smooth positive function in $M$, a conformal metric $\tilde{g} = u^{\frac{4}{n-2k}}g$ has constant \emph{positive} $Q$-curvature if and only if $u$, possibly up to scaling by a constant depending on $n$ and $k$, is a positive solution of the $2k$-th order $Q$-curvature equation
\begin{align}\label{eq:zero}
P_{g}u=u^{2^{*}_{k}-1} \quad \text{ in } M,
\end{align}
where  $2^{*}_{k}:=2n/\(n-2k\)$ is the critical Sobolev exponent. The case $k=1$ is the celebrated Yamabe problem \cite{Yamabe} and was solved in full generality by Trudinger \cite{Trudinger}, Aubin \cite{AubinYamabe} and Schoen \cite{SchoenYamabe}. For $k=2$, existence of a conformal metric with constant $Q$-curvature was proven in Gursky--Malchiodi \cite{GurskyMalchiodi} assuming $S_g \ge 0$
 and $Q_g \ge 0, Q_g \not \equiv 0$, and the assumption $S_g \ge 0$ was later relaxed in Hang--Yang  \cite{HangYang} to the positivity of the Yamabe invariant. Solutions of the constant $Q$-curvature equation are not unique in general: multiplicity results have been obtained by Schoen \cite{SchoenNumberMetrics} when $k=1$ and in Alarcon--Petean--Rey \cite{AlarconPeteanRey}, Andrade--Case--Piccione--Wei \cite{AndradeCasePiccioneWei}, Batalla--Petean \cite{BatallaPetean}, Bettiol--Piccione--Sire \cite{BettiolPiccioneSire} when $k\ge2$ (see also the discussion after Theorem 1.3 in \cite{CaseMalchiodi}). However, solutions to \eqref{eq:zero} are unique if $(M,g)$ is conformally Einstein as proven by Obata \cite{Obata} when $k=1$ and recently generalised by V\'etois \cite{VetoisUniqueness} for $k=2$. In the case $k=3$ partial existence results for \eqref{eq:zero} were obtained in Chen--Hou \cite{ChenHou}. For an arbitrary $k<\frac{n}{2}$, the first existence result of constant $Q$-curvature metrics in $[g]$ is due to Qing--Raske in \cite{QingRaske} when $(M,g)$ is locally conformally flat. For an arbitrary $k < \frac{n}{2}$, the most general existence result for \eqref{eq:zero} to this day was recently obtained in \cite{MazumdarVetois} assuming a positive mass assumption if $2k+1 \le n \le 2k+3$ and that the Weyl tensor has positive norm somewhere when $n \ge 2k+4$. 

\smallskip

In this paper we consider the question of compactness of the full set of conformal metrics with constant positive $Q$-curvature for an arbitrary $1 \le k < \frac{n}{2}$. Analytically speaking, this amounts to showing the compactness in strong spaces of the full set of positive solutions for the $2k$-th order $Q$-curvature equation \eqref{eq:zero}. Throughout this paper we will assume that $P_g$ satisfies the following strong positivity assumption:
\begin{align}\label{positivity}
\text{Ker$(P_{g})=\{0\}$ and the Green's function $G_{g}$ of $P_{g}$ is positive in $M$.}
\end{align}
We recall that the Green's function of $P_g$ is the unique function $G_g \in C^\infty(M \times M \backslash \{x=y\})$ satisfying $P_g G_g(x, \cdot) = \delta_x$ for all $x \in M$. We say that $G_g$ is positive if $G_g(x,y) >0$ for all $x \neq y$. Assumption \eqref{positivity} is natural when working with polyharmonic problems: it implies that $P_g$ is coercive and satisfies the maximum principle, that is, if $u\geq 0$ satisfies $P_{g}u>0$ in $M$ then either $u>0$ or $u\equiv 0$ in $M$. Unlike in the case $k=1$, the polyharmonic operator $P_g$ does not satisfy the maximum principle in general when $k \ge 2$. Assumption \eqref{positivity} is obviously satisfied for the standard round sphere and its finite quotients.  Examples of product manifolds for which $P_g$ satisfies \eqref{positivity} are in Case--Malchiodi \cite{CaseMalchiodi}. Note also that if the $k$-th  Yamabe invariant satisfies $Y_{2k}(M,g)>0$ then Ker$(P_{g})=\{0\}$. For $k=2$, \eqref{positivity} was shown to hold in \cite{HangYang} (see also \cite{GurskyMalchiodi}) provided $(M, g)$ is of positive Yamabe type and $Q_g \geq 0$ in $M$, $Q_g\not\equiv 0$.

\medskip

Our compactness theorem for the constant $Q$-curvature equation \eqref{eq:zero} is as follows: 

\begin{theorem} \label{theo.main}
Let $(M,g)$ be a closed Riemannian manifold of dimension $n \ge 3$ and let $k$ be a positive integer such that $2k < n$. Let $P_g$ be the GJMS operator of order $2k$ and assume that it satisfies the strong positivity assumption \eqref{positivity}. Suppose one of the following three assumptions holds: 
\begin{itemize}
\item
$(M,g)$ is locally conformally flat and $P_g$ has positive mass at every point. 
\item
$2k+1 \le n \le 2k+3$ and $P_g$ has positive mass at every point. 
\item
$2k+4 \le n \le 2k+5$ and $P_g$ has positive mass at every point where the Weyl curvature $\Weyl_{g}$ of $(M,g)$ vanishes. 
\item 
$n \ge 2k+4$ and the Weyl curvature $\Weyl_{g}$ of $(M,g)$ never vanishes in $M$. 
\end{itemize}
Let $2 < p_0 < 2^{*}_{k}$. Then there exists a constant $C >0$ depending only on $n,k,g,p_0$ such that every positive solution $u\in C^{2k}(M)$ of 
$$ P_g u = u^{p }~\text{ in } M, $$
with $p_0 \le p \le 2^{*}_{k}$, satisfies 
\begin{align*}
\Vert u \Vert_{C^{2k}(M)} + \Vert 1/u \Vert_{C^{2k}(M)} \le C.
\end{align*}
\end{theorem}

The positive mass assumption in results like Theorem \ref{theo.main} has been known to be crucial to ensure compactness. We investigate in detail the Green's function $G_g$ of $P_{g}$ in Appendix \ref{sec.green.expan}, and we refer to Proposition \ref{expansion.Green.local} for the definition of the mass and to \eqref{positive:mass} for the meaning of our positive mass assumption. The mass of $P_g$ at a point $\xi \in M$ is defined, as usual, as the constant term in the expansion of $G_g$ in conformal normal coordinates at $\xi$: it is well-defined at any point $\xi \in M$ in dimensions $2k+1 \le n \le 2k+3$, while it is only well-defined at points $\xi \in M$ where the Weyl tensor vanishes in dimensions $2k+4 \le n \le 2k+5$.  
When $k=1,2$ the assumptions of Theorem \ref{theo.main} are satisfied by the positive mass theorems of Schoen--Yau \cite{SchoenYau} (for $k=1$) and Gong--Kim--Wei \cite{GongKimWei} (for $k=2$). When $k \ge 3$ examples of manifolds where \eqref{positivity} is satisfied and $P_g$ has positive mass are given by non-trivial finite quotients of the round sphere (see Michel \cite[Section 4]{Michel}). Also, Qing-Raske \cite{QingRaske2, QingRaske} gave examples of locally conformally flat manifolds with positive Yamabe constant where \eqref{positivity} is satisfied, obtained by identifying it as an infinite quotient of an open domain on the round sphere. Obviously, Theorem \ref{theo.main} does not apply to the round sphere where the mass of $P_g$ vanishes everywhere and the set of solutions of \eqref{eq:zero} is non-compact since the conformal group of the round sphere is non-compact. To the best of our knowledge a positive mass theorem for $P_g$ when $k \ge 3$ is not yet available.

\smallskip

Compactness results for constant $Q$-curvature equations like \eqref{eq:zero} have originated a vast amount of work in the last 30 years. They are structural results that are important on their own but also allow to compute the Leray-Schauder degree of \eqref{eq:zero} and to prove strong Morse inqualities, as was for example done in Khuri--Marques--Schoen \cite{KhuMaSc} for the Yamabe equation. We briefly review these compactness results to put Theorem \ref{theo.main} in perspective. When $k=1$, compactness for the Yamabe equation was gradually established over several years through the works of Schoen \cite{SchoenVariational}, Li--Zhu  \cite{LiZhu}  when $n = 3$, Druet \cite{DruetYlowdim}  when $n\le5$, Marques \cite{Marques} when $n\le7$, Li--Zhang \cite{LiZhang1,LiZhang2} when $n\le11$ and finally Khuri--Marques--Schoen \cite{KhuMaSc} when $n\le24$. A key ingredient in all these proofs is the positive mass theorem of \cite{SchoenYau}. Dimension $24$ is the threshold dimension for compactness, and examples where non-compactness holds were constructed when $n\ge25$ by Brendle \cite{Brendle} and Brendle--Marques \cite{BrendleMarques}. The case $k=2$ was addressed more recently. The first compactness results for \eqref{eq:zero} when $k=2$ were independently obtained by Qing--Raske \cite{QingRaske2} and Hebey--Robert \cite{HebeyRobert} in the locally conformally flat case. Assuming \eqref{positivity} and a positive mass assumption, compactness was then shown in Li--Xiong \cite{LiXiong} for $5\le n\le9$. The full compactness of \eqref{eq:zero} when $k=2$ was recently extended by Gong--Kim--Wei in \cite{GongKimWei} to all dimensions $5\le n \le 24$. In \cite{GongKimWei} the authors also establish a positive mass theorem in the case $k=2$ assuming only that $(M, g)$ is of positive Yamabe type and $Q_g \geq 0$ on $M$, $Q_g\not\equiv 0$, building up for this on previous work by Humbert--Raulot \cite{HumbertRaulot} and Avalos--Laurain--Lira \cite{AvalosLaurainLira}. When $k=2$, remarkably, $n=24$ again emerges as the threshold dimension for compactness: examples of non-compactness when $n \ge 25$ have indeed been known since Wei--Zhao \cite{WeiZhao}. In the case $k=3$, the compactness of \eqref{eq:zero} has been very recently announced in \cite{GongKimWei} in all dimensions $7 \le n \le 26$, assuming the validity of the positive mass theorem. The author also constructed in \cite{GongKimWei} counter-examples to compactness when $n \ge 27$, thus establishing $n=26$ as the threshold dimension for compactness when $k=3$. This series of works thus almost entirely settles the question of compactness for \eqref{eq:zero} under assumption \eqref{positivity} (assuming the positive mass theorem when $k=3$). 

\medskip

A key aspect of the previously cited works proving compactness results when $1 \le k \le 3$ is that determining the threshold dimension for compactness requires knowing the explicit algebraic structure of the linearised equation of \eqref{eq:zero} at a positive bubble: in \cite{KhuMaSc, GongKimWei} this is achieved via the explicit expression of, respectively, the Yamabe, Paneitz and sixth-order GJMS operator. In Theorem \ref{theo.main}, by contrast, we establish the compactness of \eqref{eq:zero} for an arbitrary $k < \frac{n}{2}$, hence in a setting where an explicit expression of $P_g$ is not available. We make a crucial use of Juhl's recursive formulae \cite{JuhlGJMS} for this. As is classical with compactness results like Theorem \ref{theo.main}, our proof goes by contradiction and relies first on a local a priori blow-up analysis and then on a local sign restriction argument. We construct a suitable set of concentration points for a blowing-up sequence of solutions of \eqref{eq:zero} and we show that there are only finitely many such points and that they are isolated by proving a Weyl vanishing result at concentration points. We then conclude by a global argument, using a Pohozaev identity around each point in combination with the positive mass assumption. This strategy of proof is now well-established: it was pioneered in \cite{LiZhu, DruetYlowdim, Marques, LiZhang1, LiZhang2, KhuMaSc} for the Yamabe equation and then applied when $k=2,3$ in \cite{LiXiong, GongKimWei}. In our case, however, the main difficulty is the lack of an explicit expression for $P_g$: this is a major obstacle to adapt this strategy of proof to the general polyharmonic case, in particular since proving the Weyl vanishing result typically requires to prove the coercivity of an \emph{explicit} quadratic form arising from Pohozaev identity. We overcome this by decomposing $P_g$ using Juhl's recursive formulae \cite{JuhlGJMS}: we crucially use them to obtain an expansion of $P_g$ at sixth order in conformal normal coordinates around any point in $M$, where fourth-order terms determine the Pohozaev quadratic form and fifth-order ones are antisymmetric. This is done in Proposition \ref{GJMS.exp0} below. The fourth-order terms that we obtain in this way are however still too cumbersome to be directly dealt with, due to their rising complexity as $k$ grows. We overcome this by observing that we actually \emph{do not need} to estimate these terms: we prove indeed that the Pohozaev quadratic form is equal, to fourth-order, to the formal differentiation of the energy of a single bubbling profile with respect to its concentration parameter; this is shown in Proposition \ref{GJMS.exp2} below (see also \eqref{GJMS.exp2.5}). This is one of the key observations in our work and it is what allows our strategy to work. The fourth-order expansion of the energy of a single positive bubbling profile (suitably defined in $M$) was not known until recently and it was first computed in \cite{MazumdarVetois}, a work we strongly rely on. As a byproduct of our expansion of $P_g$ we describe the asymptotic behavior of the Green's function $G_g$ around any point (see Proposition \ref{expansion.Green.local} below). The purely analytical side of our proof closely follows the approach for the polyharmonic case that was laid out when $k=2$ in \cite{LiXiong}, in particular when it comes to the integral formulation of \eqref{eq:zero} to obtain Harnack-type inequalities. We extend the results of \cite{LiXiong} to the polyharmonic setting of an arbitrary $k < \frac{n}{2}$ and establish results that will be used in future compactness results for \eqref{eq:zero}, including a general Pohozaev identity in Appendix \ref{sec.pohozaev}. 

 Theorem \ref{theo.main} covers dimensions up to $2k+5$ since we were only able to use Juhl's formulae to expand $P_g$ at sixth order in conformal normal coordinates: as such it should be understood as the analogue for an arbitrary $k < \frac{n}{2}$ of the results \cite{Marques} ($k=1$) and \cite{LiXiong} (for $k=2$). Improved expansions for $P_g$ in conformal normal coordinates would, in principle, yield symmetric estimates of any order as was done in \cite{KhuMaSc, GongKimWei}. In the case of an arbitrary $k < \frac{n}{2}$ the main obstacle to generalising our approach to dimensions higher than $2k+5$ lies for the moment in the difficulty of obtaining an explicit expression of the lower order operators arising in Juhl's formulae (see \eqref{Pr1Step1Eq1:bis} below). Nevertheless, Theorem \ref{theo.main} is the first compactness result for \eqref{eq:zero} for an arbitrary $k < \frac{n}{2}$. Its proof demonstrates the applicability of Juhl's formulae to the careful study of the behavior of blowing-up sequences of solutions of \eqref{eq:zero}.  
  It is not yet clear what the threshold dimension for compactness would be for a $4 \le k < \frac{n}{2}$ but in \cite{GongKimWei} the authors conjecture it to be $2k+20$.

%

\medskip

The paper is organised as follows. In Section \ref{sec.prelims} we recall the relevant background and fix our notations. In Section \ref{sec.GJMS.expan} we prove the improved expansion of $P_{g}$ in conformal normal coordinates.  In Section \ref{sec.local.analysis} we begin the proof of our compactness result and construct a canonical family of suitable concentration points for a blowing-up sequence of solutions of \eqref{eq:zero} around which we perform an asymptotic analysis. The crucial symmetry estimates are proved in Section \ref{sec.order.two}, and we obtain improved pointwise estimates around a blow-up point in Proposition \ref{sym.es}.  In Section \ref{sec.estimate.weyl} we obtain estimates on the Weyl curvature around a concentration point using the symmetry estimates and the Pohozaev identity. In Section \ref{sec.final.compactness} we show that concentration points are isolated and conclude the proof of Theorem \ref{theo.main}. The Appendix contains technical results that are used throughout the paper. In Appendix \ref{sec.pohozaev} we extend the Pohozaev identity to polyharmonic operators. The expansion of the Green's function in conformal normal coordinates is given in Appendix  \ref{sec.green.expan}. We conclude the paper with a Giraud-type lemma in Appendix \ref{sec.giraud}. 
\medskip

\subsection*{Acknowledgements:} 
S. M. gratefully acknowledges the support from the MATRICS grant MTR/2022/000447 of the Science and Engineering Research Board (currently ANRF) of India. Part of this work was carried out during S.M.'s visits to Universit\'e Libre de Bruxelles (ULB), and S.M. is grateful for the support and hospitality provided by ULB. B. P. was supported by the Fonds Th\'elam, an ARC Avanc\'e 2020 grant and an EoS FNRS grant.


The authors are truly indebted to the anonymous referee whose numerous observations greatly improved the quality of the paper.

\section{Preliminaries and Notations}\label{sec.prelims}

We first recall the definition of conformal normal coordinates. We denote by $\Scal_{g}$, $\Ricci_{g}$, $\Weyl_{g}$ respectively the scalar curvature and the Ricci and Weyl curvature tensors of $\(M,g\)$. Fix $N>2$ large.  Following Lee--Parker \cite{LeeParker} (see also Cao \cite{Cao}, G\"{u}nther \cite {Gunther}), there exists $\Lambda\in C^{\infty}(M\times M)$ such that, defining $\Lambda_{\xi}:=\Lambda(\xi,\cdot)$, we have for all $\xi \in M$
\begin{align}\label{conf.1}
\Lambda_{\xi}>0, ~\Lambda_{\xi}(\xi)=1~ \quad \hbox{and}~ \quad   \nabla \Lambda_{\xi}(\xi)=0,
\end{align}
and that the conformal metric $g_{\xi}:=\Lambda_{\xi}^{\frac{4}{n-2k}}g$ satisfies 
\begin{equation}\label{conf.2}
\det g_{\xi}\(x\)=1+\bigO(|x|^{N}) 
\end{equation}
around $0$ in geodesic normal coordinates given by the the exponential map $\exp_{\xi}^{g_{\xi}}$ at $\xi$ with respect to the metric $g_{\xi}$. Moreover (see \cite{LeeParker}),
\begin{align}\label{conf.3}
&\Scal_{g_{\xi}}\(\xi\)=0,~\nabla \Scal_{g_{\xi}}\(\xi\)=0,~\Ricci_{g_{\xi}}\(\xi\)=0, ~\Sym\nabla\Ricci_{g_{\xi}}\(\xi\)=0, \text{ and }  \notag \\
& ~\Sym\((\Ricci_{g_{\xi}})_{ab;cd}\(\xi\)+\frac{2}{9}(\Weyl_{g_{\xi}})_{eabf}\(\xi\)\tensor{\Weyl}{^e_{cd}^f}\(\xi\)\)=0.
\end{align}
\noindent 
It follows from tracing the last equation in \eqref{conf.3} that the following relation holds: 
$$ ~\Delta_{g_{\xi}} \Scal_{g_{\xi}}\(\xi\)=\frac{1}{6}\left|\Weyl_{g_{\xi}}\(\xi\)\right|_g^2 .$$ 
Throughout the paper we will let $2^{*}_{k}=2n/\(n-2k\)$ and $\Delta_{0}$ will denote the non-negative Euclidean Laplacian: $\Delta_0 = - \sum_{i=1}^n \partial_i^2$. For $x \in \R^n$ we define 
\begin{align}\label{bubble1}
U\(x\):={\big(1+\mathfrak{c}_{n,k}^{-1}\left|x\right|^2\big)^{- \frac{n-2k}{2}}}\quad, \hbox{ where }\mathfrak{c}_{n,k}=\[\prod \limits_{j=-k}^{k-1}\(n+2j\)\]^{1/k}.
\end{align}
By the classification result in Wei-Xu \cite[Theorem 1.3]{WeiXu} $U$ is, up to translations and rescalings, the unique $C^{2k}(\R^{n})$ positive solution of 
\begin{equation} \label{eq.bubble1}
\Delta_0^k\,U=U^{2^*_k-1}\quad\text{in }\R^n.
\end{equation}
It is in particular the unique solution of \eqref{eq.bubble1} that satisfies  $U(0)=1$ and $0 < U(x) \le 1$ for every $x \in \R^n$. The fundamental solution of $\Delta_0^k$ in $\R^n$ centered at $0$ is given by $G_0(x) = b_{n,k} |x|^{2k-n}$, where 
\begin{equation} \label{def.bnk}
 b_{n,k}^{-1} =2^{k-1}(k-1)! \prod_{i=1}^k (n-2i) \omega_{n-1} 
 \end{equation}
and $\omega_{n-1}$ is the area of the standard sphere $\mathbb{S}^{n-1}$. A simple argument using a representation formula for \eqref{eq.bubble1} (see e.g. Premoselli \cite[Lemma 2.1]{Premoselli13}) thus shows that 
\begin{equation} \label{eq.constants.bubble}
\mathfrak{c}_{n,k}^{\frac{n-2k}{2}} = b_{n,k} \int_{\R^n} U^{2^{*}_{k}-1}\, dx.
\end{equation}
Let $\xi \in M$ be fixed. In this paper we will need precise expansions of $P_{\exp_\xi^*g} - \Delta_0^k$. Since the leading order term in the expansion of $P_g$ is $\Delta_g^k$, a first rough estimate, using the expansion of $g$ in its exponential coordinates, is as follows: for any smooth function $u$ in $\R^n$ and any $x \in \R^n$,
 \begin{equation}  \label{expansion.Pg}
\begin{aligned}
P_{\exp_\xi^*g} u (x) & = \Delta_0^k u(x) +\bigO \big (|x|^2 |\nabla^{2k} u(x)|_{g} \big ) \\
& + \bigO \big( |x| |\nabla^{2k-1} u(x)|_{g} \big) + \bigO\Big(\sum_{\ell=0}^{2k-2}  |\nabla^{\ell} u(x)|_{g}\Big),
\end{aligned}
\end{equation}
where the constants in the $\bigO(\cdot)$ terms are independent of $\xi, x$. A similar expansion was also obtained in Robert \cite{RobPoly1}. In the case where $g$ is the conformal metric $g_\xi$ of \eqref{conf.1}, \eqref{conf.2}, \eqref{conf.3} and $u$ is a radial function, we prove much more precise expansions in Appendix \ref{sec.GJMS.expan}: we refer to Proposition \ref{GJMS.exp1} below. 
\smallskip

\noindent
Throughout this paper, $C$ will denote a generic positive constant that depends on $n,k$ and possibly $(M,g)$.  $|a| \lesssim b$ will equivalently denote $a= O(b)$.
\smallskip

\section{The GJMS operator in conformal normal coordinates}\label{sec.GJMS.expan}

For $\xi \in M, \mu >0$ we define
\begin{align}\label{bubble3}  
U_{\xi,\mu}\(x\)&:=\frac{\mu^{\frac{n-2k}{2}}}{(\mu^{2}+\mathfrak{c}^{-1}_{n,k}~d_{g_{\xi}}(x,\xi)^{2})^{\frac{n-2k}{2}}}\quad\hbox{ for } x\in M.
\end{align}
This is a rescaling of the Euclidean bubble \eqref{bubble1} centered at $\xi\in M$. 
Here $g_\xi$ is a conformal metric to $g$ whose exponential map defines conformal normal coordinates as in \eqref{conf.2}, \eqref{conf.3}. If $(M,g)$ is Euclidean then $P_{g}=\Delta_{0}^k$ and we have $\Delta_{0}^k\,U_{\xi,\mu} =U_{\xi,\mu}^{2^{*}_{k}-1}$. In this section we expand $P_{g_\xi} U_{\xi,\mu} - U_{\xi,\mu}^{2^{*}_{k}-1}$ to fourth-order in conformal normal coordinates at $\xi$ for a general metric $g$. This will turn out to be crucial in the next Sections, in particular in Section \ref{sec.order.two}, to precisely describe the concentration of blowing-up sequences of solutions of \eqref{eq:zero}. This section is also the main technical improvement of this paper. Throughout this section, $\xi \in M$ will be fixed and we will let $g_\xi$ be defined as in \eqref{conf.2}. We let 
\begin{align} \label{mathfrakg}
\mathfrak{g} = \big( \exp_{\xi}^{g_\xi} \big)^*g_{\xi} .
\end{align}
Let, for any $x \in \R^n$,
\begin{align}\label{bubble3.1}  
\tilde{U}_\mu(x) = \frac{\mu^{\frac{n-2k}{2}}}{\left(\mu^{2}+\mathfrak{c}^{-1}_{n,k}~|x|^2\right)^{\frac{n-2k}{2}}} = \mu^{- \frac{n-2k}{2}} U \Big( \frac{x}{\mu} \Big)
 \end{align}
  where $U$ is given by \eqref{bubble1}.
By \eqref{bubble3} we have $U_{\xi, \mu} \big(\exp_{\xi}^{g_\xi}(x) \big) = \tilde{U}_\mu(x)$, so that  
$$ \big( P_{g_\xi}U_{\xi, \mu} - U_{\xi, \mu}^{2^*-1} \big) \big( \exp_{\xi}^{g_\xi}(x)\big) = \big( P_{\mathfrak{g}} - \Delta_0^k \big)\tilde{U}_\mu(x). $$
Since $\tilde{U}_\mu$ is radial we need to obtain an expansion of $P_{\mathfrak{g}} u$ for smooth radial functions $u: \R^n \to \R$. 
In this section all norms are measured with respect to $\mathfrak{g}$ and repeated indices appearing in an expression are summed over. If $T$ is a tensor field of rank $\ell \ge 1$, both the notations $ \nabla_aT_{i_1 \dots i_\ell}$ or $T_{i_1 \dots i_\ell;a}$ denote the coordinates of $\nabla T$ where $\nabla$ is the covariant derivative for $g_\xi$, so that $T_{i_1 \dots i_\ell;ab} = \nabla_b \nabla_aT_{i_1 \dots i_\ell}$.

\begin{proposition}\label{GJMS.exp0}
Fix $N$ large enough as in \eqref{conf.2}. Let $u(x) = u(|x|)$ be a smooth radial function in $\R^n$ and let $r  = |x|$. We have, for $x \in \R^n$,
\begin{equation} \label{GJMS.exp0.1}
\begin{aligned}
& \big(P_{\mathfrak{g}}\,u - \Delta_{0}^{k}\,u \big)(x) \\
& = A_{1,u}(r) |\Weyl_{g_{\xi}}(\xi)|_{g_\xi}^2 + A_{2,u}(r)  (\Scal_{g_\xi})_{;ab}(\xi) \frac{x^a x^b}{r^2} +A_{3,u}(r) \Delta_{g_\xi} (\Ricci_{g_\xi})_{ab}(\xi)\frac{x^a x^b}{r^2}  \\
&~+ A_{4,u}(r)  (\Weyl_{g_\xi})_{pabq}(\xi) (\Weyl_{g_{\xi}})_{pcdq}(\xi) \frac{x^a x^b x^c x^d}{r^4}\\
& ~+ \Psi_u(x) 
+ \bigO \big( \sum_{j=0}^{3} r^{4-j}|\nabla^{2k-2-j} u(r)|\big) + 
\bigO \big( \sum_{j=0}^{2k-6}|\nabla^{j} u(r)|\big) \\
& + \bigO \big( r^N |\nabla^{2k-1} u(x)|\big).
\end{aligned} 
\end{equation}
In \eqref{GJMS.exp0.1} the functions $A_{i,u}$, $1 \le i \le 4$, are explicit and given by
\begin{equation*}
\begin{aligned}
A_{1,u}(r) & = C_{1,1} \Delta_0^{k-2}u(r) + C_{1,2} \partial_r^2 \Delta_0^{k-3} u(r) + C_{1,3} \frac{1}{r} \partial_r \Delta_0^{k-3} u(r) \\
&+ C_{1,4}\frac{1}{r} \partial_r^3 \Delta_0^{k-4} u(r) + C_{1,5}\frac{1}{r^{2}} \partial_r^2 \Delta_0^{k-4} u(r) + C_{1,6}\frac{1}{r^{3}} \partial_r \Delta_0^{k-4} u(r), \\
\end{aligned} 
\end{equation*}
\begin{equation*}
\begin{aligned}
A_{2,u}(r) & = C_{2,1} r^2 \Delta_0^{k-1} u(r) + C_{2,2}  r \partial_r \Delta_0^{k-2} u(r) + C_{2,3} \frac{1}{r} \partial_r \Delta_0^{k-3}u (r) \\
& + C_{2,4}\partial_r^2 \Delta_0^{k-3}u(r) + C_{2,5} \partial_r^3 \Delta_0^{k-3}u(r) + C_{2,6} \frac{1}{r} \partial_r^3 \Delta_0^{k-4} u(r) \\
&+ C_{2,7}  \frac{1}{r^2} \partial_r^2 \Delta_0^{k-4} u(r) + C_{2,8} \frac{1}{r^3} \partial_r \Delta_0^{k-4} u(r)+ C_{2,9}  \partial_r^4 \Delta_0^{k-4}u(r),
\end{aligned} 
\end{equation*}
\begin{equation*}
\begin{aligned}
A_{3,u}(r) & = C _{3,1}  \frac{1}{r} \partial_r^3 \Delta_0^{k-4} u(r) + C_{3,2}\frac{1}{r^{2}}\partial_r^2 \Delta_0^{k-4} u(r) + C_{3,3} \frac{1}{r^3} \partial_r \Delta_0^{k-4} u(r) \\
& + C_{3,4} \partial_r^2 \Delta_0^{k-3}u(r) + C_{3,5} \frac{1}{r} \partial_r \Delta_0^{k-3}u (r), \\
\end{aligned} 
\end{equation*}
and 
\begin{equation*}
\begin{aligned}
A_{4,u}(r) & = C_{4,1} r^2 \partial_r^2 \Delta_0^{k-2} u(r)+C_{4,2} r \partial_r \Delta_0^{k-2} u(r) + C_{4,3} r \partial_r^3 \Delta_0^{k-3} u(r)\\
& + C_{4,4} \partial_r^2 \Delta_0^{k-3} u(r) + C_{4,5} \partial_r^4 \Delta_0^{k-4} u(r) + C_{4,6} \frac{1}{r} \partial_r^3 \Delta_0^{k-4} u(r) \\
&+ C_{4,7}\frac{1}{r^2} \partial_r^2 \Delta_0^{k-4} u(r)+ C_{4,8} \frac{1}{r^3} \partial_r \Delta_0^{k-4} u(r), \\
\end{aligned} 
\end{equation*}
where $C_{i,j}$ denotes a numerical constant that only depends on $n$ and $k$. Also, in \eqref{GJMS.exp0.1}, the function $\Psi_u(x)$ has the following form:
\begin{equation} \label{defPsiu}
\begin{aligned}
\Psi_u(x) & = \psi_{2k-2}^{(3)}(x) \Delta_0^{k-1}u(r) + \psi_{2k-2}^{(5)} \frac{1}{r^2} \partial_r^2 \Delta_0^{k-2}u(r) \\
& + \Big(\psi_{2k-3}^{(3)}(x) + \frac{ \psi_{2k-3}^{(5)}(x)}{r^2} \Big) \frac{1}{r} \partial_r \Delta_0^{k-2}u(r)  +  \psi_{2k-4}^{(1)}(x) \Delta_0^{k-2}u(r) \\
& +  \psi_{2k-4}^{(3)}(x)\frac{1}{r^2} \partial_r^2 \Delta_0^{k-3}u(r) +  \psi_{2k-4}^{(5)}(x)\frac{1}{r^4} \partial_r^4 \Delta_0^{k-4}u(r) \\
& + \Big( \psi_{2k-5}^{(1)}(x) + \frac{  \psi_{2k-5}^{(3)}(x)}{r^2} \Big)\frac{1}{r} \partial_r \Delta_0^{k-3}u(r) \\
& +\Big(  \psi_{2k-5}^{(3)}(x) + \frac{\tilde{\psi}_{2k-5}^{(5)}(x)}{r^2} \Big)\frac{1}{r^3} \partial_r^3 \Delta_0^{k-4}u(r) \\
& + \Big( \psi_{2k-6}^{(1)}(x) + \frac{ \psi_{2k-6}^{(3)}(x)}{r^2} \Big)\frac{1}{r^2} \partial_r^2 \Delta_0^{k-4}u(r)\\ 
& + \Big( \psi_{2k-7}^{(1)}(x) + \frac{ \psi_{2k-7}^{(3)}(x)}{r^2} + \frac{ \psi_{2k-7}^{(5)}(x)}{r^4} \Big)\frac{1}{r^3} \partial_r \Delta_0^{k-4}u(r)\\ 
&+ \sum_{p=0}^{2k-6} \sum_{q=0}^{\big[\frac{2k-5}{2}\big]} \frac{\psi_q^{(2k-5-2q)}(x)}{r^{2k-5-2q + p}} \partial_r^{2k-5-p} u(r), \\
\end{aligned} 
\end{equation}
where $\psi_\ell^{(i)}, \tilde{\psi}_\ell^{(i)}$ are homogeneous polynomials of degree $i \ge 1$ in $\R^n$ whose coefficients only depend on the geometry of $g_\xi$. 
\end{proposition}
\noindent Similar expansions had already been obtained in \cite{Marques} when $k=1$, in \cite{LiXiong, GongKimWei} when $k=2$ and in \cite{ChenHou} when $k=3$, and they used the explicit expression of $P_g$. Here we obtain Proposition \ref{GJMS.exp0} simply by relying on Juhl's formulae, which allow us to overcome the lack of an explicit expression for $P_g$ when $k \ge 4$. Remark that when $k \ge 3$ an additional term $ \Delta_{g_\xi} (\text{Ric}_{g_\xi})_{ab}(\xi)\frac{x^a x^b}{r^2}$ appears in \eqref{GJMS.exp0.1}.

\begin{proof}
Throughout this proof, for simplicity, we will simply denote $g_{\xi}$ by $g$ and we will omit the dependence in $g_\xi$ on the curvature tensors and on the norms. For instance the Ricci and Weyl tensors of $g_\xi$ will symply be denoted by $\Ricci$ and $\Weyl$. In the following we will denote the Riemann tensor by $\Riemann$ and we use the convention for $\Riemann$ as in the paper of Lee--Parker \cite{LeeParker}. Similarly, the covariant derivative with respect to $g_\xi$ will be simply be denoted $\nabla$, and $\Delta = \nabla^*\nabla$ will denote the Bochner Laplacian for $g_\xi$ acting on tensors. In coordinates, if $T$ is of rank $\ell$, we have $\Delta T_{i_1 \dots i_\ell} = - \tensor{T}{_{i_1 \dots i_\ell; a}^{a}}$. 
All the coordinates in the following will be taken in exponential coordinates $\exp_{\xi}^{g_\xi}$ for $g_\xi$ at $\xi$, so that \eqref{conf.1}, \eqref{conf.2} and \eqref{conf.3} hold. We will let $\Schouten$ be the Schouten tensor of $g_\xi$ defined by
$$\Schouten:=\frac{1}{n-2}\(\Ricci-\frac{\Scal}{2\(n-1\)}\,g\)$$
and $\Bach$ be the Bach tensor whose coordinates are given by
$$\Bach_{ij}:=\Schouten_{ab}\tensor{\Weyl}{_i^a_j^b}+\tensor{\Schouten}{_{ij;a}^a}-\tensor{\Schouten}{_{ia;j}^a}.$$
We let $\(\cdot,\cdot\)$ be the inner product induced by $g_\xi$ on tensors of same rank: that is, $\(S,T\)=S^{i_1\dotsc i_\ell}T_{i_1\dotsc i_\ell}$ for all tensors $S$ and $T$ of rank $\ell\in\N$. Throughout the proof we will use the following notation: if $\ell$ is any integer,  $Z^{(\ell)}$ will denote a smooth linear operator of order less than or equal to $\ell$, possibly tensor-valued, that may change from one line to the other, and which is zero if $\ell \le 0$. 

\medskip
Let $u$ be a smooth function in $\R^n$ (not necessarily radial). The expansions in \cite[Step $2.1$]{MazumdarVetois}, which rely on Juhl's formulae \cite{JuhlGJMS}, show that 
\begin{align}\label{Pr1Step1Eq1:bis}
P_{\mathfrak{g}}u  &=\Delta^k u +k\Delta^{k-1}\(J_1 u\)+k\(k-1\)\Delta^{k-2}\(J_2 u +\(T_1,\nabla u \)+\(T_2,\nabla^2 u \)\)\nonumber\\
&\quad+k\(k-1\)\(k-2\)\Delta^{k-3}\(\(T_3,\nabla^2 u\)+\(\nabla T_2,\nabla^3 u\)\)\allowdisplaybreaks \nonumber\\
&\quad+k\(k-1\)\(k-2\)\(k-3\)\Delta^{k-4}\(T_4,\nabla^4 u\) \nonumber\\
& \quad +Z^{(2k-5)}u + \bigO \big( \sum_{j=0}^{2k-6}|\nabla^{j}u|\big) ,
\end{align}
where we have let 
\begin{align*}
& J_1:=\frac{n-2}{4\(n-1\)}\Scal,\allowdisplaybreaks\\
& J_2:=\frac{1}{6}\(\frac{3n^2-12n-4k+8}{16\(n-1\)^2}\Scal^2-\(k+1\)\(n-4\)\left|\Schouten\right|^2-\frac{3n+2k-4}{4\(n-1\)}\Delta\Scal\), \\
&T_1:=\frac{n-2}{4\(n-1\)}\nabla\Scal-\frac{2}{3}\(k+1\)\delta \Schouten,\allowdisplaybreaks\\
&T_2:=\frac{2}{3}\(k+1\)\Schouten,\allowdisplaybreaks\\
&T_3:=\frac{n-2}{6\(n-1\)}\nabla^2\Scal+\frac{\(k+1\)\(n-2\)}{6\(n-1\)}\Scal\Schouten-\frac{k+1}{3}\(\nabla^* \nabla\Schouten+2\nabla\delta\Schouten+2\Riemann\ast\Schouten\)\\
&\qquad-\frac{2}{15}\(k+1\)\(k+2\)\(3\Schouten^{\#}\Schouten+\frac{\Bach}{n-4}\),\allowdisplaybreaks\\
\end{align*}
and
$$T_4:=\frac{2}{5}\(k+1\)\(\frac{5k+7}{9}\Schouten\otimes \Schouten+\nabla^2\Schouten\),$$
where $\#$ stands for the musical isomorphism with respect to $g$ (i.e. $\Schouten^\#:=g^{-1}\Schouten$), and $\nabla\delta\Schouten$ and $\Riemann\ast\Schouten$ stand for the covariant tensors whose coordinates are given in the exponential chart of $g_\xi$ at $\xi$ by 
\begin{equation}\label{Pr1Step1Eq2:bis}
\begin{aligned} &
\(\nabla\delta\Schouten\)_{ij}:=-\tensor{\Schouten}{_i^a_{;aj}}\text{ and }  \(\Riemann\ast\Schouten\)_{ij}:=\tensor{\Riemann}{_{ia}^a_b}\tensor{\Schouten}{_j^b}+\Riemann_{ibja}\Schouten^{ab},
\end{aligned}
\end{equation}
where $\Riemann$ is the Riemann tensor of $g_\xi$. In particular if $T$ is a tensor of rank $\ell \ge 1$ we have 
\begin{equation} \label{commute:der}
 \tensor{T}{_{i_1 \dots i_\ell; pq}} -  \tensor{T}{_{i_1 \dots i_\ell; qp}} = \sum_{s=1}^\ell \tensor{\Riemann}{^{r}_{i_spq}}T_{i_1\dots r \dots i_s}. 
\end{equation}
We obtain \eqref{GJMS.exp0.1} by precisely expanding, when $u$ is radial, the right-hand side of \eqref{Pr1Step1Eq1:bis} to fourth order. The proof relies on rather long computations, and we only sketch them in the following. We first observe that \eqref{commute:der} implies that  
\begin{equation} \label{GJMS.exp0.proof.2}
 \begin{aligned} 
 \Delta \nabla v &= \nabla \Delta v + \Ricci(\cdot, \nabla v), \\
\Delta \nabla^2 v & =  \nabla^2 \Delta v + \Riemann\ast \nabla^2 v + Z^{(1)}v  \quad \text{ and } \\
\Delta \nabla^j v & =  \nabla^j \Delta v + Z^{(j)} v \quad \text{ for } j \ge 3, 
\end{aligned}  
\end{equation}
for any smooth function $v$, where we have let 
$$ (\Riemann\ast \nabla^2 v)_{ij} = \tensor{\Ricci}{_{jp}} \tensor{v}{_{;i}^{p}} +  \tensor{\Ricci}{_{i p}} \tensor{v}{_{;j}^{p}} - 2 \tensor{\Riemann}{^{k}_{i}^{p}_{j}}v_{;kp}. $$
Independently, if $S,T$ are tensors of rank $\ell$ we have 
$$ \Delta (T,W) = (\Delta T, W) + (T, \Delta W) - 2 (\nabla T, \nabla W). $$
Using \eqref{GJMS.exp0.proof.2} we thus have 
$$\Delta \big( T, \Riemann\ast \nabla^2 v \big) = \big(T, \Riemann\ast \nabla^2 \Delta v \big) + Z^{(3)}v. $$
Applying the latter two relations recursively with \eqref{GJMS.exp0.proof.2} shows that for $j \ge 1$
\begin{equation}
\label{GJMS.exp0.proof.3}
 \begin{aligned} 
 \Delta^j \big(T, \nabla u) & =  \big( T, \nabla \Delta^j u \big) -2j \big( \nabla T, \nabla^2 \Delta^{j-1}u \big) + Z^{(2j-1)}u, \\
 \Delta^j \big(T, \nabla^2 u) & =  \big( T, \nabla^2 \Delta^j u \big) -2j \big( \nabla T, \nabla^3 \Delta^{j-1}u \big) \\
 & + 2j(j-1)\big( \nabla^2 T, \nabla^4 \Delta^{j-2}u \big) + j \big(\Delta T, \nabla^2 \Delta^{j-1}u \big) \\
 & + j \big( T, \Riemann \ast \nabla^2 \Delta^{j-1}u\big) + Z^{(2j-1)}u, \\
  \Delta^j \big(T, \nabla^3 u) & =  \big( T, \nabla^3 \Delta^j u \big) -2j \big( \nabla T, \nabla^4 \Delta^{j-1}u \big) + Z^{(2j+1)}u, 
 \end{aligned}  
\end{equation}
where in the previous relations $T$ denotes a tensor of rank $1,2$ or $3$. Let now $v,w$ be smooth functions and $j \ge 1$ be an integer. We claim that
\begin{equation}  \label{GJMS.exp0.proof.1}
\begin{aligned}
\Delta^{j}(vw) & = v \Delta^j w - 2j \big( \nabla v, \nabla \Delta^{j-1} w\big) + j \Delta v \Delta^{j-1}w\\
&  + 2j(j-1) \big( \nabla^2 v, \nabla^2 \Delta^{j-2}w \big) + Z^{(2j-3)} w
\end{aligned}
\end{equation}
holds. If $j=1$, \eqref{GJMS.exp0.proof.1} is simply Leibniz's formula. If $j \ge 1$ we write that 
$$ \Delta^j(vw) = \Delta^{j-1}\big(v \Delta w + w\Delta v - 2 \big( \nabla v, \nabla w \big) \big) $$
and we recursively apply \eqref{GJMS.exp0.proof.3} to conclude. We now apply \eqref{GJMS.exp0.proof.2}, \eqref{GJMS.exp0.proof.3} and \eqref{GJMS.exp0.proof.1} to each term in \eqref{Pr1Step1Eq1:bis}. Combining the latter expansions into \eqref{Pr1Step1Eq1:bis} yields 
\begin{equation} \label{GJMS.exp0.proof.4}
\begin{aligned}
&P_{\mathfrak{g}}u -\Delta^k u \\
&  = k \Big \{ J_1 \Delta^{k-1}u - 2(k-1) \big(\nabla J_1, \nabla \Delta^{k-2}u \big) \\
& + (k-1) \Delta J_1 \Delta^{k-2}u + 2(k-1)(k-2) \big( \nabla^2 J_1, \nabla^2 \Delta^{k-3}u \big)   \Big\} \\
& + k(k-1) J_2 \Delta^{k-2}u + k(k-1) \Big\{\big( T_1, \nabla \Delta^{k-2}u \big) - 2(k-2) \big( \nabla T_1, \nabla^2 \Delta^{k-3}u \big) \Big\} \\
& + k(k-1) \Big\{ \big( T_2, \nabla^2 \Delta^{k-2}u \big) - 2(k-2) \big( \nabla T_2, \nabla^3 \Delta^{k-3}u \big)  \\
    &  +2(k-2)(k-3) \big( \nabla^2 T_2, \nabla^4 \Delta^{k-4}u \big) + (k-2) \big( \Delta T_2, \nabla^2 \Delta^{k-3} \big) \\ 
& + (k-2)\big( T_2, \Riemann\ast \nabla^2 \Delta^{k-3} u  \big) \Big\} + k(k-1)(k-2)\big( T_3, \nabla^2 \Delta^{k-3}u \big) \\
& + k(k-1)(k-2)  \Big\{\big( \nabla T_2, \nabla^3 \Delta^{k-3}u \big) - 2(k-3) \big( \nabla^2 T_2, \nabla^4 \Delta^{k-4}u \big)  \Big\} \\
& k(k-1)(k-2)(k-3)  \big( T_4, \nabla^4 \Delta^{k-4}u \big) + Z^{(2k-5)}u + \bigO \big( \sum_{j=0}^{2k-6}|\nabla^{j}u|\big).
\end{aligned}
\end{equation}
We now assume, until the end of this proof, that $u$ is radial. We prove \eqref{GJMS.exp0.1} by expanding each term in \eqref{GJMS.exp0.proof.4} to the fourth order in $x$. First, since we are working with conformal normal coordinates, any power of the laplacian of $u$ for $g_\xi$ can be replaced, up to a negligible error term, by the  same power of the Euclidean laplacian of $u$. Indeed, by \eqref{conf.2} and since $u$ is radial we have, for every $j \ge 1$, 
\begin{equation} \label{DL:radial}
\begin{aligned}
 \Delta^j  u & = \Delta_0^j u + \bigO \big( r^{N'} \sum_{p=0}^{2j-1} |\nabla^{p} u|\big)  
\end{aligned}
 \end{equation}
 for some large integer $N'$. Without loss of generality we may thus replace every term $\Delta^j u$ in \eqref{GJMS.exp0.proof.4} by $\Delta_0^j u$, which is a radial function. If now $v$ is any smooth radial function we have
 \begin{equation} \label{GJMS.exp0.proof.5}
\begin{aligned}
v_{;ab}(x) & = v_{,ab}(x) - \partial_d \Gamma_{ab}^c(\xi)\frac{x^c x^d}{r} \partial_r v(r)  + \bigO(r^2 |\partial_r v(r)|),  \\
v_{;abc}(x) & = v_{,abc}(x) - \Gamma_{bc}^d(x) v_{, da}(x) - \Gamma_{ac}^d(x) v_{, de}(x) \\
& - \partial_a \Gamma_{bc}^d(\xi) \frac{x^d}{r} \partial_r v(r) +  \bigO (r^2|\nabla^2 v(x)| + r |\nabla v(x)| ),\\
v_{;abcd}(x) & = v_{,abcd}(x)  - \Gamma_{ad}^e (x) v_{, ebc}(x) - \Gamma_{bd}^e(x) v_{,a ec}(x) \\
& -  \Gamma_{cd}^e(x) v_{,abe}(x)  +  \bigO ( r^2 |\nabla^3 v(x)| +  |\nabla^2 v(x)| + |\nabla v(x)|) , \\
\end{aligned}
\end{equation}
 where we denoted by $v_{,a}$ the covariant derivatives with respect to the Euclidean metric. The proof of \eqref{GJMS.exp0.1} now follows from an asymptotic expansion at $x=0$ of all the quantities involved in \eqref{GJMS.exp0.proof.4}: the Christoffel symbols at $\xi$ and the tensors in  \eqref{Pr1Step1Eq1:bis} are expanded at $\xi$ using \eqref{conf.3}, covariant derivatives of $\Delta_0^j $ are expanded by applying \eqref{GJMS.exp0.proof.5} to $v = \Delta_0^j u$ and explicitly computing $v_{,ab}, v_{,abc}$ and $v_{,abcd}$. First derivatives of $\Ricci$ at $\xi$ only appear symmetrised and thus vanish by \eqref{conf.3}, and second derivatives of $\Ricci$ at $\xi$ are traced and yield quadratic terms in the $\Weyl_{g_\xi}$ tensor by \eqref{conf.3}. We omit the details.
\end{proof}

As a consequence of Proposition~\ref{GJMS.exp0} we obtain an expansion of $P_{g_\xi}U_{\xi, \mu} - U_{\xi, \mu}^{2^*-1}$ to fourth-order in conformal normal coordinates at $\xi$: 
\begin{proposition}\label{GJMS.exp1}
Let $\mu >0$, $\tilde{U}_\mu$ be given by \eqref{bubble3.1} and let $\mathfrak{g}$ be given by \eqref{mathfrakg}. We have, for $x \in \R^n$,  
\begin{equation} \label{GJMS.exp1.eq} 
\begin{aligned}
&\mu^{- \frac{n-2k}{2}} \big( P_{\mathfrak{g}} - \Delta_0^k \big)\tilde{U}_\mu(x)\\
&  = \mu^{4-n} F_{1}\Big(\frac{r}{\mu}\Big) |\Weyl_{g}(\xi)|_{g}^2 +\mu^{4-n} F_{2}\Big(\frac{r}{\mu}\Big)  (S_{g_\xi})_{;ab}(\xi) \frac{x^a x^b}{r^2} \\
& ~+\mu^{4-n} F_{3}\Big(\frac{r}{\mu}\Big)  \Delta_{g_\xi} (\Ricci_{g_\xi})_{ab}(\xi)\frac{x^a x^b}{r^2} \\
&~+ \mu^{4-n} F_{4}\Big(\frac{r}{\mu}\Big)  (\Weyl_{g_\xi})_{pabq}(\xi) (\Weyl_{g_{\xi}})_{pcdq}(\xi) \frac{x^a x^b x^c x^d}{r^4} \\
&~+ \mu^{5-n} F_5\Big(\frac{x}{\mu}\Big)  +  \bigO \Big((\mu^2 + \mathfrak{c}^{-1}_{n,k}\,|x|^{2})^{-\frac{n-6}{2}}\Big).
\end{aligned}
\end{equation}
In \eqref{GJMS.exp1.eq} the $F_i$, for $1 \le i \le 4$, are smooth radial functions in $\R$ that satisfy 
$$F_i(r) = c_i r^{4-n} + O(r^{5-n}) \quad \text{ as } r \to + \infty$$
 for some $c_i \in \R$, and $F_5$ can be written as $F_5(x) = \Psi(x) R(r)$, where $\Psi$ is a homogeneous polynomial of degree $5$ and $R$ is a smooth function in $\R^n$ which satisfies
$$ |R(x)| \lesssim (1+|x|)^{-n} \quad \text{ for all } x \in \R^n.$$
Also, in \eqref{GJMS.exp1.eq}, the constants in the $\bigO(\cdot)$ term are uniform in $\mu$ and $\xi$. 
 \end{proposition}

\begin{proof}
This is an application of Proposition \ref{GJMS.exp0} with $u(x) =\widetilde{U}_\mu(x)$. Using the expression of $\widetilde{U}_\mu$ given by \eqref{bubble3.1} it is easily seen that the $\bigO( \cdot)$ terms in \eqref{GJMS.exp0.1} can be estimated as $\bigO \Big(\mu^{\frac{n-2k}{2}}(\mu^2 + \mathfrak{c}^{-1}_{n,k}\,|x|^{2})^{-\frac{n-6}{2}}\Big)$. Since $\widetilde{U}_\mu(x) = \mu^{- \frac{n-2k}{2}} U \Big( \frac{x}{\mu} \Big)$, we have 
$$A_{i,\widetilde{U}_\mu}(r) = \mu^{-\frac{n-2k}{2} + 4-2k} A_{i,U}(\frac{r}{\mu}) = \mu^{\frac{n-2k}{2} + 4-n} A_{i,U}(\frac{r}{\mu}),$$
 and we may thus let $F_i(r) = A_{i,U}(r)$ for $1 \le i \le 4$. The behavior at infinity of $F_i$ follows from the explicit expression of $U$ and of $A_{i,U}$. Finally, we again have $\Psi_{\widetilde{U}_\mu}(x) = \mu^{\frac{n-2k}{2} + 5-n} \Psi_U(\frac{x}{\mu})$. The expression of $F_5$ then follows from the expression of $\Psi_U$ in \eqref{defPsiu}. 
\end{proof}

We now use \eqref{GJMS.exp1.eq} to obtain an expansion to fourth-order of the so-called Pohozaev's quadratic form associated to the bubble $\tilde{U}_\mu$ defined in \eqref{bubble3.1}:

\begin{proposition} \label{GJMS.exp2}
Assume that $2k+1 \le n \le 2k+5$. Let $\mu >0$, $\tilde{U}_\mu$ be given by \eqref{bubble3.1} and let $\mathfrak{g}$ be given by \eqref{mathfrakg}. Let $\iota_{0}:=\inf\limits_{\xi\in M}\iota(M,g_{\xi})$ and  $0 < \delta <\frac{\iota_0}{2}$. There exist $C = C(n,k) >0$ such that 
\begin{equation} \label{pohozaev.quadratic.form}
\begin{aligned}
 & \int_{B(0,\delta)}\(\frac{n-2k}{2}\widetilde{U}_{\mu}+x^{i}\partial_{i}\widetilde{U}_{\mu}\)(P_{\mathfrak{g}}\widetilde{U}_{\mu}-\Delta_{0}^{k}\,\widetilde{U}_{\mu})\,dx \\
 & =  C |\Weyl_{g}(\xi)|_g^{2}\times\left\{\begin{aligned}
  &\mu^{4}\ln(1/\mu) + \bigO(\mu^4) &&\text{if }n=2k+4\\&\mu^{4} + \bigO(\mu^5) &&\text{if }n=2k+5\end{aligned}\right \} + \bigO(\delta \mu^{n-2k}), \\
\end{aligned} 
\end{equation}
where the constants in the $\bigO(\cdot)$ term are independent of $\mu, \xi, \delta$. 
\end{proposition}
It is implicit in the statement of Proposition \ref{GJMS.exp2} that the term inside the brackets vanishes when $2k+1 \le n \le 2k+3$.

\begin{proof}
We first assume that $2k+1 \le n \le 2k+3$. Proposition~\ref{GJMS.exp1} shows that 
$$\begin{aligned} (P_{\mathfrak{g}}\widetilde{U}_{\mu}-\Delta_{0}^{k}\,\widetilde{U}_{\mu}\big)(x)  = \bigO\(\frac{\mu^{\frac{n-2k}{2}}}{\(\mu+|x|\)^{n-4}}\),
\end{aligned} $$
where the constant in the $\bigO(\cdot)$ term is independent of $\mu, \xi$ and $\delta$. As consequence straightforward computations shows that
\begin{equation*}
\begin{aligned}
 \int_{B(0,\delta)}\(\frac{n-2k}{2}\widetilde{U}_{\mu}+x^{i}\partial_{i}\widetilde{U}_{\mu}\)(P_{\mathfrak{g}}\widetilde{U}_{\mu}-\Delta_{0}^{k}\,\widetilde{U}_{\mu})\,dx =  \bigO( \delta \mu^{n-2k} ).
\end{aligned} 
\end{equation*}
We now assume that $2k+4 \le n \le 2k+5$. We again use Proposition~\ref{GJMS.exp1} to estimate $P_{\mathfrak{g}}\widetilde{U}_{\mu}-\Delta_{0}^{k}\,\widetilde{U}_{\mu}$. First, since $n \le 2k+5$ and $|\frac{n-2k}{2}\widetilde{U}_{\mu}+x^{i}\partial_{i}\widetilde{U}_{\mu}| \lesssim \widetilde{U}_\mu$ in $\R^n$, straightforward computations show that 
$$  \int_{B(0,\delta)}\left| \frac{n-2k}{2}\widetilde{U}_{\mu}+x^{i}\partial_{i}\widetilde{U}_{\mu}\right| \frac{\mu^{\frac{n-2k}{2}}}{(\mu^2 + \mathfrak{c}^{-1}_{n,k}\,|x|^{2})^{\frac{n-6}{2}}} \, dx = \bigO(\delta \mu^{n-2k}). $$ 
We now observe that $\frac{n-2k}{2}\widetilde{U}_{\mu}+x^{i}\partial_{i}\widetilde{U}_{\mu}$ is a radial function since $\widetilde{U}_{\mu}$ is itself radial. As a consequence, and since $F_5$ is odd in $x$ since $\Psi$ is of odd degree, we have 
$$ \int_{B(0,\delta)}\(\frac{n-2k}{2}\widetilde{U}_{\mu}+x^{i}\partial_{i}\widetilde{U}_{\mu}\)\mu^{5-n} F_5\Big(\frac{x}{\mu}\Big) \, dx = 0. $$
We now integrate the remaining  terms in \eqref{GJMS.exp1.eq}. These terms are the product of a radial function and of a homogeneous polynomial of order $2$ or $4$. A simple antisymmetry argument shows that they are traced after integration: we thus obtain, for instance, 
$$ \begin{aligned} 
& \int_{B(0,\delta)}\(\frac{n-2k}{2}\widetilde{U}_{\mu}+x^{i}\partial_{i}\widetilde{U}_{\mu}\) F_{2}\Big(\frac{r}{\mu}\Big)  (S_{g_\xi})_{;ab}(\xi) \frac{x^a x^b}{r^2} \\
& = \frac{1}{n}  \Delta_{g_{\xi}} \Scal_{g_{\xi}}\(\xi\)  \times\left\{\begin{aligned}
 &c_2 \omega_3 \mu^{4}\ln(1/\mu) + \bigO(\mu^4) &&\text{if }n=2k+4\\&I \mu^{4} + \bigO(\mu^5) &&\text{if }n=2k+5\end{aligned}\right \}, 
\end{aligned} $$ 
where $F_2(r) \sim c_2 r^{4-n}$ as $r \to + \infty$ when $n  = 2k+4$ and where we have let $I = \int_{\R^n}\(\frac{n-2k}{2}U +x^{i} \partial_{i} U \)F_2(r)\, dx$ when $n = 2k+5$. The other terms involving $F_3$ and $F_4$ are computed in the same way. Since by  \eqref{conf.3} we have $ \Delta_{g_{\xi}} \Scal_{g_{\xi}}\(\xi\)=\frac{1}{6}|\Weyl_{g}\(\xi\)|_{g}^2$ and since $\Weyl_{g_\xi}$ is totally traceless we obtain in the end that there exists $C\in \R$ such that 
\begin{equation} \label{GJMS.exp2.2}
\begin{aligned}
 & \int_{B(0,\delta)}\(\frac{n-2k}{2}\widetilde{U}_{\mu}+x^{i}\partial_{i}\widetilde{U}_{\mu}\)(P_{\mathfrak{g}}\widetilde{U}_{\mu}-\Delta_{0}^{k}\,\widetilde{U}_{\mu})\,dx =  \bigO(\delta \mu^{n-2k})  \\
 &  + C |\Weyl_{g}(\xi)|_g^{2}\times\left\{\begin{aligned}
 &\mu^{4}\ln(1/\mu) + \bigO(\mu^4) &&\text{if }n=2k+4\\&\mu^{4} + \bigO(\mu^5) &&\text{if }n=2k+5\end{aligned}\right \} , \\
\end{aligned} 
\end{equation}
where as before the constants in the $\bigO(\cdot)$ terms are independent of $\mu, \xi$ and $\delta$. We now prove that $C$ in \eqref{GJMS.exp2.2} is positive. 
We define 
$$\begin{aligned}
 W_{\xi, \mu}\(x\)&  = \chi\(d_{g_\xi}\(\xi, x\)\) U_{\xi, \mu}\(x\) \quad &\text{ for } x \in M, \\
 \widetilde{W}_{\mu}\(x\)&  = \chi\(|x|\) \widetilde{U}_\mu(x)  \quad &\text{ for } x \in \R^n, 
  \end{aligned}$$ 
where $U_{\xi, \mu}\(x\)$ is as in \eqref{bubble3},  and  where  $\chi:[0,+\infty)\to[0,1]$ is a smooth cutoff function such that $\chi\equiv1$ in $\[0,\delta\]$ and $\chi\equiv0$ in $\[2 \delta ,+\infty\)$. We have $W_{\xi, \mu} \big(\exp_{\xi}^{g_\xi}(x) \big) = \tilde{W}_\mu(x)$ for any $x \in \R^n$. Since $  \tilde{W}_\mu =  \tilde{U}_\mu$ in $B(0,\delta)$ straightforward computations show that 
\begin{equation}  \label{GJMS.exp2.3}
\begin{aligned}
 & \int_{\R^n}\(\frac{n-2k}{2}\widetilde{W}_{\mu}+x^{i}\partial_{i}\widetilde{W}_{\mu}\)(P_{\mathfrak{g}}\widetilde{W}_{\mu}-\Delta_{0}^{k}\,\widetilde{W}_{\mu})\,dx \\
& = \int_{B(0,\delta)}\(\frac{n-2k}{2}\widetilde{U}_{\mu}+x^{i}\partial_{i}\widetilde{U}_{\mu}\)(P_{\mathfrak{g}}\widetilde{U}_{\mu}-\Delta_{0}^{k}\,\widetilde{U}_{\mu})\,dx\\
& + \bigO( \mu^{n-2k}). \\
\end{aligned} 
\end{equation} 
We compute independently the first term in \eqref{GJMS.exp2.3}. Observe first that 
$$\mu\frac{\partial}{\partial \mu} \widetilde{W}_{\mu}(x)=-\(\dfrac{n-2k}{2}\,\widetilde{W}_{\mu}(x)+x^{i}\partial_{i}\widetilde{W}_{\mu}(x)\)$$
for any $x \in \R^n$. As a consequence, 
\begin{equation} \label{GJMS.exp2.4} 
\begin{aligned} 
 \int_{\R^n} \(\dfrac{n-2k}{2}\,\widetilde{W}_{\mu}+x^{i}\partial_{i}\widetilde{W}_{\mu}\) \Delta_0^k \widetilde{W}_{\mu} \, dx &= - \mu\frac{\partial}{\partial \mu} \int_{\R^n} \big| \Delta_0^{\frac{k}{2}} \widetilde{W}_{\mu} \big|_{\xi}^2 \, dx \\
& =  \bigO ( \mu^{n-2k} ),  
\end {aligned} 
\end{equation}
where the last line follows from the equality
$$\int_{\R^n}\big| \Delta_0^{\frac{k}{2}} \widetilde{W}_{\mu} \big|_{\xi}^2 \, dx  = \int_{\R^n}\big| \Delta_0^{\frac{k}{2}} \widetilde{U}_{\mu} \big|_{\xi}^2 \, dx + \bigO \( \mu^{n-2k} \) = \int_{\R^n} \big| \Delta_0^{\frac{k}{2}} U \big|_{\xi}^2 + \bigO ( \mu^{n-2k} )$$
where $U$ is as in \eqref{bubble1}, and where the latter expansion can be differentiated in $\mu$. Using \eqref{GJMS.exp2.4} and the self-adjointness of $P_{\mathfrak{g}}$ we can thus write
\begin{equation} \label{GJMS.exp2.5}
\begin{aligned}
& \int_{\R^n}\(\frac{n-2k}{2}\widetilde{W}_{\mu}+x^{i}\partial_{i}\widetilde{W}_{\mu}\)(P_{\mathfrak{g}}\widetilde{W}_{\mu}-\Delta_{0}^{k}\,\widetilde{W}_{\mu})\,dx \\
 & =  \int_{\R^n}\(\frac{n-2k}{2}\widetilde{W}_{\mu}+x^{i}\partial_{i}\widetilde{W}_{\mu}\)P_{\mathfrak{g}}\widetilde{W}_{\mu} \, dx + \bigO ( \mu^{n-2k} ) \\
 & = - \frac{1}{2}\, \mu\frac{d}{d\mu} \( \int_{\R^n} \widetilde{W}_{\mu}P_{\mathfrak{g}}\widetilde{W}_{\mu} \, dx \) +\bigO( \mu^{n-2k}),\\
  & = - \frac{1}{2}\, \mu\frac{d}{d\mu} \( \int_{M} W_{\xi, \mu} P_{g_{\xi}}W_{\xi, \mu} \, dv_{g_\xi} \) +\bigO( \mu^{n-2k}),\\
\end{aligned} 
\end{equation}
where the last line follows from \eqref{conf.2}. We have 
\begin{equation} \label{GJMS.exp2.6}
\begin{aligned}
\int_{M}W_{\xi, \mu} P_{g_{\xi}}W_{\xi, \mu} \, dv_{g_\xi} &=\(\int_{M}W_{\xi, \mu}^{\,2^{*}_{k}}\,dv_{g_{\xi}}\)^{\frac{n-2k}{n}}I_{k,g_{\xi}}(W_{\xi, \mu})\\
&=\(\|U\|^{2^{*}_{k}}_{L^{2^{*}_{k}}(\R^{n})}+\bigO(\mu^{n})\)I_{k,g_{\xi}}(W_{\xi, \mu}),
\end{aligned}
\end{equation}
where we have let, for $u\in C^{2k}(M), u\not\equiv0$:
\begin{align*}
I_{k,g_\xi}(u):=\frac{\displaystyle{\int_{M}u\,P_{g_\xi}u\,dv_{g_\xi}}}{\(\displaystyle{\int_{M}|u|^{\,2^{*}_{k}}\,dv_{g_\xi}}\)^{\frac{n-2k}{n}}}. 
\end{align*}
It was recently proven in \cite{MazumdarVetois} that
\begin{align} \label{DL.Mazumdar.Vetois}
&I_{k,g_{\xi}}(W_{\xi, \mu})=\Vert U \Vert_{L^{2^*_k}(\R^n)}^{\frac{4k}{n-2k}} \notag\\
&~-\,\mathcal{C}(n,k)\times\left\{\begin{aligned}&|\Weyl_{g}(\xi)|^{2}\mu^{4}\ln(1/\mu)+\bigO(\mu^{4})&&\text{if }n=2k+4\\&|\Weyl_{g}(\xi)|^{2}\mu^{4}+\smallo(\mu^{4})&&\text{if }n=2k+5,\end{aligned}\right.
\end{align}
for some positive constant $\mathcal{C}(n,k)$. Differentiating the latter with respect to $\mu$ and combining the latter with \eqref{GJMS.exp2.2}, \eqref{GJMS.exp2.3}, \eqref{GJMS.exp2.5} and \eqref{GJMS.exp2.6} shows that the constant $C$ in  \eqref{GJMS.exp2.2} is positive and concludes the proof of~Proposition \ref{GJMS.exp2}.
 \end{proof}
 
We will heavily rely on Proposition \ref{GJMS.exp2} in our proof of Theorem \ref{theo.main}. The left-hand side of \eqref{pohozaev.quadratic.form} --- the so-called Pohozaev quadratic form --- crucially appears as a local obstruction when one performs a local analysis of a sequence $(u_\alpha)_\alpha$ of blowing-up solutions of \eqref{eq:zero} around its concentration points. 
 The novelty of our approach in this paper lies precisely in the proof of Proposition \ref{GJMS.exp2}, where we show that the left-hand side of \eqref{pohozaev.quadratic.form} is obtained by differentiating with respect to $\mu$ the energy associated to a suitable Riemannian version of $\tilde{U}_\mu$ (see \eqref{DL.Mazumdar.Vetois}). Obtaining the expansion to fourth-order of this energy is a complicated task --- which again relies on Juhl's formulae --- and was only recently achieved in \cite{MazumdarVetois}. What Proposition \ref{GJMS.exp2} shows is that the expansion of \cite{MazumdarVetois} is enough to obtain a local sign on \eqref{pohozaev.quadratic.form}, and in particular it shows that we do not need to precisely compute the functions $F_1, \dots, F_5$ arising in  \eqref{GJMS.exp1.eq} --- or, equivalently, the exact numerical value of the constants appearing in \eqref{GJMS.exp0.1}.   
Proposition \ref{GJMS.exp2} is therefore a significant technical simplification and the cornerstone of our proof, since it is unclear to us whether a tractable and fully explicit expression of $F_1, \dots, F_5$ can be obtained. Note, as was already observed in \cite{Marques}, that we use the antisymmetry of the fifth-order term $F_5$ which ensures that remainder terms in \eqref{GJMS.exp1.eq} can be computed at the desired precision and yields $\bigO(\delta \mu^{n-2k})$ in dimensions $n \le 2k+5$.

\section{Local analysis around a blow-up point}\label{sec.local.analysis}

We introduce the setting considered in this paper. We let $f\in C^{\infty}(M)$, $f>0$ be fixed. Throughout this paper we will consider a sequence of positive functions $(u_{\alpha})_{\alpha} \in C^{\infty}(M)$ satisfying :
\begin{equation}\label{eq:one}
P_{g}u_{\alpha}=f^{p_{\alpha}-2^{*}_{k}}~u_{\alpha}^{p_{\alpha}-1}~\hbox{ in } M,
\end{equation}
where $p_{\alpha}\leq2^{*}_{k}$ for all $\alpha$ and $\lim \limits_{\alpha\to+\infty}p_{\alpha}=2^{*}_{k}$. If $(u_{\alpha})_{\alpha}$ is uniformly bounded in $L^\infty(M)$ it converges, up to a subsequence, to a positive smooth solution $u$ of $P_gu = u^{2^{*}_{k}-1}$ in $M$. We thus investigate the case where $(u_{\alpha})_{\alpha}$ does not have an a priori $L^\infty(M)$ bound and we assume that it blows up as $\alpha \to + \infty$, that is 
\begin{equation} \label{blowup}
\Vert \ua \Vert_{L^\infty(M)} \to + \infty \quad \text{ as } \alpha \to + \infty. 
\end{equation}
We follow the strategy of proof in \cite{DruetYlowdim, Marques, LiZhang1, KhuMaSc, LiXiong, GongKimWei} for the Yamabe, Paneitz and third-order equations. In the $1 \le k \le 3$ cases, a major step in the proof of the compactness result consisted in showing that, at every concentration point of the sequence $(u_{\alpha})_{\alpha}$, the Weyl tensor $\Weyl_g$ of $(M,g)$ vanishes at a sufficiently large order given by $[\frac{n-2k-4}{2}]$. This step is referred to in the literature as the Weyl vanishing conjecture. In the case $k \ge 4$ we follow the same approach: we perform, in Sections \ref{sec.local.analysis} and \ref{sec.order.two}, an asymptotic analysis of $(u_{\alpha})_{\alpha}$ around its concentration points and we adapt the arguments in \cite{DruetYlowdim, Marques, LiZhang1, KhuMaSc, LiXiong, GongKimWei} to a general polyharmonic setting $k \ge 4$. When $k \ge 4$, however, no explicit expression of $P_g$ is available and this is the main obstacle to adapting the proofs of the Weyl vanishing conjecture as they are performed in \cite{KhuMaSc} (when $k=1$) and \cite{GongKimWei} (when $k=2,3$) up to the maximal dimension where stability holds. For a general $k\ge4$ we are nevertheless able to adapt the analysis of \cite{Marques} to show the vanishing of the Weyl tensor at a concentration point: we follow for this, overall, the strategy of the works cited above but we include new ideas from our analysis in Section \ref{sec.GJMS.expan}. We begin with the following result, which is inspired from  \cite[Proposition 7.1]{LiXiong} and constructs a suitable family of critical points of $(\ua)_\alpha$:

\begin{proposition} \label{weak.es}
Let $(\ua)_\alpha$ be a sequence of positive solutions of \eqref{eq:one} satisfying \eqref{blowup}. Assume that \eqref{positivity} is satisfied. There exist $N_\alpha \ge 1$ points $(x_{1,\alpha} , \dotsc, x_{N_\alpha, \alpha})$ of $M$ satisfying, up to a subsequence, 
\begin{enumerate}
\item $\nabla \ua (x_{i, \alpha}) = 0 $ for $1 \le i \le N_\alpha$, 
\item $d_g \left( x_{i, \alpha}, x_{j, \alpha} \right)^{\frac{2k}{p_\alpha-2}} u_\alpha(x_{i,\alpha}) \ge 1$ for all $1 \le i \neq j \le N_\alpha$, and
\item there exists a positive constant $C$ independent of $\alpha$ such that
\begin{equation} \label{contptsconc}
 \Big(  \min_{1 \le i \le N_\alpha} d_g \left( x_{i,\alpha}, x \right) \Big)^{\frac{2k}{p_\alpha-2}}  \ua(x) \le C 
 \end{equation}
for any $x \in M$.
\end{enumerate}
\end{proposition}

\begin{proof}
First, an adaptation of Lemma $6.6$ in Hebey~\cite{HebeyZLAM} shows that for any $\alpha$ there exist $N_\alpha \ge 1$ critical points $x_{1,\alpha}, \dotsc, x_{N_\alpha, \alpha}$ of $\ua$ that satisfy the following: for any $1 \le i \neq j \le N_\alpha$, one has
$$ d_g(x_{i, \alpha},x_{j,\alpha})^{\frac{2k}{p_\alpha-2}} u_\alpha(x_{i,\alpha}) \ge 1 $$
and for any critical point $x$ of $\ua$, one has 
\begin{equation}  \label{controleptscrit}
\Big(  \min_{1 \le i \le N_\alpha} d_g \left( x_{i,\alpha}, x \right) \Big)^{\frac{2k}{p_\alpha-2}}  u_\alpha(x) \le 1.
\end{equation}
We prove \eqref{contptsconc} by contradiction: we let, up to a subsequence, $\ya \in M$ be such that
\begin{equation} \label{conc1}
\begin{aligned}
 \Big( & \min_{1 \le i \le N_\alpha} d_g \left( x_{i,\alpha}, \ya \right) \Big)^{\frac{2k}{p_\alpha-2}}  u_\alpha(\ya) \\
 &= \max_{y \in M } \Big(  \min_{1 \le i \le N_\alpha} d_g \left( x_{i,\alpha},  y\right) \Big)^{\frac{2k}{p_\alpha-2}}  \ua(y)  \longrightarrow + \infty
\end{aligned}
\end{equation}
  as $\alpha \to + \infty$. Letting $\nu_\alpha:= u_\alpha(\ya)^{-(p_\alpha-2)/2k}$, \eqref{conc1} shows that
  \begin{equation} \label{conc2}
\frac{1}{\nu_\alpha} \left( \min_{1 \le i \le N_\alpha} d_g \left( x_{i,\alpha}, \ya \right)  \right) \to + \infty \textrm{ as } \alpha \to + \infty .
  \end{equation}
For $0 < \delta < \frac12 i_g(M)$ and $x \in B\(0,\delta/\nu_k\)$ we define
$$ \hua(x):= \nu_\alpha^{\frac{2k}{p_\alpha-2}}\ua \( \exp_{\ya}^g (\nu_\alpha x) \). $$
Here $B\(0,\delta/\nu_k\)$ denotes the Euclidean ball of centre $0$ and radius $\delta/\nu_k$. Using \eqref{conc1} and \eqref{conc2} we have $\hua(0) = 1$ and, for $R >0$,
$$ |\hua(x)| \le 1 + \smallo(1)\quad\textrm{for any } x \in B(0,R).$$
Since $\ua$ satisfies \eqref{eq:one} it is easily seen that $\hua$ satisfies $P_{g_\alpha} \hua = \hat{f}_\alpha^{p_\alpha - 2_k^*} \hua^{p_\alpha-1}$ in $B(0,\delta/\nu_k)$, where we have let $\hat{f}_\alpha = f\big(   \exp_{\ya}^g (\nu_\alpha \cdot )\big)$ and $g_\alpha =  (\exp_{\ya}^g )^*g(\nu_\alpha \cdot)$. The sequence $(g_\alpha)_\alpha$ strongly converges to the Euclidean metric in $C^{2k}_{loc}(\R^n)$, so that by standard elliptic theory $\hua$ converges in $C^{2k}_{loc}(\mathbb{R}^n)$ towards a function $\hat{u}_0$  which satisfies $0 \le \hat{u}_0 \le 1$, $\hat{u}_0(0) = 1$ and  solves 
$$ \Delta_0^k  \hat{u}_0 = \hat{u}_0^{2_k^*-1}\quad\text{in }\R^n.$$
Since $P_g$ is formally self-adjoint, by definition of the Green's function $G_g$ (see the discussion after \eqref{positivity}) and by \eqref{eq:one} we have the following representation formula: for any $z \in M$ fixed, 
$$ u_\alpha(z) = \int_M G_g(z, \cdot) f^{p_{\alpha}-2^{*}_{k}}~u_{\alpha}^{p_{\alpha}-1} dv_g. $$
Let $x \in B(0,R)$ be fixed. We apply the latter formula to the point $z = \exp_{\ya}^g(\nu_\alpha x)$: using assumption \eqref{positivity} and a simple change of variables yields
$$ \hua(x) \ge \int_{B(0,R)} \hat{G}_{\alpha}(x,\cdot) \hat{f}_\alpha^{p_\alpha-2_k^*} \hua^{p_\alpha-1} dv_{g_\alpha}, $$
where we have let  $\hat{G}_{\alpha}(x,y) = \nu_\alpha^{n-2k} G_g \big(  \exp_{\ya}^g(\nu_\alpha x),  \exp_{\ya}^g(\nu_\alpha y) \big)$. Expansion \eqref{expansion.Green} below shows that $\hat{G}_\alpha(x,y) \to b_{n,k} |x-y|^{2k-n} $ pointwise in $\R^n \backslash \{x\}$ as $\alpha \to + \infty$, where $b_{n,k}$ is given by \eqref{def.bnk}. Passing to the limit as $\alpha \to + \infty$ with Fatou's lemma then shows that 
\begin{equation} \label{positivity.u0}
 \hat{u}_0(x) \ge \int_{B(0,R)} \frac{b_{n,k}}{|x-y|^{n-2k}} \hat{u}_0(y)^{2_k^*-1} \, dy ,
 \end{equation}
and hence $\hat{u}_0(x) >0$ since $\hat{u}_0 \ge 0$ and $\hat{u}_0(0) = 1$. Thus $\hat{u}_0 >0$ in $\R^n$ and by the classification result of~\cite{WeiXu} we have $ \hat{u}_0 = U$, where $U$ is given by \eqref{bubble1}. The origin is then a non-degenerate critical point of $\hat{u}_0$, which implies in particular that for $\alpha$ large enough $\ua$ possesses a critical point $\za \in M$, with $d_g(\ya, \za) = \smallo(\nu_\alpha)$ and $\nu_\alpha^{2k/(p_\alpha-2)} \ua(\za) = 1 + \smallo(1)$ as $\alpha\to\infty$. But then
$$ \Big(  \min_{1 \le i \le N_\alpha} d_g \left( x_{i,\alpha}, \za \right) \Big)^{\frac{2k}{p_\alpha-2}}  \ua(\za) \longrightarrow + \infty $$
as $\alpha\to\infty$ by \eqref{conc2}, which is in contradiction with \eqref{controleptscrit}. 
\end{proof}

Let $(\ua)_\alpha$ a sequence of positive solutions of \eqref{eq:one}. Throughout the rest of this section we consider sequences $(\xa)_\alpha$ and $(\rhoa)_\alpha$ of points in $M$ and of positive numbers satisfying
\begin{equation} \label{cond.xarhoa}
\begin{aligned}
& \nabla \ua (\xa) = 0 \\
& d_g(\xa, x)^{\frac{2k}{p_\alpha-2}} \ua(x) \le C \quad \text{ for all } x \in B_g(\xa, 8 \rhoa) \\
\end{aligned}
\end{equation}
for all $\alpha \ge 1$, where $C >0$ is a fixed constant independent of $\alpha$. Assumptions \eqref{cond.xarhoa} are for instance satisfied by the family $(x_{i, \alpha})_{1 \le i \le N_\alpha}$ of points constructed in Proposition \ref{weak.es} when we choose $\rhoa = \frac{1}{16} \min_{1 \le i \neq j \le N_\alpha} d_g(x_{i,\alpha}, x_{j,\alpha})$. For $\alpha \ge 1$ we let $\Lambda_{\xa} = \Lambda(\xa, \cdot)$ be the Lee-Parker conformal factor given by \eqref{conf.1}, \eqref{conf.2} and \eqref{conf.3}, and we let 
\begin{equation} \label{eq.defva}
v_{\alpha} =\Lambda^{-1}_{x_{\alpha}}u_{\alpha}.
\end{equation} The conformal invariance of $P_g$ then shows that $(v_{\alpha})_{\alpha}$ is a sequence of smooth positive functions in $M$ satisfiying:  
\begin{equation}\label{eq:bis:0}
P_{g_{\alpha}}v_{\alpha}=f_{\alpha}^{p_{\alpha}-2^{*}_{k}}v_{\alpha}^{p_{\alpha}-1},
\end{equation} 
where we have let $f_{\alpha}=f\Lambda_{x_{\alpha}}$ and, following \eqref{conf.1}, $g_\alpha = \Lambda_{\xa}^{\frac{4}{n-2k}} g$. If we let $x_0 = \lim_{\alpha \to + \infty} \xa \in M$, where the limit is taken up to a subsequence, it is easily seen that $g_{\xa}$ strongly converges to $g_{x_0}$ as $\alpha \to + \infty$.  Property \eqref{conf.1} together with \eqref{cond.xarhoa} ensure that $\va$ satisfies 
\begin{equation} \label{cond.xarhoa.va}
\begin{aligned}
& \va(\xa) = \ua(\xa), \quad \nabla \va (\xa) = 0, \\
& d_{g_\alpha}(\xa, x)^{\frac{2k}{p_\alpha-2}} \va(x) \le C \quad \text{ for all } x \in B_{g_\alpha}(\xa, 8 \rhoa), \\
\end{aligned}
\end{equation}
where  $d_{g_{\alpha}}$ is the geodesic distance with respect to $g_{\alpha}$. We will assume that $(\xa)_\alpha$ and $(\rhoa)_\alpha$ are chosen in addition to satisfy 
\begin{equation} \label{cond.xarhoa.va2}
 \rhoa^{\frac{2k}{p_\alpha-2}} \max_{B_{g_\alpha}(\xa, 4 \rhoa)} \va \to + \infty \quad \text{ as } \alpha \to + \infty,
\end{equation}
which implies in particular that $(\va)_\alpha$ blows up as $\alpha \to + \infty$ since $M$ is compact. Throughout this section we assume that \eqref{cond.xarhoa.va} and \eqref{cond.xarhoa.va2} hold.  We let 
\begin{equation} \label{def.mua}
 \ma = \va(\xa)^{- \frac{2k}{p_\alpha-2}}= \ua(\xa)^{- \frac{2k}{p_\alpha-2}}. 
 \end{equation}
The following result is standard:

\begin{lemma} \label{local.1}
Let $(\va)_\alpha$ be a sequence of positive solutions of \eqref{eq:bis:0} and let $(\xa)_\alpha$ and $(\rhoa)_\alpha$ satisfy \eqref{cond.xarhoa.va} and \eqref{cond.xarhoa.va2}. Assume that \eqref{positivity} holds. As $\alpha \to + \infty$ one has  $\mu_\alpha \to 0$ and 
$$ \mu_\alpha^{\frac{2k}{p_\alpha-2}} \va \big( \exp_{\xa}^{g_{\alpha}} (\ma \cdot )\big) \longrightarrow U $$ 
in $C^{2k}_{loc}(\mathbb{R}^n)$, where $U$ is given by \eqref{bubble1}. 
\end{lemma}

\begin{proof}
Let $\ya \in B_{g_\alpha}(\xa,4 \rhoa)$ be such that  
 $$ \va(\ya)  = \max_{B_{g_\alpha}(\xa, 4 \rhoa)} \va \longrightarrow + \infty $$
as $\alpha \to + \infty$. Let $\na:=  \va (\ya)^{- (p_\alpha-2)/2k}$. By \eqref{cond.xarhoa.va2} one has $\frac{\rhoa}{\na} \to + \infty$ as $\alpha \to + \infty$. For any $x \in B\(0,\rhoa/\na\)$ define  
$$ \hva(x):= \na^{\frac{2k}{p_\alpha-2}} \va \big( \exp_{\xa}^{g_\alpha}(\na x )\big). $$
It is easily seen that $\hva$ satisfies $P_{\hga} \hva = \hat{f}_\alpha^{p_\alpha - 2_k^*} \hva^{p_\alpha-1}$ in $B(0,\rhoa/\na)$, where we have let $\hat{f}_\alpha = f_\alpha \big(   \exp_{\ya}^{\ga} (\na \cdot )\big)$ and $\hga =  (\exp_{\ya}^{g_\alpha} )^*g(\na \cdot)$. Condition \eqref{cond.xarhoa.va} ensures that $d_{\ga}(\xa, \ya) = O(\na)$, so that  $\hat{y}_0 = \lim_{\alpha \to + \infty} \frac{1}{\na} (\exp_{\xa}^{\ga})^{-1}(\ya)$ exists up to a subsequence. The sequence $(\hga)_\alpha$ strongly converges to the Euclidean metric in $C^{2k}_{loc}(\R^n)$, so that by standard elliptic theory $\hva$ converges in $C^{2k}_{loc}(\mathbb{R}^n)$ towards a function $\hat{v}_0$  which satisfies $0 \le \hat{v}_0 \le 1$, $\hat{v}_0(\hat{y}_0) = 1$, and solves 
$$ \Delta_0^k  \hat{v}_0 = \hat{v}_0^{2_k^*-1}\quad\text{in }\R^n,$$
Arguing as in the proof of Proposition \ref{weak.es} (in particular \eqref{positivity.u0}) it is easily shown that $\hat{v}_0$ is positive in $\R^n$, and by the classification result of~\cite{WeiXu} we thus have $\hat{v}_0 = U(x - \hat{y}_0)$. By \eqref{cond.xarhoa.va} $0$ is a critical point of $\hat{v}_0$ and we thus have $\hat{y}_0 = 0$, which implies that $\frac{\na}{\ma} = \hva(0)^{(p_\alpha-2)/2k}\to 1$ as $\alpha \to \infty$, where $\ma$ is as in \eqref{def.mua}, which concludes the proof of Lemma~\ref{local.1}. 
\end{proof}
The next result is a version of the maximum principle adapted to \eqref{eq:bis:0}:

\begin{lemma} \label{local.2}
Let $(\va)_\alpha$ be a sequence of positive solutions of \eqref{eq:bis:0} and let $(\xa)_\alpha$ and $(\rhoa)_\alpha$ satisfy \eqref{cond.xarhoa.va} and \eqref{cond.xarhoa.va2}. Assume that \eqref{positivity} holds. There exists a positive constant $C >0$ such that the following holds: for any sequence $(s_\alpha)_\alpha$ of positive numbers satisfying $0 < s_\alpha \le \rhoa$, we have, for $\alpha \ge 1$, 
$$ \sum_{\ell = 0}^{2k} s_\alpha^{\ell} \Vert \nabla^{\ell} \va \Vert_{L^\infty(\Omega_\alpha)} \le C \min_{\Omega_\alpha} \va, $$
where we have let $\Omega_\alpha = B_{\ga}(x_{\alpha,}3 s_\alpha) \backslash B_{\ga}(x_{\alpha},s_\alpha/3)$.  
\end{lemma}

\begin{proof}
As before we let $G_{\ga}$ be the Green's function of $P_{\ga}$ in $M$. It is positive by \eqref{positivity} and has a uniform lower bound in $M$ by \eqref{bounds.Green}. Let $x \in \Omega_\alpha =  B_{\ga}(x_{\alpha,}3 s_\alpha) \backslash B_{\ga}(x_{\alpha},  s_\alpha/3)$. Since $\va$ satisfies \eqref{eq:bis:0} a representation formula for $\va$ in $M$ writes as
$$ \va(x) = \int_{ B_{\ga}(x_{\alpha}, 4 s_\alpha) \backslash B_{\ga}(x_{\alpha}, s_\alpha/4)} G_{\ga}(x, \cdot) f_\alpha^{p_\alpha-2^*_k} \va^{p_\alpha-1} dv_{\ga} + h_{1,\alpha}(x) + h_{2,\alpha}(x), $$
where we have let 
$$\begin{aligned}
 h_{1,\alpha}(x) & = \int_{M \backslash B_{\ga}(x_{\alpha}, 4 s_\alpha)} G_{\ga}(x, \cdot) f_\alpha^{p_\alpha-2^*_k} \va^{p_\alpha-1} dv_{\ga} \quad \text{ and } \\
 h_{2,\alpha}(x) & = \int_{B_{\ga}(x_{\alpha},s_\alpha/4)} G_{\ga}(x, \cdot) f_\alpha^{p_\alpha-2^*_k} \va^{p_\alpha-1} dv_{\ga} 
 \end{aligned} $$ 
 Let $x,y \in \Omega_\alpha$. It is easily seen that there exists $C >0$ such that  
 $$ d_{\ga}(x,z) \ge \frac{1}{C} d_{\ga}(y, z) \quad \text{ for any } z \in M \backslash  B_{\ga}(x_{\alpha}, 4 s_\alpha)$$
 and 
 $$ d_{\ga}(x,z) \ge \frac{1}{C} d_{\ga}(y, z)  \quad \text{ for any } z \in  B_{\ga}(x_{\alpha},  s_\alpha/4).$$
Differentiating under the integral and using \eqref{bounds.Green} below then shows that 
\begin{equation} \label{local.2.1}
 \sum_{\ell = 0}^{2k} s_\alpha^{\ell} \big| \nabla^{\ell} h_{i,\alpha}(x) \big|_{\ga} \le C h_{i,\alpha}(y) 
 \end{equation}
for some $C >0$ independent of $x,y$ and $\alpha$ and for $i=1,2$. Independently,  the representation formula above can be rewritten as 
$$ \va(x) = \int_{ B_{\ga}(x_{\alpha},4 s_\alpha) \backslash B_{\ga}(x_{\alpha},s_\alpha/4)} G_{\ga}(x, \cdot) V_\alpha \va dv_{\ga} + h_{1,\alpha}(x) + h_{2,\alpha}(x), $$
where $V_\alpha: = f_\alpha^{p_\alpha-2^*_k} \va^{p_\alpha-2}$. Using \eqref{cond.xarhoa.va}, $V_\alpha$ satisfies 
$$|V_\alpha(x)| \le C s_\alpha^{-2k} \quad \text{ for any } \quad x \in B_{\ga}(x_{\alpha}, 4 s_\alpha) \backslash B_{\ga}(x_{\alpha}, s_\alpha/4). $$ 
The conclusion of Lemma \ref{local.2} now follows from \eqref{local.2.1} and from the maximum principle for integral equations of \cite[Proposition 2.3]{JinLiXiong2} (see also \cite[Lemma 4.1, Proposition A.1, Proposition A.2]{LiXiong} for the fourth-order case). 
\end{proof}

We define in the following, for $x \in M$,
\begin{equation}\label{bubble2}
B_{\alpha}(x)=\frac{\mu_{\alpha}^{n-2k-\frac{2k}{p_{\alpha}-2}}}{\left(\mu^{2}_{\alpha}+\mathfrak{c}^{-1}_{n,k}~d_{g_{\alpha}}(x,x_{\alpha})^{2}\right)^{\frac{n-2k}{2}}}
\end{equation}
where $\mathfrak{c}^{-1}_{n,k}$ is as in \eqref{bubble1} and $\ma$ is given by \eqref{def.mua}. The function $B_\alpha$ in \eqref{bubble2} is commonly referred to as a \emph{positive bubble} and models the local loss of compactness for blowing-up sequences  of positive solutions of \eqref{eq:bis:0}. Let $0< \ve < 1$ be fixed. We define the radius of influence of the concentration point $(\xa)_\alpha$ as follows:  
\begin{equation} \label{def.ra}
r_\alpha = \sup \Big \{ r \in (0, \rhoa),  |\va - B_\alpha| \le \ve \Ba  \text{ at every point of } \overline{B_{g_\alpha}(x_\alpha, r)} \Big \}.
\end{equation}
Lemma \ref{local.1} shows that $\frac{\ra}{\ma} \to + \infty$ as $\alpha \to + \infty$. A simple application of Lemma \ref{local.2} shows that there exists $C >0$ such that, for all $x \in B_{\ga}(\xa, 4 \rhoa)$,
\begin{equation} \label{local.est1}
\sum_{\ell=0}^{2k} \big( \ma + d_{\ga}(\xa, x) \big)^{\ell} |\nabla^{\ell} \va(x)|_{\ga} \le C \Ba(x) 
\end{equation}
holds. The following is the main result of this section: 

\begin{proposition} \label{local.3}
Let $(\va)_\alpha$ be a sequence of positive solutions of \eqref{eq:bis:0} and let $(\xa)_\alpha$ and $(\rhoa)_\alpha$ satisfy \eqref{cond.xarhoa.va} and \eqref{cond.xarhoa.va2}. Assume that \eqref{positivity} holds. We let $\ra$ be defined by \eqref{def.ra} and we assume that $\ra \to 0$ as $\alpha \to + \infty$. 

Let, for $x \in B(0,2) \backslash \{0\}$, $\tva(x) = \ma^{2k + \frac{2k}{p_\alpha-2} - n} \ra^{n-2k} \va \big( \exp_{\xa}^{\ga}(\ra x) \big)$. Then there exists $H \in C^{2k}(B(0,2))$ satisfying $H \ge 0$ and  $\Delta_0^k H = 0$ in $B(0,2)$ such that 
$$ \tva \to \frac{\mathfrak{c}_{n,k}^{\frac{n-2k}{2}}}{|x|^{n-2k}} + H \quad \text{ in } C^{2k}_{loc}(B(0,2) \backslash \{0\}) $$
up to a subsequence as $\alpha \to + \infty$.  If in addition $\ra < \rhoa$ for all $\alpha \ge 1$ then $H(0) >0$. 
\end{proposition}

\begin{proof}
We assume throughout this proof that $\ra \to 0$. As a first observation, we claim that there is a sequence $(\ve_\alpha)_{\alpha}$ of positive real numbers that goes to $0$ as $\alpha \to + \infty$ such that, for any $x \in M$, 
\begin{equation} \label{local.3.1}
\va(x) \ge (1 - \ve_\alpha) \Ba(x). 
\end{equation}
Indeed, let $(\ya)_\alpha$ be a sequence of points in $M$. If $d_{\ga}(\ya, \xa) = \bigO(\ma)$, \eqref{local.3.1} follows from Lemma~\ref{local.1}. We may thus assume that $\frac{d_{\ga}(\ya,\xa)}{\ma} \to + \infty$ as $\alpha \to + \infty$. We let $\ya = \exp_{\xa}^{\ga}(\ma \hat{y}_\alpha)$, for some $\hat{y}_\alpha \in B(0, \frac{\ra}{\ma})$, with $|\hat{y}_\alpha| \to + \infty$ as $\alpha \to + \infty$. We write again a representation formula for $\va$ in $M$: using \eqref{positivity} we obtain that 
$$ \va(\ya) \ge \int_{B_{\ga}(\xa, \ra)} G_{\ga}(\ya,\cdot) f_\alpha^{p_\alpha-2^*_k} \va^{p_\alpha-1} dv_{\ga}. $$
Using \eqref{expansion.Green} below, and since $\ra \to 0$, it is easily seen that, for a fixed $x \in B(0,\frac{\ra}{\ma})$, 
$$ b_{n,k}^{-1} d_{\ga}(\xa, \ya)^{n-2k} G_{\ga} \big(\ya, \exp_{\xa}^{\ga}(\ma x) \big) \to 1$$
as $\alpha \to + \infty$. With Lemma~\ref{local.1}, Fatou's lemma then shows that 
$$ \va(\ya) \ge (1+\smallo(1))  d_{\ga}(\xa, \ya)^{2k-n} \ma^{n-2k-\frac{2k}{p_\alpha-2}} b_{n,k} \int_{\R^n} U^{2^*_k-1} \, dx, $$
where $U$ is as in \eqref{bubble1}. Using \eqref{eq.constants.bubble} and \eqref{bubble2} concludes the proof of \eqref{local.3.1}. 

 We now consider $\tva$ as in the statement of Proposition~\ref{local.3}. By \eqref{local.est1} we have 
\begin{equation} \label{local.3.2}
 |\nabla^{\ell} \tva(x)| \le C_\ell \Big( \frac{\ma}{\ra} + |x| \Big)^{2k - \ell -n } \quad \text{ for all } x \in B(0,2) \backslash \{0\},
 \end{equation}
where the norm is taken with respect to the Euclidean metric. By \eqref{eq:bis:0} it is easily seen that $\tva$ satisfies 
$$P_{\tga} \tva = \left( \frac{\ma}{\ra} \right)^{(n-2k)(p_\alpha-2)-2}\tilde{f}_\alpha^{p_\alpha - 2_k^*} \tva^{p_\alpha-1} \quad \text{ in } B(0,2) \backslash \{0\},$$
 where we have let $\tilde{f}_\alpha = f_\alpha \big(   \exp_{\xa}^{\ga} (\ra \cdot )\big)$ and $\tga =  (\exp_{\xa}^{g_\alpha} )^*g(\ra \cdot)$. Since $\ra \to 0$, $\tga$ converges to the Euclidean metric in $C^{p}(B(0,2))$ for all $p \ge 1$, so that by \eqref{local.3.2} and standard elliptic theory $\tva \to \tilde{v}_0$ in $C^{2k}_{loc}(B(0,2) \backslash \{0\})$ where $\tilde{v}_0$ satisfies $ \Delta_0^k \tilde{v}_0  = 0$ in $B(0,2) \backslash \{0\}$ and $ |\nabla^{\ell} \tilde{v}_0(x)| \le C_\ell |x|^{2k-\ell-n}$, so that in particular $\tilde{v}_0 \in W^{2k-1,1}(B(0,2))$. We claim that $\tilde{v}_0$ satisfies 
 \begin{equation}  \label{local.3.3}
   \Delta_0^k \tilde{v}_0 = \frac{\mathfrak{c}_{n,k}^{\frac{n-2k}{2}}}{b_{n,k}} \delta_0 \quad \text{ in } B(0,2) 
 \end{equation}
 in the distributional sense, where $\mathfrak{c}_{n,k}$ is as in \eqref{bubble1} and $b_{n,k}$ is as in \eqref{def.bnk}. To prove this, we let $\vp \in C^\infty_c(B(0,2))$ and define, for $x \in B_{\ga}(\xa, 2\ra)$, $\vp_\alpha(x) = \vp \big( \frac{1}{\ra} (\exp_{\xa}^{\ga})^{-1}(x) \big)$. Since $\vp_\alpha$ is supported in $B_{\ga}(\xa, 2 \ra)$ we get, using Lemma~\ref{local.1}, \eqref{local.3.2} and dominated convergence, that
 \begin{equation} \label{local.3.4}
  \int_{M} f_\alpha^{p_\alpha-2^*_k} \va^{p_\alpha-1} \vp_\alpha v_{\ga} = (1+\smallo(1))\ma^{n-2k-\frac{2k}{p_\alpha-2}} \vp(0) \int_{\R^n} U^{2^*_k-1}\, dx  
 \end{equation}
 as $\alpha \to + \infty$. Independently, \eqref{expansion.Pg} and \eqref{local.est1} show that, for any $x \in B(0, 2\ra) $, 
 \begin{equation} \label{label.tardif.1}
  P_{\ga} \va \big( \exp_{\xa}^{\ga}(x) \big) = \Delta_0^k \tilde{v}_\alpha + \bigO \big( \ma^{n-2k-\frac{2k}{p_\alpha-2}} \big( \ma + |x| \big)^{2  - n} \big), 
  \end{equation}
   where we have let $\tilde{v}_\alpha(x) = \va \big( \exp_{\xa}^{\ga}( x) \big)$. Direct computations using also \eqref{conf.2} then give 
 $$ \begin{aligned}
 \int_M \vp_\alpha P_{\ga} \va dv_{\ga} & = \int_{B(0,2\ra)} \Delta_0^k \vp\,  \tilde{v}_\alpha \, dx + \bigO\big( \ra^2 \ma^{n-2k - \frac{2k}{p_\alpha-2}} \big) \\
 & = \ma^{n-2k - \frac{2k}{p_\alpha-2}} \int_{B(0,2)} \tva \Delta_0^k \vp \, dx + \smallo \big( \ma^{n-2k - \frac{2k}{p_\alpha-2}}\big) \\
& = \ma^{n-2k - \frac{2k}{p_\alpha-2}} \int_{B(0,2)} \tilde{v}_0 \Delta_0^k \vp \, dx + \smallo \big( \ma^{n-2k - \frac{2k}{p_\alpha-2}}\big),
 \end{aligned} $$
where we used again that $\ra \to 0$ as $\alpha \to + \infty$ and where the last line follows from \eqref{local.est1} and dominated convergence. Combining the latter with \eqref{local.3.4} shows that $\tilde{v}_0$ satisfies $\Delta_0^k \tilde{v}_0 = \Big( \int_{\R^n} U^{2^*_k-1}\, dx  \Big) \delta_0$ in $B(0,2)$ in the distributional sense, and \eqref{local.3.3} finally follows from \eqref{eq.constants.bubble}. Using \eqref{local.3.3}, simple regularity arguments then show that 
 \begin{equation} \label{local.3.41}
  \tilde{v}_0(x) = \frac{\mathfrak{c}_{n,k}^{\frac{n-2k}{2}}}{|x|^{n-2k}} + H 
  \end{equation} for every $x \in B(0,2) \backslash \{0\}$, where $H \in C^{2k}_{loc}(B(0,2))$ satisfies $\Delta_0^k H = 0$. That $H$ is nonnegative follows from \eqref{local.3.1}: by \eqref{bubble2} we indeed have, for a fixed $x \in B(0,2) \backslash \{0\}$, 
 $$ \tva(x) \ge (1- \ve_\alpha) \( \frac{\ma^2}{\ra^2} + \mathfrak{c}_{n,k}^{-1} |x|^2 \)^{- \frac{n-2k}{2}}, $$
and passing this expression to the limit as $\alpha \to +\infty$ gives $ \tilde{v}_0(x) \ge \mathfrak{c}_{n,k}^{\frac{n-2k}{2}} |x|^{2k-n}$, which implies that $H \ge 0$ in $B(0,2)$.  We now prove that $H$ still satisfies the Harnack inequality. Precisely, we claim that there exists $C>0$ such that
\begin{equation} \label{local.3.5}
\max_{\overline{B(0,1)}} H \le C \min_{\overline{B(0,1)}} H.
\end{equation}
To prove \eqref{local.3.5} we let $0 < \delta < 1$ be fixed, $x \in B(0,1) \backslash B(0,\delta)$ and we let $\ya = \exp_{\xa}^{\ga}(\ra x) \in B_{\ga}(\xa, \ra) \backslash B_{\ga}(\xa, \delta \ra)$. A representation formula shows that 
\begin{equation} \label{local.3.6} 
\va(\ya) =  \int_{ B_{\ga}(2 \ra)} G_{\ga}(\ya, \cdot) f_\alpha^{p_\alpha-2^*_k} \va^{p_\alpha-1} dv_{\ga} + H_{\alpha}(\ya), 
\end{equation}
where we have let 
$$H_\alpha(\ya) = \int_{ M \backslash B_{\ga}(2 \ra)} G_{\ga}(\ya, \cdot) f_\alpha^{p_\alpha-2^*_k} \va^{p_\alpha-1} dv_{\ga}. $$ 
By \eqref{positivity} and  \eqref{bounds.Green}, arguing as in the proof of Lemma~\ref{local.2}, there exists $C >0$ which does not depend on $\delta$ such that, for any $\alpha \ge 1$, 
\begin{equation} \label{local.3.61}
 \max_{\overline{B_{\ga}(\xa, \ra)}} H_\alpha \le C  \min_{\overline{B_{\ga}(\xa, \ra)}} H_\alpha. 
 \end{equation}
Since $\ra \to 0$, straightforward computations using Lemma~\ref{local.1}, \eqref{local.est1}, \eqref{expansion.Green} and dominated convergence show that 
\begin{equation} \label{local.3.7}
 \begin{aligned}
\int_{B_{\ga}(2 \ra)} &G_{\ga}(\ya, \cdot) f_\alpha^{p_\alpha-2^*_k} \va^{p_\alpha-1} dv_{\ga} \\
& = \Ba(\ya) + o \big( \Ba(\ya) \big)\\
& = (1+\smallo(1)) \frac{\mu_{\alpha}^{n-2k-\frac{2k}{p_{\alpha}-2}}\mathfrak{c}_{n,k}^{\frac{n-2k}{2}}}{\ra^{n-2k}|x|^{n-2k}}
\end{aligned} 
\end{equation}
as $\alpha \to + \infty$ (see for instance the arguments in the proof of Hebey \cite[Proposition 6.1]{HebeyZLAM}). Going back to \eqref{local.3.6} together with \eqref{local.3.7} shows that 
$$ \tva(x) = (1+\smallo(1)) \frac{\mathfrak{c}_{n,k}^{\frac{n-2k}{2}}}{|x|^{n-2k}} + \ma^{\frac{2k}{p_\alpha-2} - n} \ra^{n-2k} H_\alpha \big( \exp_{\xa}^{\ga}(\ra x) \big). $$
Passing to the limit as $\alpha \to + \infty$ then shows, thanks to \eqref{local.3.41}, that 
$$\ma^{\frac{2k}{p_\alpha-2} - n} \ra^{n-2k} H_\alpha \big( \exp_{\xa}^{\ga}(\ra x) \big) \to H(x) $$
pointwise in $B(0,1) \backslash B(0, \delta)$ as $\alpha \to + \infty$. Passing \eqref{local.3.61} to the limit shows that for any $x,y \in B(0,1) \backslash B(0,\delta)$ we have $ H(x) \le C H(y) $ where $C$ is independent of $\delta$. Letting $\delta \to 0$ finally proves \eqref{local.3.5}.

We finally assume that, up to a subsequence, $\ra < \rhoa$. By definition of $\ra$ in \eqref{def.ra} and by \eqref{local.3.1} there thus exists $\ya \in \partial B_{\ga}(\xa, \ra)$ such that $\va(\ya) = (1+\ve) \Ba(\ya)$. Let $z_\alpha = \frac{1}{\ra}( \exp_{\xa}^{\ga})^{-1}(\ya)$, and let $z_0 = \lim_{\alpha \to + \infty} z_\alpha$. Passing to the limit we obtain that $|z_0| = 1$ and that $\tilde{v}_0(z_0) = (1+\ve) \mathfrak{c}_{n,k}^{\frac{n-2k}{2}}$. By \eqref{local.3.41} we thus have $H(z_0) = \ve   \mathfrak{c}_{n,k}^{\frac{n-2k}{2}} >0$, and \eqref{local.3.5} implies that $H(0) >0$. This concludes the proof of Proposition~\ref{local.3}.
\end{proof}
\smallskip

\section{2nd order pointwise estimates} \label{sec.order.two}

We work in the same setting than Section \ref{sec.local.analysis}. We let $(\ua)_\alpha$ be a sequence of positive solutions of \eqref{eq:one} that blows up as in \eqref{blowup}. We let $(\xa)_\alpha$ and $(\rhoa)_\alpha$ be sequences satisfying \eqref{cond.xarhoa}. As in Section \ref{sec.local.analysis} we define the metrics $g_{\alpha}=\Lambda_{x_{\alpha}}^{\frac{4}{n-2k}}g$, which satisfy \eqref{conf.2} and \eqref{conf.3}, and we define $\va$ as in \eqref{eq.defva}, which is a positive solution of \eqref{eq:bis:0} and satisfies \eqref{cond.xarhoa.va}. Throughout this section we will also assume that $(\rhoa)_\alpha$ satisfies \eqref{cond.xarhoa.va2}. In this section we obtain improved estimates on $v_{\alpha}-B_{\alpha}$ around $\xa$, where $\Ba$ is as in \eqref{bubble2}. Our main result in this section is as follows:

\begin{proposition}\label{sym.es}
Let $(\va)_\alpha$ be a sequence of positive solutions of \eqref{eq:bis:0} and let $(\xa)_\alpha$ and $(\rhoa)_\alpha$ satisfy \eqref{cond.xarhoa.va} and \eqref{cond.xarhoa.va2}. Assume that \eqref{positivity} holds and let $\ra$ be defined by \eqref{def.ra}.
We have for $\ell=0,1,\ldots,2k-1$ and $x\in B_{g_{\alpha}}(x_{\alpha}, 2r_{\alpha})$
\begin{align}\label{sym.esII}
&\left|\nabla^{\ell}(v_{\alpha}-B_{\alpha})(x)\right|\lesssim~\frac{\mu_{\alpha}^{n-2k-\frac{2k}{p_{\alpha}-2}}}{r_{\alpha}^{n-2k}}\(\mu_{\alpha}+d_{g_{\alpha}}(x,x_{\alpha})\)^{-\ell}\notag\\
&\qquad+\left\{\begin{aligned}&\(\mu_{\alpha}+d_{g_{\alpha}}(x,x_{\alpha})\)^{4-\ell}\ln\(1+\frac{r_{\alpha}}{\mu_{\alpha}}\)B_{\alpha}(x)&&\text{if }n=2k+4\\&\(\mu_{\alpha}+d_{g_{\alpha}}(x,x_{\alpha})\)^{4-\ell}B_{\alpha}(x)&&\text{if }n>2k+4.\end{aligned}\right.
\end{align}
For $n=2k+4$ the above estimate can be improved to give for all $x\in B_{g_{\alpha}}(x_{\alpha}, 2r_{\alpha})$ and  $\ell=0,1,\ldots,2k-1$, 
\begin{align}\label{sym.esIIb}
\left|\nabla^{\ell}(v_{\alpha}-B_{\alpha})(x)\right|\lesssim&~\frac{\mu_{\alpha}^{n-2k-\frac{2k}{p_{\alpha}-2}}}{r_{\alpha}^{n-2k}}\(\mu_{\alpha}+d_{g_{\alpha}}(x,x_{\alpha})\)^{-\ell}\notag\\
&+\(r_{\alpha}+\mu_{\alpha}\ln\(\frac{1}{\mu_{\alpha}}\)\)\frac{\mu_{\alpha}^{4-\frac{2k}{p_{\alpha}-2}}}{\(\mu_{\alpha}+d_{g_{\alpha}}(x,x_{\alpha})\)^{1+\ell}}.
\end{align}
\end{proposition}
\noindent In \eqref{sym.esII} we gain four orders of smallness in $x$. This is related to the fact that $P_{g_{\alpha}}-\Delta_{0}^{k}$ acting on radial functions vanishes up to $4$th order, and has been known for the cases $k=1$ and $k=2$ since Marques \cite{Marques} and Li--Xiong \cite{LiXiong}. When $k \ge 3$ this is also true and we prove it in Proposition \ref{GJMS.exp0} below. We obtain a local expansion of $P_{g_{\alpha}}-\Delta_{0}^{k}$ to fourth-order even though $P_{g_\alpha}$ is not explicit, and we rely on symmetry arguments and on the energy expansions in Mazumdar--V\'etois \cite{MazumdarVetois} in order to simplify the expansions; in particular we do not need to compute the constants in the expansion of $P_{g_{\alpha}}-\Delta_{0}^{k}$.

\begin{proof}
We divide the proof into several steps, successively improving the precision of error estimates. 
\smallskip

\noindent
{\textbf{Step 1:}} We start by estimating the closeness of $p_{\alpha}$ to $2^{*}_{k}$. 
\begin{claim}\label{sym.es4}
We have
\begin{equation*}
2^{*}_{k}-p_{\alpha}=\bigO(\mu_{\alpha}^{2})+\bigO\(\(\frac{\mu_{\alpha}}{r_{\alpha}}\)^{n-2k}\)\quad\hbox{as}~\alpha\to+\infty.
\end{equation*}
\end{claim}
\begin{proof}
Let $\widetilde{v}_{\alpha}:=v_{\alpha}\circ\exp_{x_{\alpha}}^{g_{\alpha}}$. The Pohozaev identity \eqref{poho.id0} gives 
\begin{align*}
&\int_{B(0,r_{\alpha})}\(\frac{n-2k}{2}\,\widetilde{v}_{\alpha}+x^{i}\partial_{i}\widetilde{v}_{\alpha}\)\(\Delta_{0}^{k}\widetilde{v}_{\alpha}-\tilde{f}_{\alpha}^{\,p_{\alpha}-2^{*}_{k}}~\widetilde{v}_{\alpha}^{\,p-1}\)\,dx~ \\
& =~\mathcal{P}_{k}(r_{\alpha};\widetilde{v}_{\alpha}) \, -\frac{r_{\alpha}}{p_{\alpha}}\int_{\partial B(0,r_{\alpha})}\tilde{f}_{\alpha}^{\,p_{\alpha}-2^{*}_{k}}~\widetilde{v}_{\alpha}^{\,p_{\alpha}}\,d\sigma
\\
&- \(\frac{n-2k}{2}-\frac{n}{p_{\alpha}}\)\int_{B(0,r_{\alpha})}\tilde{f}_{\alpha}^{\,p_{\alpha}-2^{*}_{k}}~\widetilde{v}_{\alpha}^{\,p_{\alpha}}\,dx+\frac{1}{p_{\alpha}}\int_{B(0,r_{\alpha})}x^{i}\,\partial_{i}\tilde{f}_{\alpha}^{\,p_{\alpha}-2^{*}_{k}}~\widetilde{v}_{\alpha}^{\,p_{\alpha}}\,dx\notag,\\
\end{align*}
where we have let $\tilde{f}_\alpha = f_\alpha \circ\exp_{x_{\alpha}}^{g_{\alpha}}$. Here $\mathcal{P}_{k}(r_{\alpha};\widetilde{v}_{\alpha})$ denotes boundary terms whose expression depends on the parity of $k$ and which are defined in \eqref{poho.id01}. With \eqref{local.est1}, these boundary terms satisfy 
\begin{equation} \label{est.boundary.terms} \begin{aligned}
|\mathcal{P}_{k}(r_{\alpha};\widetilde{v}_{\alpha})| 
\lesssim\mu_{\alpha}^{n-2k-\frac{4k}{p_{\alpha}-2}}\(\frac{\mu_{\alpha}}{r_{\alpha}}\)^{n-2k}.
\end{aligned}
\end{equation}
Using \eqref{local.est1} and \eqref{label.tardif.1} we have
\begin{align*}
&\int_{B(0,r_{\alpha})}\(\frac{n-2k}{2}\widetilde{v}_{\alpha}+x^{i}\partial_{i}\widetilde{v}_{\alpha}\)\(\Delta_{0}^{k}\,\widetilde{v}_{\alpha}-f_{\alpha}^{p_{\alpha}-2^{*}_{k}}\,\widetilde{v}_{\alpha}^{\,p-1}\)dx\\
&=\bigO\(\mu_{\alpha}^{n-2k-\frac{4k}{p_{\alpha}-2}}\(\frac{\mu_{\alpha}}{r_{\alpha}}\)^{n-2k}\)+\bigO\(\mu_{\alpha}^{2+n-2k-\frac{4k}{p_{\alpha}-2}}\).
\end{align*}
Next, using again \eqref{local.est1} and the dominated convergence theorem we obtain
$$\begin{aligned}
& \frac{1}{p_{\alpha}}\int_{B(0,r_{\alpha})}x^{i}\,\partial_{i}f_{\alpha}^{\,p_{\alpha}-2^{*}_{k}}~\widetilde{v}_{\alpha}^{\,p_{\alpha}}\,dx -\frac{r_{\alpha}}{p_{\alpha}}\int_{\partial B(0,r_{\alpha})}f_{\alpha}^{\,p_{\alpha}-2^{*}_{k}}~\widetilde{v}_{\alpha}^{\,p_{\alpha}}\,d\sigma~\\
 & = \smallo\(\mu_{\alpha}^{n-2k-\frac{4k}{p_{\alpha}-2}}\(\frac{\mu_{\alpha}}{r_{\alpha}}\)^{n-2k}\) \\
\end{aligned}$$
and
$$\begin{aligned}
& \(\frac{n-2k}{2}-\frac{n}{p_{\alpha}}\)\int_{B(0,r_{\alpha})}f_{\alpha}^{\,p_{\alpha}-2^{*}_{k}}~\widetilde{v}_{\alpha}^{\,p_{\alpha}}\,dx \\
&  \quad \quad =\frac{(n-2k)^{2}}{4n}\[(p_{\alpha}-2^{*}_{k})\|U\|^{2^{*}_{k}}_{L^{2^{*}_{k}}(\R^{n})}+\smallo(2^{*}_{k}-p_{\alpha})\]\mu_{\alpha}^{n-2k-\frac{4k}{p_{\alpha}-2}}.\\
\end{aligned}$$
Thus combining the bounds on all the terms we obtain Claim \ref{sym.es4}.
\end{proof}

\noindent
{\textbf{Step 2:}} We obtain refined second-order estimates on $\va$. We proceed as in \cite{Marques} and we will work at the level of rescaled quantities. Let
\begin{align*}
\widehat{v}_{\alpha}(x):=\mu_{\alpha}^{\frac{2k}{p_{\alpha}-2}} v_{\alpha}\(\exp_{x_{\alpha}}^{g_{\alpha}}(\mu_{\alpha}x)\),\,w_{\alpha}(x):=\widehat{v}_{\alpha}(x)-U(x) 
\end{align*}
for  $x\in B\(0,2\mu_{\alpha}^{-1}r_{\alpha}\)$, where $U$ is as in \eqref{bubble1}. 
\begin{claim}
For any $x \in B(0, 2 \ma^{-1} \ra)$ and for any $0 \le \ell \le 2k-1$ we have  
\begin{equation} \label{sym.es15}
\begin{aligned}
\left|\nabla^{\ell}w_{\alpha}(x)\right|& \lesssim\(\frac{\mu_{\alpha}}{r_{\alpha}}\)^{n-2k+\ell} +  \(\frac{\mu_{\alpha}}{r_{\alpha}}\)^{n-2k}\(1+|x|\)^{-\ell}\\
&+\[(2^{*}_{k}-p_{\alpha})+\|w_{\alpha}\|_{L^{\infty}(B(0,2\mu^{-1}_{\alpha}r_{\alpha}))}\] \(1+|x|\)^{2k-n-\ell} \\
&~+\left\{\begin{aligned}&\mu_{\alpha}^{4}\ln\(1+\frac{r_{\alpha}}{\mu_{\alpha}}\)\(1+|x|\)^{-\ell}&&\text{if }n=2k+4\\&\mu_{\alpha}^{4}\(1+|x|\)^{2k+4-n-\ell}&&\text{if }n>2k+4.\end{aligned}\right.
\end{aligned}
\end{equation}
\end{claim}
\begin{proof}[Proof of \eqref{sym.es15}] Since $\ga$ satisfies \eqref{conf.1}, \eqref{conf.2} and \eqref{conf.3} we may now use the expansion of $P_{g_\alpha}$ in conformal normal coordinates given in \eqref{GJMS.exp1} below. We obtain
\begin{equation}\label{sym.es5}
\begin{aligned}
P_{g_{\alpha}}B_\alpha(x) - B_{\alpha}(x)^{2^*_k-1}&  = \bigO\(\frac{\mu_{\alpha}^{n-2k-\frac{2k}{p_{\alpha}-2}}}{\(\mu_{\alpha}+d_{g_{\alpha}}(x,x_{\alpha})\)^{n-4}}\).
\end{aligned}
\end{equation}
From Lemma~\ref{local.1} it follows that $w_{\alpha}\to0$ in $C^{2k}_{loc}(\R^{n})$ as $\alpha\to+\infty$. And from \eqref{local.est1} we have, for $0 \le \ell \le 2k-1$
\begin{align}\label{sym.es7}
|\nabla^{\ell} w_{\alpha}(y)|\lesssim \frac{\mu_{\alpha}^{n-2k}}{r_{\alpha}^{n-2k+\ell}} \quad\hbox{ for }y \in B(0,2\mu^{-1}_{\alpha}r_{\alpha})\setminus B(0,\mu^{-1}_{\alpha}r_{\alpha}).
\end{align}

\noindent
For $x \in B(0, 2 \ma^{-1} \ra)$ we define $\widehat{g}_{\alpha}(x):=\(\exp_{x_{\alpha}}^{g_{\alpha}}\)^{*}g_{\alpha}(\mu_{\alpha}x)$ and $\widehat{f}_{\alpha}(x):=f_{\alpha}\(\exp_{x_{\alpha}}^{g_{\alpha}}(\mu_{\alpha}x)\)$. By equation \eqref{eq:bis:0} it is easily seen that $\hva$ satisfies 
\begin{align}\label{sym.es8}
P_{\widehat{g}_{\alpha}}w_{\alpha}=&\(\widehat{f}_{\alpha}^{~p_{\alpha}-2^{*}_{k}}-1\)\widehat{v}_{\alpha}^{~p_{\alpha}-1}+\(\widehat{v}_{\alpha}^{~p_{\alpha}-1}-U^{p_{\alpha}-1}\)+\(U^{p_{\alpha}-1}-U^{2^{*}_{k}-1}\)\notag\\
&+\(U^{2^{*}_{k}-1}-P_{\widehat{g}_{\alpha}}U\).
\end{align}
We control each of the terms in the brackets. First, since $\widehat{v}_{\alpha}\lesssim U$ for $x\in B(0,2\mu_{\alpha}^{-1}r_{\alpha})$ uniformly by \eqref{local.est1}, we have 
$$ \begin{aligned}
\(\widehat{f}_{\alpha}^{~p_{\alpha}-2^{*}_{k}}-1\)\widehat{v}_{\alpha}^{~p_{\alpha}-1} + U^{p_{\alpha}-1}-U^{2^{*}_{k}-1}
&~= \bigO\((2^{*}_{k}-p_{\alpha})U^{2^{*}_{k}-1}|\ln U|\).
\end{aligned} $$ 
Next, for $\alpha\gg1$ we have for some fixed $0<\theta<2^{*}_{k}-2$
\begin{align*}
\widehat{v}_{\alpha}^{~p_{\alpha}-1}-U^{p_{\alpha}-1}
&=\(p_{\alpha}-1\)U^{p_{\alpha}-2}w_{\alpha}+\bigO\(U^{p_{\alpha}-2-\theta}|w_{\alpha}|^{1+\theta}\).
\end{align*}
This is uniform in $B(0,2\mu_{\alpha})$  since $|w_{\alpha}(x)|\lesssim U(x)$ in  $B(0,2\mu_{\alpha}^{-1}r_{\alpha})$ by \eqref{local.est1}. A change of variables in  \eqref{sym.es5} gives 
\begin{align}\label{sym.es9}
\left|P_{\widehat{g}_{\alpha}}U(x)-U(x)^{2^{*}_{k}-1}\right|= \bigO \left( \mu_{\alpha}^{4}\(1+|x|\)^{4-n}\right).
\end{align}
Collecting all the terms we have obtained that
\begin{align}\label{sym.es10}
P_{\widehat{g}_{\alpha}}w_{\alpha}
=&\(p_{\alpha}-1\)U^{p_{\alpha}-2}w_{\alpha} +\bigO\(U^{p_{\alpha}-2-\theta}|w_{\alpha}|^{1+\theta}\) \notag \\
&+\bigO\((2^{*}_{k}-p_{\alpha})U^{2^{*}_{k}-1}|\ln U|\)+\bigO\( \mu_{\alpha}^{4}\(1+|x|\)^{4-n}\).
\end{align}
Thus we have in particular
\begin{equation}\label{sym.es11}
|P_{\widehat{g}_{\alpha}}w_{\alpha}|\lesssim U^{p_{\alpha}-2}|w_{\alpha}|+(2^{*}_{k}-p_{\alpha})U^{2^{*}_{k}-1}|\ln U|+ \mu_{\alpha}^{4}\(1+|x|\)^{4-n}.
\end{equation}
A simple change of variables shows that 
$$\widehat{G}_{\alpha}(x,y):=\mu_{\alpha}^{n-2k}G_{\alpha}\(\exp_{x_{\alpha}}^{g_{\alpha}}(\mu_{\alpha}x),\exp_{x_{\alpha}}^{g_{\alpha}}(\mu_{\alpha}y)\)$$
is the fundamental solution for $P_{\widehat{g}_{\alpha}}$ in  $B(0,2\mu_{\alpha}^{-1}r_{\alpha})$ and satisfies $\widehat{G}_{\alpha}(x,y)\lesssim |y-x|^{2k-n}$ by \eqref{bounds.Green} below. We can then write a representation formula for $w_{\alpha}$ in  $B(0,2\mu_{\alpha}^{-1}r_{\alpha})$. Let $y_{\alpha}\in B(0,\mu_{\alpha}^{-1}r_{\alpha})$ be a sequence of points. Using \eqref{sym.es7} and \eqref{sym.es11} we get
\begin{align}\label{sym.es19}
|w_{\alpha}(y_{\alpha})|&\lesssim\int_{B(0,2\mu_{\alpha}^{-1}r_{\alpha})}|y_{\alpha}-y|^{2k-n} U(y)^{p_{\alpha}-2}|w_{\alpha}(y)| \, dy\notag\\
&~+(2^{*}_{k}-p_{\alpha})\int_{B(0,2\mu_{\alpha}^{-1}r_{\alpha})}|y_{\alpha}-y|^{2k-n}U(y)^{2^{*}_{k}-1}|\ln U(y)|\,dy\notag\\
&~+\mu_{\alpha}^{4}\int_{B(0,2\mu_{\alpha}^{-1}r_{\alpha})}|y_{\alpha}-y|^{2k-n}\(1+|y|\)^{4-n}\,dy+\(\frac{\mu_{\alpha}}{r_{\alpha}}\)^{n-2k}.
\end{align}
Using Lemma~\ref{lemma.Giraud} below we obtain that 
\begin{align}
\int_{B(0,2\mu_{\alpha}^{-1}r_{\alpha})}|y_{\alpha}-y|^{2k-n}\(U(y)^{2^{*}_{k}-1}|\ln U(y)|\)\,dy\lesssim U(y_{\alpha})
\end{align}
and that 
\begin{align}
\mu_{\alpha}^{4}\int_{B(0,2\mu_{\alpha}^{-1}r_{\alpha})}|y_{\alpha}-y|^{2k-n}\(1+|y|\)^{4-n}\,dy\lesssim\mathcal{I}_{\alpha}(y_{\alpha}),
\end{align}
where we have let 
\begin{align}\label{def.I}
\mathcal{I}_{\alpha}(x):=\left\{\begin{aligned}&~\mu_{\alpha}^{n-2k}r_{\alpha}^{2k+4-n}&&\text{if }n<2k+4\\&\mu_{\alpha}^{4}\ln\(1+\frac{r_{\alpha}}{\mu_{\alpha}}\)&&\text{if }n=2k+4\\&\mu_{\alpha}^{4}\(1+|x|\)^{2k+4-n}&&\text{if }n>2k+4 \end{aligned}\right. .
\end{align}
Thus, \eqref{sym.es19} becomes 
\begin{equation} \label{sym.es13}
\begin{aligned}
|w_{\alpha}(y_{\alpha})|\lesssim&\,\int_{B(0,2\mu_{\alpha}^{-1}r_{\alpha})}|y_{\alpha}-y|^{2k-n} U^{p_{\alpha}-2}|w_{\alpha}|\,dy \\
&~+(2^{*}_{k}-p_{\alpha})U(y_{\alpha}) +\mathcal{I}_{\alpha}(y_{\alpha})+\(\frac{\mu_{\alpha}}{r_{\alpha}}\)^{n-2k}.
\end{aligned}
\end{equation}
To estimate the integral in \eqref{sym.es13} we naively bound $w_\alpha$ by its $L^\infty$ norm and use again Lemma~\ref{lemma.Giraud} below: we get 
\begin{equation} \label{sym.es12}
\begin{aligned}
&\int_{B(0,2\mu_{\alpha}^{-1}r_{\alpha})}|y_{\alpha}-y|^{2k-n} U(y)^{p_{\alpha}-2}|w_{\alpha}(y)|\,dy\\
&\lesssim\|w_{\alpha}\|_{L^{\infty}(B(0,2\mu^{-1}_{\alpha}r_{\alpha}))}\(\(1+|y_{\alpha}|\)^{2k-(p_{\alpha}-2)(n-2k)}+U(y_{\alpha})\).\\
\end{aligned}
\end{equation}
Using \eqref{sym.es12} in \eqref{sym.es13} gives that for any sequence of points $y_{\alpha}\in B(0,2\mu_{\alpha}^{-1}r_{\alpha})$
\begin{align}
|w_{\alpha}(y_{\alpha})|\lesssim&~\|w_{\alpha}\|_{L^{\infty}(B(0,2\mu^{-1}_{\alpha}r_{\alpha}))}\(\(1+|y_{\alpha}|\)^{2k-(p_{\alpha}-2)(n-2k)}+B_{0}(y_{\alpha})\)\notag\\
&+(2^{*}_{k}-p_{\alpha})U(y_{\alpha})+\mathcal{I}_{\alpha}(y_{\alpha})+\(\frac{\mu_{\alpha}}{r_{\alpha}}\)^{n-2k}.
\end{align}
We now use this new estimate to compute again the integral term in \eqref{sym.es13} and improve its precision. After a finite number of iterations we obtain that 
\begin{align}\label{sym.es14}
|w_{\alpha}(y)|\lesssim&~\[(2^{*}_{k}-p_{\alpha})+\|w_{\alpha}\|_{L^{\infty}(B(0,2\mu^{-1}_{\alpha}r_{\alpha}))}\]U(y)+\mathcal{I}_{\alpha}(y),\notag\\
&+\(\frac{\mu_{\alpha}}{r_{\alpha}}\)^{n-2k}\quad\hbox{ for all } y \in B(0,2\mu_{\alpha}^{-1}r_{\alpha}).
\end{align}
This proves \eqref{sym.es15} when $\ell = 0$. Differentiating the representation formula, using \eqref{sym.es11} and \eqref{sym.es14} finally yields \eqref{sym.es15} for all $0 \le \ell \le 2k-1$. 
\end{proof}

\smallskip

\noindent
{\textbf{Step 3:}} Coming back to the definition \eqref{def.I},  we note that $\mathcal{I}_{\alpha}(y)=\mathcal{I}_{\alpha}(|y|)\leq \mathcal{I}_{\alpha}(1)$ when $|y| \ge 1$. Our next result shows that $\|w_{\alpha}\|_{L^{\infty}(B(0,2\mu^{-1}_{\alpha}r_{\alpha}))}$ is controlled by the value of $\mathcal{I}_{\alpha}$ at the scale $1$:
\begin{claim}
We have 
\begin{align}\label{sym.es16}
\|w_{\alpha}\|_{L^{\infty}(B(0,2\mu^{-1}_{\alpha}r_{\alpha}))}\lesssim&\(\frac{\mu_{\alpha}}{r_{\alpha}}\)^{n-2k}+(2^{*}_{k}-p_{\alpha})\notag\\
&+\left\{\begin{aligned}&\mu_{\alpha}^{4}\ln\(1+\frac{r_{\alpha}}{\mu_{\alpha}}\)&&\text{if }n=2k+4\\&\mu_{\alpha}^{4}&&\text{if }n>2k+4.\end{aligned}\right.
\end{align}

\end{claim}
\begin{proof}
We proceed by contradiction and assume that up to a subsequence 
\begin{align}
&\(\frac{\mu_{\alpha}}{r_{\alpha}}\)^{n-2k}+(2^{*}_{k}-p_{\alpha})+\left\{\begin{aligned}&\mu_{\alpha}^{4}\ln\(1+\frac{r_{\alpha}}{\mu_{\alpha}}\)&&\text{if }n=2k+4\\&\mu_{\alpha}^{4}&&\text{if }n>2k+4\end{aligned}\right.\notag\\
&~=\smallo\(\|w_{\alpha}\|_{L^{\infty}(B(0,2\mu^{-1}_{\alpha}r_{\alpha}))}\).
\end{align}
Let $\widetilde{w}_{\alpha}(x)=\|w_{\alpha}\|_{L^{\infty}(B(0,2\mu^{-1}_{\alpha}r_{\alpha}))}^{-1}w_{\alpha}(x)$ for $x\in B(0,2\mu_{\alpha}^{-1}r_{\alpha})$. It follows from \eqref{sym.es14} that
\begin{equation} \label{label.tardif.2}
|\widetilde{w}_{\alpha}(x)|\lesssim U(x)+\smallo(1)\quad\hbox{for } x\in B(0,2\mu_{\alpha}^{-1}r_{\alpha}) \text{ fixed. }
\end{equation}
Using equation \eqref{sym.es10} it follows that $\widetilde{w}_{\alpha}$ satisfies:
$$ \begin{aligned}
P_{\widehat{g}_{\alpha}}\widetilde{w}_{\alpha}
=& \(p_{\alpha}-1\)U^{p_{\alpha}-2}\widetilde{w}_{\alpha} +\bigO\( U^{p_{\alpha}-2-\theta}|\widetilde{w}_{\alpha}| |w_{\alpha}|^{\theta}\) \\
& +\smallo\(U^{2^{*}_{k}-1}|\ln U|\)+\smallo\(\(1+|x|\)^{4-n}\).
\end{aligned} $$
By standard elliptic theory it follows that $\widetilde{w}_{\alpha}\to\widetilde{w}_{\infty}$ in $C^{2k}_{loc}(\R^{n})$ as $\alpha\to+\infty$. From Lemma~\ref{local.1} we know that $w_{\alpha}\to0$ in $C^{2k}_{loc}(\R^{n})$ as $\alpha\to+\infty$, which then shows that  $\widetilde{w}_{\infty}$ satisfies the linearized equation:
\begin{align}
\Delta_{0}^{k}\widetilde{w}_{\infty}=\(2^{*}_{k}-1\)U^{2^{*}_{k}-2}\widetilde{w}_{\infty}\quad\hbox{ in } \R^{n}.
\end{align}
Passing \eqref{sym.es15} to the limit as $\alpha \to + \infty$ shows that 
\begin{align*}
\left|\nabla^{\ell}\widetilde{w}_{\infty}(x)\right|\lesssim(1+|x|)^{2k-n-\ell}\quad\hbox{ for all $0 \le \ell \le 2k-1$ and $x\in\R^{n}$}.
\end{align*}
We may then apply lemma $5.1$ of \cite{LiXiong}, which shows the existence of $a_{0}\in\R$ and $(a_{1},\ldots,a_{n})\in \R^{n}$ such that
\begin{align}
\widetilde{w}_{\infty}(z)=a_{0}\(\frac{n-2k}{2}U+\langle x, \nabla U \rangle\)+\sum\limits_{i=1}^{n}a_{i}\partial_{i} U.
\end{align}
By definition of $B_{\alpha}$ and $w_{\alpha}$, for all $\alpha>1$ we have $\widetilde{w}_{\alpha}(0)=0$ and $\nabla \widetilde{w}_{\alpha}(0)=0$. This implies  $\widetilde{w}_{\infty}(0)=0$ and $\nabla \widetilde{w}_{\infty}(0)=0$,  which gives  $a_{i}=0$ for all $i=0,1, \ldots, n$. Thus $\widetilde{w}_{\infty}\equiv0$ in $\R^{n}$.  Independently, and by definition, we have $w_{\alpha}\not\equiv0$ for all $\alpha\geq1$. Let $z_{\alpha}\in B(0,2\mu_{\alpha}^{-1}r_{\alpha})$ be such that $\left|w_{\alpha}(z_{\alpha})\right|=\|w_{\alpha}\|_{L^{\infty}(B(0,2\mu^{-1}_{\alpha}r_{\alpha}))}$. Then we have by \eqref{label.tardif.2}
\begin{align*}
1=\left|\widetilde{w}_{\alpha}(z_{\alpha})\right|\lesssim(1+|z_{\alpha}|)^{2k-n}+\smallo(1).
\end{align*}
Therefore $z_{\alpha}=\bigO\(1\)$ as $\alpha\to+\infty$, and we may let $z_{\infty}\in \R^{n}$ be such that $z_{\alpha}\to z_{\infty}$ as $\alpha\to+\infty$. This gives us that  $|\widetilde{w}_{\infty}(z_{\infty})|=1$ and hence $\widetilde{w}_{\infty}\not\equiv0$. This is a contradiction and \eqref{sym.es16} follows. 
\end{proof}
\smallskip

\noindent
{\textbf{Step 4:}} we may now obtain optimal estimates on $2^*_k - p_\alpha$: 
\begin{claim}
We have
\begin{align}\label{sym.es18}
&2^{*}_{k}-p_{\alpha}\lesssim\(\frac{\mu_{\alpha}}{r_{\alpha}}\)^{n-2k}+\left\{\begin{aligned}&~\mu_{\alpha}^{4}\ln\(1+\frac{r_{\alpha}}{\mu_{\alpha}}\)&&\text{if }n=2k+4\\&\mu_{\alpha}^{4}&&\text{if }n>2k+4.\end{aligned}\right.\notag\\
\end{align}
\end{claim}
\begin{proof}
Combining \eqref{sym.es16} with \eqref{sym.es15} we have thus far obtained that
\begin{equation} \label{sym.es17}
\begin{aligned}
|w_{\alpha}(x)| & \lesssim\(\frac{\mu_{\alpha}}{r_{\alpha}}\)^{n-2k}+(2^{*}_{k}-p_{\alpha})U(x) \\
&+\left\{\begin{aligned}&\mu_{\alpha}^{4}\ln\(1+\frac{r_{\alpha}}{\mu_{\alpha}}\)&&\text{if }n=2k+4\\&\mu_{\alpha}^{4}\(1+|x|\)^{2k+4-n}&&\text{if }n>2k+4.\end{aligned}\right.\\
\end{aligned}
\end{equation}
for all $x\in B(0, 2\mu_{\alpha}^{-1}r_{\alpha})$. We proceed as in the proof of  \eqref{sym.es16} and assume up to a subsequence that
\begin{align*}
2^{*}_{k}-p_{\alpha}\gg\(\frac{\mu_{\alpha}}{r_{\alpha}}\)^{n-2k}+\left\{\begin{aligned}&\mu_{\alpha}^{4}\ln\(1+\frac{r_{\alpha}}{\mu_{\alpha}}\)&&\text{if }n=2k+4\\&\mu_{\alpha}^{4}&&\text{if }n>2k+4.\end{aligned}\right.
\end{align*}
Let $\widetilde{w}_{\alpha}(x)=\(2^{*}_{k}-p_{\alpha}\)^{-1}w_{\alpha}(x)$ for $x\in B(0,2\mu_{\alpha}^{-1}r_{\alpha})$. It then follows from \eqref{sym.es17} that $|\widetilde{w}_{\alpha}(x)|\lesssim U(x)+\smallo(1)$ for $x\in B(0,2\mu_{\alpha}^{-1}r_{\alpha})$, and \eqref{sym.es10} gives that $\widetilde{w}_{\alpha}$ satisfies:
\begin{align*}
P_{\widehat{g}_{\alpha}}\widetilde{w}_{\alpha}=&  \(p_{\alpha}-1\)U^{p_{\alpha}-2}\widetilde{w}_{\alpha} +\(\ln\widehat{f}_{\alpha}+\ln U\)U^{2^{*}_{k}-1}+\smallo(1).
\end{align*}
By standard elliptic theory it follows that $\widetilde{w}_{\alpha}\to\widetilde{w}_{\infty}$ in $C^{2k}_{loc}(\R^{n})$ as $\alpha\to+\infty$ where 
\begin{align} \label{label.tardif.3}
\Delta_{0}^{k}\widetilde{w}_{\infty}=\(2^{*}_{k}-1\)U^{2^{*}_{k}-2}\widetilde{w}_{\infty}+\(\ln f(x_{\infty})+\ln U\) U^{2^{*}_{k}-1}\quad\hbox{ in } \R^{n},
\end{align}
with $x_{\infty}:=\lim \limits_{\alpha\to+\infty}x_{\alpha}\in M$. Since $\(\ln f(x_{\infty})+\ln U\) U^{2^{*}_{k}-1}$ is non-constant it follows that $\widetilde{w}_{\infty}\not\equiv0$. Using \eqref{sym.es15} we again have  
\begin{align*}
\left|\nabla^{\ell}\widetilde{w}_{\infty}(x)\right|\lesssim(1+|x|)^{2k-n-\ell}\quad\hbox{ for all $\ell\geq1$ and $x\in\R^{n}$}.
\end{align*}
Let $\ds{\mathcal{Z}_{0}:=\frac{n-2k}{2}U+\langle x, \nabla U\rangle}$. Then $\Delta_{0}^{k}\mathcal{Z}_{0}=\(2^{*}_{k}-1\)U^{2^{*}_{k}-2}\mathcal{Z}_{0}$, which we multiply by $\widetilde{w}_{\infty}$ and integrate by parts. This gives with \eqref{label.tardif.3}
\begin{align} \label{label.tardif.4}
0=\ln f(x_{\infty})\int_{\R^{n}}U^{2^{*}_{k}-1}\mathcal{Z}_{0}\,dx+\int_{\R^{n}} U^{2^{*}_{k}-1}\ln U\,\mathcal{Z}_{0}\,dx.
\end{align}
Denote $\tilde{U}_{\lambda}:=\lambda^{\frac{2k-n}{2}}U\(x/\lambda\)$. We have $U=\tilde{U}_{1}$ and $\ds{\mathcal{Z}_{0}=-\frac{d}{d\lambda}\tilde{U}_{\lambda}\mid_{\lambda=1}}$. Differentiating $\ds \int_{\R^{n}}\tilde{U}_{\lambda}^{2^{*}_{k}}\,dx\equiv\hbox{const.}$ with respect to $\lambda$ at $\lambda = 1$ gives $\ds\int_{\R^{n}}U^{2^{*}_{k}-1}\mathcal{Z}_{0}\,dx=0$.  But this implies  with \eqref{label.tardif.4} that $\ds\int_{\R^{n}} U^{2^{*}_{k}-1}\ln U\,\mathcal{Z}_{0}\,dx=0$, which is a contradiction since 
\begin{align*}
&\int_{\R^{n}} U^{2^{*}_{k}-1}\ln U~\mathcal{Z}_{0}\,dx= \frac{n-2k}{2}\int_{\R^{n}} U^{2^{*}_{k}-1}\ln U \frac{1-\mathfrak{c}_{n,k}^{-1}\left|x\right|^2}{\(1+\mathfrak{c}_{n,k}^{-1}\left|x\right|^2\)^{\frac{n-2k}{2}+1}}\,dx\notag\\
&=\(\frac{n-2k}{2}\)^{\frac{n-2k}{2}}\omega_{n-1}\int_{0}^{+\infty}\frac{\mathfrak{c}_{n,k}^{-1}r^{2}-1}{\(1+\mathfrak{c}_{n,k}^{-1}r^2\)^{n+1}} \ln\(1+\mathfrak{c}_{n,k}^{-1}r^2\)r^{n-1}~dr<0,
\end{align*}
and where the strict inequality follows from  \cite[Equation $(5.13)$]{Marques}.
\end{proof}

\noindent
{\textbf{Step 5:}} Combining the estimates \eqref{sym.es15},
\eqref{sym.es16} and \eqref{sym.es18} shows that, for any $0 \le \ell \le 2k-1$ and for any  $x\in B(0, 2\mu_{\alpha}^{-1}r_{\alpha})$, 
\begin{align}\label{sym.es19b}
\left|\nabla^{\ell} w_{\alpha}(x)\right|&\lesssim~
\(\frac{\mu_{\alpha}}{r_{\alpha}}\)^{n-2k}\(1+|x|\)^{-\ell}\notag\\
&+\left\{\begin{aligned}&\mu_{\alpha}^{4}\ln\(1+\frac{r_{\alpha}}{\mu_{\alpha}}\)\(1+|x|\)^{-\ell}&&\text{if }n=2k+4\\&\mu_{\alpha}^{4}\(1+|x|\)^{2k+4-n-\ell}&&\text{if }n>2k+4.\end{aligned}\right. 
\end{align}
holds. Scaling back in the original variables proves \eqref{sym.esII}. We thus only have to prove estimate \eqref{sym.esIIb} when $n=2k+4$. Keeping the notations of Step 2 and using \eqref{sym.es19b} we now obtain that 
\begin{align*}
&\int_{B(0,2\mu_{\alpha}^{-1}r_{\alpha})}|y_{\alpha}-y|^{-4} B_{0}^{p_{\alpha}-2}|w_{\alpha}|\,dy 
&\lesssim\(\frac{\mu_{\alpha}}{r_{\alpha}}\)^{4}+\mu_{\alpha}^{4}\ln\(1+\frac{r_{\alpha}}{\mu_{\alpha}}\)\frac{1}{1+|y_{\alpha}|}.
\end{align*}
Plugging the latter in \eqref{sym.es13} and using \eqref{def.I} and \eqref{sym.es16} shows that
for all $y\in B(0,2\mu_{\alpha}^{-1}r_{\alpha})$
\begin{align}
|w_{\alpha}(y)|\lesssim&\(\frac{\mu_{\alpha}}{r_{\alpha}}\)^{4}+\(r_{\alpha}+\mu_{\alpha}\ln\(1+\frac{r_{\alpha}}{\mu_{\alpha}}\)\)\frac{\mu_{\alpha}^{3}}{1+|y|}.
\end{align}
The estimates on the derivatives follow similarly, using again \eqref{sym.es19b}. Scaling back to the original variables proves \eqref{sym.esIIb}. This completes the proof of proposition~\ref{sym.es}.
\end{proof}

\begin{remark}
In Proposition \ref{sym.es} we only prove symmetric estimates at fourth-order, which are analogous to the estimates obtained in \cite{Marques}. This is because in Proposition \ref{GJMS.exp0} we only obtain a fourth-order expansion of $P_{(\exp_{\xi}^{g_\xi})^*g_{\xi}} - \Delta_0^k$ around any point $\xi \in M$ (where $g_{\xi}$ satisfies \eqref{conf.1}, \eqref{conf.2}, \eqref{conf.3}). In principle, if an explicit expansion of $P_{(\exp_{\xi}^{g_\xi})^*g_{\xi}} - \Delta_0^k$ in terms of the Taylor expansion of $g_{\xi}$ at $\xi$ at any order was available, we could obtain symmetric estimates of any order in the spirit of those obtained in \cite{KhuMaSc} (when $k=1$) and \cite{GongKimWei} (when $k=2,3$), since the strategy of proof remains essentially unchanged. An expansion of $P_{(\exp_{\xi}^{g_\xi})^*g_{\xi}} - \Delta_0^k$ to any order is not known yet, and 
this is to this day the main obstacle to improving the estimates of Proposition \ref{sym.es} for any $k$ with this approach. 
\end{remark}

\noindent
\begin{remark}\label{sym.es2k}
The pointwise estimates of Proposition \ref{sym.es} remain true for derivatives of order $\ell \ge 2k$. This follows from standard elliptic theory and the expressions one obtains are the same as in \eqref{sym.esII}. 
\end{remark}
\smallskip

\section{Estimates on the Weyl Curvature at a concentration point}\label{sec.estimate.weyl}

We keep in this section the notations of Section \ref{sec.order.two}. The next result estimates the value of the Weyl curvature at a concentration point of a blowing-up sequence in dimensions $n \ge 2k+4$:
\begin{lemma}\label{weyl.es}
Assume that $n \ge 2k+4$. Let $(\va)_\alpha$ be a sequence of positive solutions of \eqref{eq:bis:0} and let $(\xa)_\alpha$ and $(\rhoa)_\alpha$ satisfy \eqref{cond.xarhoa.va} and \eqref{cond.xarhoa.va2}. Assume that \eqref{positivity} holds and let $\ra$ be defined by \eqref{def.ra}. We have
\begin{equation} \label{eq.estimate.Weyl}
\begin{aligned}
|\Weyl_{g}(x_{\alpha})|_g^{2}\lesssim\smallo(1)+\left\{\begin{aligned}&\(\ln\frac{1}{\mu_{\alpha}}\)^{-1}r_{\alpha}^{-4}&&\text{if }n=2k+4\\&\mu_{\alpha}^{n-2k-4}r_{\alpha}^{2k-n}&&\text{if }n>2k+4\end{aligned}\right.
\end{aligned}
\end{equation}
as $\alpha \to + \infty$. As a consequence, if $n \ge 2k+4$ and $\min_M | \Weyl_{g}| >0$, then $\ra \to 0$ as $\alpha \to + \infty$. 
\end{lemma}

\begin{remark}
Weyl-vanishing type results such as Lemma \ref{weyl.es} have been obtained, in the last years, for many conformally-covariant geometric equations such as  \eqref{eq:one} and \eqref{eq:bis:0}. We mention for instance  
the fractional Yamabe problem (see Kim--Musso--Wei \cite{KimMussoWeiYF}, Jin--Li--Xiong \cite{JinLiXiong1} and Qing--Raske \cite{QingRaske} and the references therein) or sign-changing solutions for the Yamabe equation, which are related to eigenvalue-optimisation results for the conformal laplacian, see Premoselli--V\'etois \cite{PremoselliVetois4}. 
\end{remark}

\begin{proof}
Fix $\delta>0$. We proceed as in the proof of \eqref{sym.es4} and write $\widetilde{v}_{\alpha}:=v_{\alpha}\circ\exp_{x_{\alpha}}^{g_{\alpha}}$ and $\tga = (\exp_{\xa}^{\ga})^* \ga$. The Pohozaev identity \eqref{poho.id0} for $\widetilde{v}_{\alpha}$ in $B(0,\delta r_{\alpha})$ gives:
\begin{align}\label{weyl.es1}
&\mathcal{P}_{k}(\delta r_{\alpha};\widetilde{v}_{\alpha})=\int_{B(0,\delta r_{\alpha})}\(\frac{n-2k}{2}\widetilde{v}_{\alpha}+x^{i}\partial_{i}\widetilde{v}_{\alpha}\)\(\Delta_{0}^{k}\widetilde{v}_{\alpha}-\tilde{f}_{\alpha}^{p_{\alpha}-2^{*}_{k}}\,\widetilde{v}_{\alpha}^{\,p_\alpha-1}\)\,dx\,+\notag\\
&\(\frac{n-2k}{2}-\frac{n}{p_{\alpha}}\)\int_{B(0,\delta r_{\alpha})}\tilde{f}_{\alpha}^{p_{\alpha}-2^{*}_{k}}\,\widetilde{v}_{\alpha}^{\,p_{\alpha}}\,dx-\frac{1}{p_{\alpha}}\int_{B(0,\delta r_{\alpha})}x^{i}\partial_{i}\tilde{f}_{\alpha}^{p_{\alpha}-2^{*}_{k}}\,\widetilde{v}_{\alpha}^{\,p_{\alpha}}\,dx\notag\\
&~+\frac{\delta r_{\alpha}}{p_{\alpha}}\int_{\partial B(0,\delta r_{\alpha})}\tilde{f}_{\alpha}^{p_{\alpha}-2^{*}_{k}}\,\widetilde{v}_{\alpha}^{\,p_{\alpha}}\,d\sigma.
\end{align}
Here $\mathcal{P}_{k}(\delta r_{\alpha};\widetilde{v}_{\alpha})$ denotes the boundary terms given by \eqref{poho.id01} below. Using again  \eqref{local.est1} and the dominated convergence theorem, as in the proof of \eqref{sym.es4}, we have
\begin{align}\label{weyl.es2}
&\(\frac{n-2k}{2}-\frac{n}{p_{\alpha}}\)\int_{B(0,\delta r_{\alpha})}\tilde{f}_{\alpha}^{\,p_{\alpha}-2^{*}_{k}}\,\widetilde{v}_{\alpha}^{\,p_{\alpha}}\,dx-\frac{1}{p_{\alpha}}\int_{B(0,\delta r_{\alpha})}x^{i}\partial_{i}\tilde{f}_{\alpha}^{p_{\alpha}-2^{*}_{k}}\,\widetilde{v}_{\alpha}^{\,p_{\alpha}}\,dx\notag\\
&~+\frac{r_{\alpha}}{p_{\alpha}}\int_{\partial B(0,\delta r_{\alpha})}\tilde{f}_{\alpha}^{p_{\alpha}-2^{*}_{k}}\,\widetilde{v}_{\alpha}^{\,p_{\alpha}}\,d\sigma = \bigO\(\mu^{n-2k-\frac{4k}{p_{\alpha}-2}}\(\frac{\mu_{\alpha}}{\delta r_{\alpha}}\)^{n-1}\) \\
&+\frac{(n-2k)^{2}}{4n}(p_{\alpha}-2^{*}_{k})\|U\|^{2^{*}_{k}}_{L^{2^{*}_{k}}(\R^{n})}\mu_{\alpha}^{n-2k-\frac{4k}{p_{\alpha}-2}}(1+\smallo(1)).\notag
\end{align}
Independently, arguing as in \eqref{est.boundary.terms} we obtain
\begin{equation} \label{weyl.es22}
\left|\mathcal{P}_{k}(\delta r_{\alpha};\widetilde{v}_{\alpha})\right|\lesssim\mu^{n-2k-\frac{4k}{p_{\alpha}-2}}\(\dfrac{\mu_{\alpha}}{\delta r_{\alpha}}\)^{n-2k}.
\end{equation}
We now estimate the remaining term in \eqref{weyl.es1}. By bilinearity, and since $P_{\tilde{g}_\alpha} \tilde{v}_\alpha = \tilde{f}_\alpha^{p_\alpha - 2^*_k} \tilde{v}_\alpha^{p_\alpha-1}$, we have
\begin{equation} \label{integrals.123}
\begin{aligned}
\int_{B(0,\delta r_{\alpha})}\(\frac{n-2k}{2}\widetilde{v}_{\alpha}+x^{i}\partial_{i}\widetilde{v}_{\alpha}\)( \tilde{f}_\alpha^{p_\alpha - 2^*_k} \tilde{v}_\alpha^{p_\alpha-1}-\Delta_{0}^{k}\widetilde{v}_{\alpha})\,dx \\
= I_{\alpha}+II_{\alpha}+ III_\alpha,
\end{aligned}
\end{equation}
where we have let 
\begin{align*}
I_\alpha &=\int_{B(0,\delta r_{\alpha})}\(\frac{n-2k}{2}\tilde{B}_{\alpha}+x^{i}\partial_{i}\tilde{B}_{\alpha}\)(P_{\tga}\tilde{B}_{\alpha}-\Delta_{0}^{k}\tilde{B}_{\alpha})\,dx,\\
II_\alpha &= \int_{B(0,\delta r_{\alpha})}\(\frac{n-2k}{2}\widetilde{v}_{\alpha}+x^{i}\partial_{i}\widetilde{v}_{\alpha}\)\(P_{\tga}(\,\widetilde{v}_{\alpha}-\tilde{B}_{\alpha})-\Delta_{0}^{k}(\,\widetilde{v}_{\alpha}-\tilde{B}_{\alpha})\)\,dx,\\
III_\alpha&= \int_{B(0,\delta r_{\alpha})}\(\frac{n-2k}{2}(\,\widetilde{v}_{\alpha}- \tilde{B}_{\alpha})+x^{i}\partial_{i}(\,\widetilde{v}_{\alpha}-\tilde{B}_{\alpha})\)(P_{\tga} \tilde{B}_{\alpha}-\Delta_{0}^{k}\tilde{B}_{\alpha})\,dx,
\end{align*}
where we have let 
\begin{equation*} 
\tilde{B}_\alpha(x) := \Ba \circ \exp_{x_{\alpha}}^{g_{\alpha}}(x) = \frac{\mu_{\alpha}^{n-2k-\frac{2k}{p_{\alpha}-2}}}{\left(\mu^{2}_{\alpha}+\mathfrak{c}^{-1}_{n,k}~|x|^2\right)^{\frac{n-2k}{2}}} 
\end{equation*}
 and $\Ba$ is as in \eqref{bubble2}. We first estimate the integrals  $II_{\alpha}$ and $III_{\alpha}$.  It follows from \eqref{local.est1} that $\tva\lesssim \tilde{B}_{\alpha}$ and $\langle x, \nabla\,\widetilde{v}_{\alpha}\rangle\lesssim \tilde{B}_{\alpha}$ in $B(0,r_{\alpha})$. Using \eqref{expansion.Pg} and Proposition \ref{sym.es}, direct computations give 
 \begin{equation}\label{weyl.es3}
\begin{aligned}
| II_{\alpha}|
&\lesssim\mu_{\alpha}^{n-2k-\frac{4k}{p_{\alpha}-2}}\delta^{2}r_{\alpha}^{2}\(\frac{\mu_{\alpha}}{r_{\alpha}}\)^{n-2k} \\
&+\left\{\begin{aligned}&\mu_{\alpha}^{4-\frac{4k}{p_{\alpha}-2}}\delta r_{\alpha}\mu_{\alpha}^{4}&&\text{if }n=2k+4\\&\mu_{\alpha}^{n-2k-\frac{4k}{p_{\alpha}-2}}\delta r_{\alpha}\mu_{\alpha}^{5}&&\text{if }n>2k+4\end{aligned}\right. .
\end{aligned}
\end{equation}
Next, using \eqref{sym.es5} together with Proposition \ref{sym.es} we have
\begin{equation} \label{weyl.es4}
\begin{aligned}
|III_{\alpha}|
&\lesssim\mu_{\alpha}^{n-2k-\frac{4k}{p_{\alpha}-2}}\delta^{4}r_{\alpha}^{4}\(\frac{\mu_{\alpha}}{r_{\alpha}}\)^{n-2k} \\
&+\left\{\begin{aligned}&\mu_{\alpha}^{4-\frac{4k}{p_{\alpha}-2}}\delta^{3}r_{\alpha}^{3}\mu_{\alpha}^{4}&&\text{if }n=2k+4\\&\mu_{\alpha}^{n-2k-\frac{4k}{p_{\alpha}-2}}\delta r_{\alpha}\mu_{\alpha}^{5}&&\text{if }n>2k+4 \end{aligned}\right. . 
\end{aligned}
\end{equation}
We now estimate $I_\alpha$. For $x \in M$ we define
$$ W_\alpha\(x\) = \chi\(\frac{d_{\ga}\(\xa, x\)}{\delta}\) U_{\xa, \ma}\(x\) ,$$ 
where $U_{\xi, \mu}\(x\)$ is as in \eqref{bubble3} below and  where  $\chi:[0,+\infty)\to[0,1]$ is a smooth cutoff function such that $\chi\equiv1$ in $\[0,\iota_{0}/2\]$ and $\chi\equiv0$ in $\[\iota_{0},+\infty\)$, with $\iota_{0}:=\inf\limits_{\xi\in M}\iota(M,g_{\xi})$. Taking $\rho_{\alpha}<\iota_{0}/4$ if necessary, we can assume that $\chi\equiv1$ in $B_{g_{\alpha}}(x_{\alpha},2r_{\alpha})$, so that  $W_{\alpha}=\mu\,^{\frac{2k}{p_{\alpha}-2}-\frac{n-2k}{2}}B_{\alpha}$ in $B_{g_{\alpha}}(x_{\alpha},2r_{\alpha})$. We let $\tilde{W}_\alpha = W_\alpha \circ \exp_{x_{\alpha}}^{g_{\alpha}}$, so that 
$$ \tilde{W}_\alpha(x) = \chi \( \frac{|x|}{\delta} \) \frac{\ma^{\frac{n-2k}{2}}}{\( \ma^2 + \mathfrak{c}_{n,k}^{-1} |x|^2 \)^{\frac{n-2k}{2}}} \quad \text{ for } x \in \mathbb{R}^n.$$ 
We thus have 
\begin{equation} \label{weyl.es51}
I_\alpha = \ma^{n-2k -\frac{4k}{p_{\alpha}-2} } \int_{B(0,\delta r_{\alpha})}\(\frac{n-2k}{2}\widetilde{W}_{\alpha}+x^{i}\partial_{i}\widetilde{W}_{\alpha}\)(P_{\tga}\widetilde{W}_{\alpha}-\Delta_{0}^{k}\,\widetilde{W}_{\alpha})\,dx.
\end{equation}
We now estimate the integral in the right-hand side of \eqref{weyl.es51}. We use the same arguments than in the proof of Proposition \ref{GJMS.exp2}.
It is easily seen that 
$$\mu_{\alpha}\frac{\partial}{\partial \ma} \widetilde{W}_{\alpha}(x)=-\(\dfrac{n-2k}{2}\,\widetilde{W}_{\alpha}+x^{i}\partial_{i}\widetilde{W}_{\alpha}\)$$
for any $x \in B(0, \iota_0)$. Since $U$ satisfies \eqref{eq.bubble1} we have 
$$ \Delta_0^k \widetilde{W}_{\alpha} = \widetilde{W}_{\alpha}^{2^*_k-1} + \bigO \( \delta^{-2k} \ma^{\frac{n-2k}{2}} \mathds{1}_{\frac{\iota_0}{2} \delta \le |x| \le \iota_0 \delta} \).$$
Independently, we have $\int_{B(0, \iota_0)} \widetilde{W}_{\alpha}^{2^*_k} \, dx =\|U\|^{2^{*}_{k}}_{L^{2^{*}_{k}}(\R^{n})} + \bigO(\ma^n)$ and this equality can be differentiated in $\ma$. Combining the latter two equations we obtain  
$$ \begin{aligned} \int_{B(0, \iota_0)} \(\dfrac{n-2k}{2}\,\widetilde{W}_{\alpha}+x^{i}\partial_{i}\widetilde{W}_{\alpha}\) \Delta_0^k \widetilde{W}_{\alpha} \, dx = \bigO \( \delta \ma^{n-2k} \),  \\
\end {aligned} $$ 
As a consequence, using the self-adjointness of $P_{\ga}$ and \eqref{conf.2} we have 
\begin{equation} \label{weyl.es23}
\begin{aligned}
&- \frac{n-2k}{4}\mu_{\alpha}\frac{d}{d \ma}\(\int_M W_\alpha P_{\ga} W_\alpha dv_{\ga} \)  \\
&= \int_{B(0, \iota_0)} \(\frac{n-2k}{2}\widetilde{W}_{\alpha}+x^{i}\partial_{i}\widetilde{W}_{\alpha}\)(P_{\tga}\widetilde{W}_{\alpha}-\Delta_{0}^{k}\,\widetilde{W}_{\alpha})\,dx + \bigO \( \delta \ma^{n-2k} \) \\
& =  \int_{B(0,\delta r_{\alpha})}\(\frac{n-2k}{2}\widetilde{W}_{\alpha}+x^{i}\partial_{i}\widetilde{W}_{\alpha}\)(P_{\tga}\widetilde{W}_{\alpha}-\Delta_{0}^{k}\,\widetilde{W}_{\alpha})\,dx + \bigO \( \delta \ma^{n-2k} \)  \\
& +\left\{\begin{aligned}&\bigO\(r_{\alpha}\(\frac{\mu_{\alpha}}{r_{\alpha}}\)^{n-2k}\)&&\text{if }n\leq2k+4\\&\bigO\(\mu_{\alpha}^{4}\(\frac{\mu_{\alpha}}{\delta r_{\alpha}}\)^{n-2-4k}\)&&\text{if }n>2k+4.\end{aligned} \right. .
\end{aligned} 
\end{equation}
Let now, for $u\in C^{2k}(M), u\not\equiv0$:
\begin{align*}
I_{k,g}(u):=\frac{\displaystyle{\int_{M}u\,P_{g}u\,dv_{g}}}{\(\displaystyle{\int_{M}|u|^{\,2^{*}_{k}}\,dv_{g}}\)^{\frac{n-2k}{n}}}. 
\end{align*}
One has (see \cite{MazumdarVetois}) 
\begin{align*}
\int_{M}W_{\alpha}P_{g_{\alpha}}W_{\alpha}\,dv_{g_{\alpha}}&=\(\int_{M}W_{\alpha}^{\,2^{*}_{k}}\,dv_{g_{\alpha}}\)^{\frac{n-2k}{n}}I_{k,g_{\alpha}}(W_{\alpha})\\
&=\(\|U\|^{2^{*}_{k}}_{L^{2^{*}_{k}}(\R^{n})}+\bigO(\mu_\alpha^{n})\)I_{k,g_{\alpha}}(W_{\alpha}),
\end{align*}
and all the expressions are $C^{1}$ with respect to the parameter $\mu_\alpha$. Together with \eqref{weyl.es23} we thus obtain
\begin{equation} \label{weyl.es53}
\begin{aligned}
&\int_{B(0,\delta r_{\alpha})}\(\frac{n-2k}{2}\widetilde{W}_{\alpha}+x^{i}\partial_{i}\widetilde{W}_{\alpha}\)(P_{\tga}-\Delta_{0}^{k})\widetilde{W}_{\alpha}\,dx\\
&=- \frac{n-2k}{4}\|U\|^{2^{*}_{k}}_{L^{2^{*}_{k}}(\R^{n})}\mu_{\alpha}\frac{d}{d\mu_{\alpha}}I_{k,g_{\alpha}}(U_{\alpha})+\bigO(\delta \mu_{\alpha}^{n-2k})\\
&\qquad+\left\{\begin{aligned}&\bigO\( r_{\alpha}\(\frac{\mu_{\alpha}}{r_{\alpha}}\)^{n-2k}\)&&\text{if }n\leq2k+4\\&\bigO\(\mu_{\alpha}^{4}\(\frac{\mu_{\alpha}}{\delta r_{\alpha}}\)^{n-2k-4}\)&&\text{if }n>2k+4 \end{aligned}\right. .
\end{aligned}
\end{equation}
The following expansion of $I_{k,\ga}(W_\alpha)$ was recently proven in \cite{MazumdarVetois}: we have 
\begin{align*}
I_{k,g_{\alpha}}(W_{\alpha})=&\Vert U \Vert_{L^{2^*_k}(\R^n)}^{\frac{4k}{n-2k}}\\
&-\,\mathcal{C}(n,k)\times\left\{\begin{aligned}&\bigO(\mu_{\alpha}^{n-2k})&&\text{if }n<2k+4\\&|\Weyl_{g}(x_{\alpha})|_g^{2}\mu_{\alpha}^{4}\ln(1/\mu_{\alpha})+\bigO(\mu_{\alpha}^{4})&&\text{if }n=2k+4\\&|\Weyl_{g}(x_{\alpha})|_g^{2}\mu_{\alpha}^{4}+\smallo(\mu_{\alpha}^{4})&&\text{if }n>2k+4,\end{aligned}\right.
\end{align*}
for some positive constant $\mathcal{C}(n,k)$. Combining the latter with \eqref{weyl.es51} and \eqref{weyl.es53} finally shows that there exists a positive constant $C = C(n,k)$ such that  
\begin{equation} \label{weyl.es6}
\begin{aligned}
I_\alpha &=\ma^{n-2k -\frac{4k}{p_{\alpha}-2} } C \left\{\begin{aligned}&\bigO(\mu_{\alpha}^{n-2k})&&\text{if }n<2k+4\\&|\Weyl_{g}(x_{\alpha})|_g^{2}\mu_{\alpha}^{4}\ln(1/\mu_{\alpha})+\bigO(\mu_{\alpha}^{4})&&\text{if }n=2k+4\\&|\Weyl_{g}(x_{\alpha})|_g^{2}\mu_{\alpha}^{4}+\bigO(\mu_{\alpha}^{5})&&\text{if }n>2k+4,\end{aligned}\right. \\
&+\left\{\begin{aligned}&\bigO\(\mu_{\alpha}^{n-2k-\frac{4k}{p_{\alpha}-2}} r_{\alpha}\(\frac{\mu_{\alpha}}{r_{\alpha}}\)^{n-2k}\)&&\text{if }n\leq2k+4\\&\bigO\(\mu_{\alpha}^{n-2k-\frac{4k}{p_{\alpha}-2}}\mu_{\alpha}^{4}\(\frac{\mu_{\alpha}}{\delta r_{\alpha}}\)^{n-2k-4}\)&&\text{if }n>2k+4.\end{aligned}\right. . 
\end{aligned}
\end{equation}
Combining \eqref{weyl.es2},  \eqref{weyl.es22}, \eqref{integrals.123}, \eqref{weyl.es3}, \eqref{weyl.es4} and \eqref{weyl.es6} in \eqref{weyl.es1}, and using that $p_\alpha \le 2^*_k$ proves \eqref{eq.estimate.Weyl}. That $\ra \to 0$ if $\min_M|\Weyl_g|_g >0$ easily follows from \eqref{eq.estimate.Weyl}. 
\end{proof}
An important consequence of the Pohozaev identity and of \eqref{eq.estimate.Weyl} is the following lemma which shows that, under the assumptions of Theorem~\ref{theo.main}, the radius of influence $\ra$ (see \eqref{def.ra}) of a concentration point $\xa$ satisfying \eqref{cond.xarhoa.va} and \eqref{cond.xarhoa.va2} is equal to $\rhoa$ (provided $\ra$ goes to $0$ in small dimensions). This will be the key ingredient in showing that concentration points are isolated (see Proposition~\ref{prop.isolated.points} in the next section). 

\begin{lemma} \label{radius.influence}
Let $(\va)_\alpha$ be a sequence of positive solutions of \eqref{eq:bis:0} and let $(\xa)_\alpha$ and $(\rhoa)_\alpha$ satisfy \eqref{cond.xarhoa.va} and \eqref{cond.xarhoa.va2}. We assume that \eqref{positivity} holds and let $\ra$ be defined by \eqref{def.ra}. Assume that
\begin{itemize}
\item either  $2k+1\leq n\leq2k+5$ and $r_{\alpha}\to0$ as $\alpha\to+\infty$
\item or $n \ge 2k+4$ and $\min_M | \Weyl_{g}| >0$ 
\end{itemize}
Then $\rhoa \to 0$ as $\alpha \to + \infty$ and $r_{\alpha}=\rho_{\alpha}$ up to a subsequence. 
\end{lemma}

\begin{proof}
By Lemma~\ref{weyl.es} we have $\ra \to 0$ as $\alpha \to + \infty$ if $n \ge 2k+4$ and $\min_M |\Weyl_g|_g >0$. Thus we may assume that $\ra = \smallo(1)$ as $\alpha \to + \infty$ for every $n \ge 2k+1$. Let $\delta >0$ be fixed. We use again the Pohozaev identity \eqref{weyl.es1}. Since $\ra \to 0$, since $p_\alpha \le 2^*_k$ and since $n \le 2k+5$,  estimates \eqref{weyl.es2}, \eqref{weyl.es3}, \eqref{weyl.es4} and \eqref{weyl.es6}  show that 
\begin{equation} \label{lem.poho.1} \mathcal{P}_{k}(\delta r_{\alpha};\widetilde{v}_{\alpha}) \le \smallo \( \mu_{\alpha}^{n-2k-\frac{4k}{p_{\alpha}-2}} \( \frac{\ma}{\ra} \)^{n-2k} \) 
\end{equation}
as $\alpha \to + \infty$. Using Proposition~\ref{local.3} we have 
\begin{equation} \label{lem.poho.2}
 \mathcal{P}_{k}(\delta r_{\alpha};\widetilde{v}_{\alpha})=\mu_{\alpha}^{n-2k-\frac{4k}{p_{\alpha}-2}}\(\frac{\mu_{\alpha}}{r_{\alpha}}\)^{n-2k}\Big[ \mathcal{P}_{k}(\delta ;\widetilde{u}) + o(1) \Big],
 \end{equation}
where $\widetilde{u} =\frac{\mathfrak{c}_{n,k}^{\frac{n-2k}{2}}}{|x|^{n-2k}} + H$ and $H$ satisfies $\Delta_0^k H = 0$ in $B(0,2)$. Recall that the expression of $ \mathcal{P}_{k}(\delta ;\widetilde{u}) $ is given by Proposition~\ref{prop.poho.appendix} below. Combining \eqref{lem.poho.1} and \eqref{lem.poho.2} and passing to the limit as $\alpha \to +\infty$ thus shows that $ \mathcal{P}_{k}(\delta ;\widetilde{u}) \le 0$ for all $\delta >0$. Using Lemma~\ref{sing+har} below we get $ \lim_{\delta \to 0} \mathcal{P}_{k}(\delta ;\widetilde{u})  = \Lambda H(0)$ for some $\Lambda =\Lambda(n,k) >0$, which implies that $H(0) \le 0$. Proposition~\ref{local.3} then shows that up to a subsequence $\ra = \rhoa$ holds. 
\end{proof}

Coming back to the definition of radius of influence $r_{\alpha}$ in \eqref{def.ra}, we can apply a diagonal argument when $\varepsilon$ is replaced by a sequence that goes to $0$ as $\alpha \to + \infty$. We have hence obtained the following proposition: 

\begin{proposition} \label{prop.order.1}
Let $(\va)_\alpha$ be a sequence of positive solutions of \eqref{eq:bis:0} and let $(\xa)_\alpha$ and $(\rhoa)_\alpha$ satisfy \eqref{cond.xarhoa.va} and \eqref{cond.xarhoa.va2}. Assume that \eqref{positivity} holds and let $\ra$ be defined by \eqref{def.ra}. Assume 
\begin{itemize}
\item either that 
$r_{\alpha}\to0$ as $\alpha\to+\infty$ and $2k+1\leq n\leq2k+5$, or
\item
$n\ge2k+4$ and $\min \limits_{M}|\Weyl_{g}|>0$.
\end{itemize}
Then $r_{\alpha}=\rho_{\alpha}$ up to a subsequence, $\rhoa \to 0$ as $\alpha\to+\infty$, and
\begin{align}
\max\limits_{B_{g}(x_{\alpha},\rho_{\alpha})}\left|\frac{v_{\alpha}}{B_{\alpha}}-1\right|=\smallo(1)~\hbox{ as }\alpha\to+\infty
\end{align}
where $\Ba$ is as in \eqref{bubble2}.
\end{proposition}
\smallskip

\section{The Compactness result}\label{sec.final.compactness}

We now present the proof of  Theorem \ref{theo.main}. As a first step, we prove that concentration points of a blowing-up sequence of positive solutions of \eqref{eq:one} are isolated. If $(\ua)_\alpha$ is a sequence of positive solutions of \eqref{eq:one} satisfying \eqref{blowup}, we let $(x_{1,\alpha} , \dotsc, x_{N_\alpha, \alpha})$ be the points constructed in Proposition~\ref{weak.es} and let 
 \begin{equation} \label{def.da}
d_\alpha =  \left \{
\begin{aligned} 
& \frac{1}{16} \min_{1 \le i \neq j \le N_\alpha} d_g(x_{i,\alpha}, x_{j,\alpha}) & \text{ if } N_\alpha \ge 2 \\
& \frac{1}{16} \iota(M, g)  & \text{ if } N_\alpha = 1,
\end{aligned}  \right. 
\end{equation}
where $\iota(M, g)$ is the injectivity radius of $(M,g)$.

\begin{proposition} \label{prop.isolated.points}
Let $(\ua)_\alpha$ be a sequence of positive solutions of \eqref{eq:one} satisfying \eqref{blowup} and assume that \eqref{positivity} is satisfied. Assume that either $2k+1 \le n \le 2k+5$ or $n \ge 2k+4$ and $\min_M |\Weyl_g|_g >0$. Let $\da$ be as in \eqref{def.da}. Then $\limsup_{\alpha \to + \infty} d_\alpha >0$.
\end{proposition}

\begin{proof}
We proceed by contradiction and assume that $\da \to 0$ as $\alpha \to + \infty$. This clearly implies that $N_\alpha \ge 2$. We assume that the concentration points are ordered so that
\begin{equation} \label{order}
 d_g(x_{1,\alpha}, x_{2,\alpha}) \le \dotsc \le d_g(x_{1,\alpha}, x_{N_\alpha,\alpha}). 
 \end{equation}
A simple remark, that follows from the definition of $\da$, is that for any $i \in \{1, \dots, N_\alpha\}$, the sequences $(x_{i,\alpha})_\alpha$ and $(\da)_\alpha$ satisfy \eqref{cond.xarhoa} with $\rho_\alpha = d_\alpha$.  Let $R \ge 1$ and define, for any $\alpha$, $N_{\alpha,R}$ by
$$ 1 \le i \le N_{\alpha,R} \iff d_g(x_{1,\alpha}, x_{i,\alpha}) \le R \da, $$
which is well-defined in view of \eqref{order}. Clearly $N_{\alpha,R} \ge 2$ for any $R \ge 16$ and by definition of $\da$, for a fixed $R$, $N_{\alpha,R}$ is uniformly bounded in $\alpha$. In what follows, we will  fix $R \ge 16$ and, up to a subsequence, we will therefore assume that $N_{\alpha,R}$ is constant and equal to $N_R \ge 2$. For any $1 \le i \le N_R$ fixed, two alternatives can occur, up to a subsequence, as $\alpha \to + \infty$:
\begin{equation} \label{alternative}
\begin{aligned}
\textrm{ either } \quad & \da^{\frac{2k}{p_\alpha-2}} \max_{B_g(x_{i,\alpha}, 4 \da)} \ua \le C & \textrm{ (Case one)} \\
\textrm{ or } \quad  & \da^{\frac{2k}{p_\alpha-2}} \max_{B_g(x_{i,\alpha}, 4 \da)} \ua\longrightarrow + \infty & \textrm{ (Case two)} \\
\end{aligned}
\end{equation}
 as $\alpha \to + \infty$, where $C >0$ is independent of $\alpha$. If case two in \eqref{alternative} occurs at $x_{i,\alpha}$ then the function $v_{i,\alpha}:= \Lambda_{x_{i,\alpha}}^{-1} \ua$ satisfies \eqref{cond.xarhoa.va} and \eqref{cond.xarhoa.va2}, and the analysis of Sections~\ref{sec.local.analysis},~\ref{sec.order.two} and ~\ref{sec.estimate.weyl} apply. Since $\da \to 0$, in particular, Proposition~\ref{prop.order.1} applies and, together with \eqref{conf.1}, shows that
\begin{equation} \label{convi}
\left| \left| \frac{\ua}{B_{i,\alpha}} - 1 \right| \right|_{L^\infty(B_g(x_{i,\alpha}, \da))} = \smallo(1) 
\end{equation}
as $\alpha \to + \infty$ where we have let 
$$ B_{i,\alpha}:= \mu_{i,\alpha}^{n-2k-\frac{2k}{p_\alpha-2}} \( \mu_{i,\alpha}^2 + \mathfrak{c}_{n,k}^{-1} d_g(x_{i,\alpha},x)^2\)^{-\frac{n-2k}{2}} $$ 
and $\mu_{i,\alpha}:= \ua(x_{i,\alpha})^{- (p_\alpha-2)/2k}$. An obvious consequence of \eqref{convi} is that 
\begin{equation} \label{eq.case.two}
\Vert \da^{\frac{2k}{p_\alpha-2}} v_{i,\alpha} \Vert_{L^\infty(B_g(x_{i,\alpha}, \da) \backslash B_g(x_{i,\alpha}, \frac12\da))} = \smallo(1)  \quad \text{ as } \alpha \to + \infty. 
\end{equation}

\begin{claim}\label{claim.unique.nature}
Assume that, for some $i_0 \in \{1, \dots, N_R \}$, $x_{i_0,\alpha}$ satisfies the first case in \eqref{alternative}. Then every $x_{i,\alpha}$, $i \in \{1, \dots, N_R \} \backslash \{ i_0 \}$, also satisfies the first case in \eqref{alternative}.
\end{claim}

 \begin{proof}
Let $i_0 \in \{1, \dots, N_R \}$ for which Case one in \eqref{alternative} holds. Then 
\begin{equation} \label{minori}
\da^{\frac{2k}{p_\alpha-2}} \min_{B_g(x_{i,\alpha}, 2 \da)} \ua \ge C_{i_0} 
\end{equation}
for some positive constant $C_{i_0} > 0$ independent of $\alpha$. Indeed, define 
$$\check{u}_\alpha(x):=  \da^{\frac{2k}{p_\alpha-2}}\ua\big( \exp_{x_{i_0,\alpha}}^g (\da x) \big)$$
 for all $x \in B(0, 4)$. Since $\da \to 0$, the assumption that case one holds shows that $\check{u}_\alpha$ converges towards some nonnegative function $\check{u}_0$ in $C^{2k}_{loc}(B(0,4))$ that satifies $\Delta_0^k \check{u}_0 = \check{u}_0^{2^*_k-1}$ in $B(0,4)$. Property $(2)$ in Proposition \ref{weak.es} together with the definition of $N_R$, ensure that $\check{u}(0) >0$. Arguing as in the proof of \eqref{positivity.u0} we then obtain that $\check{u}_0 >0$ in $B(0,2)$, from which \eqref{minori} follows. Let $i \in \{1, \dots, N_R \}$, $ i \not = i_0$, and let $(z_{\alpha})_{\alpha}$ be a sequence of points in $B_g(x_{i,k}, 2R d_k)$. We write a representation formula for $\ua$ in $M$ at $z_k$. We have 
\begin{equation}  \label{minor2}
\begin{aligned}
u_k(z_k) & \ge \int_{B_g(x_{i_0,k}, 4 d_k)} G_g(z_k,\cdot)f^{p_\alpha-2^*_k} \ua^{p_\alpha-1}  dv_g
 \ge \frac{1}{C} d_k^{-\frac{2k}{p_\alpha-2}} 
\end{aligned}
\end{equation}
for some positive constant $C$ independent of $\alpha$, where the last line follows from \eqref{minori} and \eqref{bounds.Green}. Thus \eqref{minor2} implies that $\liminf_{\alpha \to + \infty} \min_{B_g(x_{i,\alpha}, \da)} \da^{\frac{2k}{p_\alpha-2}} \ua >0 $ and \eqref{eq.case.two} then shows that case two of \eqref{alternative} cannot be satisfied at $x_{i,\alpha}$, and hence that $x_{i,\alpha}$ also satisfies case one. 
\end{proof}
 
Claim~\ref{claim.unique.nature} shows in particular that, for any $R \ge 16$, either all the concentration points $x_{i,\alpha}$, $1 \le i \le N_R$ satisfy case one in \eqref{alternative} or they all satisfy case two. We first assume that all the $i \in \{1, \dots, N_R \}$ satisfy case one in \eqref{alternative}. Using \eqref{contptsconc} we get that the function
$$w_\alpha(x):= \da^{\frac{2k}{p_\alpha-2}} \ua \big( \exp_{x_{1,\alpha} }^g(\da x) \big), $$
defined for $x \in B\(0, i_g(M)/2 \da\)$, is locally bounded. By \eqref{eq:one} and standard elliptic theory it converges in $C^{2k}_{loc}(\mathbb{R}^n)$, towards a nonnegative solution $w_0$ of $\Delta_0^k w_0 = w_0^{2^*_k-1}$ in $\R^n$. Also, $w_0(0) \ge 1$ by Proposition~\ref{weak.es}, so that arguing as in the proof of \eqref{positivity.u0} shows that $w_0 >0$ in $\R^n$. By construction $0$ and 
$$\check{x}_2: = \lim_{k \to + \infty} \frac{1}{d_k} \exp_{x_{1,k}}^{-1}(x_{2,k})$$
 are distinct critical points of $w_0$, and this contradicts the result of~\cite{WeiXu}.

 Hence, for all $R \ge 16$ fixed, all the points $x_{i,\alpha}$, $1 \le i \le N_R$ satisfy case two in \eqref{alternative}. Let $(z_\alpha)_\alpha$ be a sequence of points in $B_g(x_{1,\alpha}, 2\da)$. A representation formula for $\ua$ gives, with \eqref{convi},
 \begin{equation}  \label{repfinale}
 \begin{aligned}
\ua(z_\alpha) &\ge \big(1 + \smallo(1) \big) \int_{B_g(x_{1,\alpha}, \da)} G_g(z_\alpha,y)B_{1,\alpha}^{p_\alpha-1} dv_g \\
&+  \big(1 + \smallo(1) \big) \int_{B_g(x_{2,\alpha}, \da)} G_g(z_\alpha,y)B_{2,\alpha}^{p_\alpha-1} dv_g  \\
& \ge  \big(1 + \smallo(1)   \big) \Big( B_{1,\alpha}(z_\alpha)  +  B_{2,\alpha}(z_\alpha) \big)
\end{aligned}
\end{equation}
as $\alpha \to + \infty$, where the last line follows from Fatou's lemma, \eqref{expansion.Green} below and the assumption $d_\alpha = \smallo(1)$ as $\alpha \to + \infty$. Choose first $z_\alpha$ so that that $d_g(x_{1,\alpha},z_\alpha) = \frac14d_\alpha $, so that $d_g(x_{2,\alpha}, z_{\alpha}) \ge \frac34 d_\alpha $. Using again \eqref{convi} to estimate the left-hand side of \eqref{repfinale} gives $B_{2,\alpha}(z_\alpha) \le \smallo(B_{1,\alpha}(z_\alpha))$ and hence
\begin{equation*} 
 \( \frac{\mu_{2,\alpha}}{\mu_{1,\alpha}} \)^{n-2k-\frac{2k}{p_\alpha-2}} = \smallo \(1\)
\end{equation*}
as $\alpha \to + \infty$. Choose now $z_\alpha$ so that $d_g(x_{2,\alpha},z_\alpha)=\frac14 d_\alpha $, so that $d_g(x_{1,\alpha}, z_{\alpha}) \ge \frac34d_\alpha $. Using again \eqref{convi} to estimate the left-hand side of \eqref{repfinale} similarly gives 
\begin{equation*}
 \( \frac{\mu_{1,\alpha}}{\mu_{2,\alpha}} \)^{n-2k-\frac{2k}{p_\alpha-2}} = \smallo \(1\)
 \end{equation*}
This is a contradiction and concludes the proof of Proposition \ref{prop.isolated.points}.
\end{proof}

We are now in  position to prove our main result:

\begin{proof}[Proof of Theorem~\ref{theo.main}]
We assume that the assumptions of Theorem~\ref{theo.main} are satisfied. Let $(\ua)_\alpha$ be a sequence of positive solutions of \eqref{eq:one}. We proceed by contradiction and assume that $u_\alpha$ satisfies \eqref{blowup}.  Let $(x_{1,\alpha} , \dotsc, x_{N_\alpha, \alpha})$ be the concentration points of $\ua$ given by Proposition~\ref{weak.es} and let $\da$ be given by \eqref{def.da}. By Proposition~\ref{prop.isolated.points}, up to passing to a subsequence, we can assume that $N_\alpha = N$ and that $d_\alpha \ge d >0$ for all $\alpha$. Let, for all $\alpha \ge 1$, $ \rhoa = d$. For each $1 \le i \le N$ the sequences $v_{x_i,\alpha} = \Lambda_{x_{i,\alpha}}^{-1} \ua$,  $(x_{i,\alpha})_\alpha$ and $(\rhoa)_\alpha$ satisfy  \eqref{cond.xarhoa.va}. The assumption \eqref{blowup} ensures that there exists $1 \le i \le N$ such that $(v_{x_{i,\alpha}})_\alpha, (x_{i,\alpha})_{\alpha}, (\rhoa)_\alpha$ also satisfy \eqref{cond.xarhoa.va2}. Up to reducing $N$ if needed we may thus assume that $(v_{x_{i,\alpha}})_\alpha, (x_{i,\alpha})_{\alpha}, (\rhoa)_\alpha$ satisfy  \eqref{cond.xarhoa.va} and \eqref{cond.xarhoa.va2} for any $1 \le i \le N$.

\smallskip 

We first assume that $n \ge 2k+4$ and $\min_M |\Weyl_g|_g >0$. Proposition~\ref{prop.order.1} then applies and shows that $\rhoa \to 0$ as $\alpha \to + \infty$, a contradiction with $\rho_\alpha = d >0$. 

\smallskip

We thus assume from now on that either $2k+1 \le n \le 2k+3$ or $2k+4 \le n \le 2k+5$ and $\min_M |\Weyl_g|_g = 0$, and that $G_g$ satisfies \eqref{positive:mass}. For any $1 \le i \le N$ we will let $r_{i,\alpha}$ be defined by \eqref{def.ra} and, following \eqref{def.mua}, we will let 
$$ \mu_{i, \alpha} = \ua(x_{i, \alpha})^{- \frac{2k}{p_\alpha-2}}.$$
 First, Proposition~\ref{prop.order.1} again applies at each $1 \le i \le N$, and since $\rhoa = d >0$ it shows that $r_{i,\alpha} \ge r_0>0$ for any $\alpha \ge 1$ and $1 \le i \le N$. Up to renumbering the points and passing to a subsequence we may assume that 
$$ \mu_{1,\alpha} = \max_{1 \le i \le N} \mu_{i,\alpha} \quad \text{ for all } \alpha \ge 1.$$
A consequence of $\liminf_{\alpha \to + \infty} r_{i,\alpha} >0$ for all $1 \le i \le N$ is that, by \eqref{sym.es18}, we have 
\begin{equation} \label{eq.finale.0}
2^{*}_{k} - p_\alpha = \bigO(\min_{1 \le i \le N}\mu_{i,\alpha}).
\end{equation}
With the latter,  \eqref{local.est1} shows that for every $1 \le i \le N$ we have  
$$\ua(x) \lesssim \mu_{1,\alpha}^{\frac{n-2k}{2}} d_g(x_{i,\alpha}, x)^{2k-n} \quad \text{ for } \quad x \in B_g(x_{i,\alpha}, d),$$
 and the Harnack inequality of Lemma \ref{local.2} then shows that $\ua(x) \lesssim \mu_{1, \alpha}^{\frac{n-2k}{2}}$ in $M \backslash \bigcup_{i=1}^N B_g(x_{i,\alpha}, \frac{d}{2})$. A straightforward adaptation of the arguments in the proof of Proposition \ref{local.3} then shows that 
 $$  \mu_{1, \alpha}^{-\frac{n-2k}{2}}\ua \to G_g(x_1, \cdot) + \sum_{i=2}^N a_i G_g(x_i, \cdot) + b \quad \text{ in } C^{2s}_{loc}( M \backslash \{x_1, \dots, x_N\}), $$
 where $x_i = \lim_{\alpha \to + \infty} x_{i,\alpha}$, $G_g$ is the Green's function for $P_g$, $0 \le a_i \le 1$ for $2 \le i \le N$ are nonnegative constants and $b \in C^{2k}(M)$ satisfies $P_g b = 0$. Assumption \eqref{positivity} then implies $b \equiv 0$, so that
\begin{equation} \label{eq.finale.1}
 \mu_{1, \alpha}^{-\frac{n-2k}{2}}\ua(x) \to G_g(x_1, \cdot) + h(x) 
\end{equation}
in $C^{2k}_{loc}( M \backslash \{x_1, \dots, x_N\})$ as $\alpha \to + \infty$, where $h$ satisfies $h(x_1) \ge 0$ and  $P_g h = 0$ in $M \backslash \{x_2, \dots, x_N\}$. Following \eqref{eq.defva} we now consider $v_{1,\alpha} =\Lambda^{-1}_{x_{1,\alpha}}u_{\alpha}$. First, Lemma~\ref{weyl.es} applies and shows, since $r_{1,\alpha} \ge r_0 >0$, that 
\begin{equation} \label{weyl.vanishing}
|\Weyl_{g}(x_{1,\alpha})|_g^{2}\lesssim\smallo(1)+\left\{\begin{aligned}&\(\ln\frac{1}{\mu_{1,\alpha}}\)^{-1}&&\text{if }n=2k+4\\&\mu_{1,\alpha} &&\text{if }n=2k+5\end{aligned}\right. .
\end{equation}
Let $\delta >0$ be fixed. Since $r_{1, \alpha} \ge r_0 >0$ we may apply the Pohozaev inequality \eqref{weyl.es1} to $\widetilde{v}_{1, \alpha} = v_{1,\alpha} \circ \exp_{x_{1,\alpha}}^{g_{x_{1,\alpha}}}$ in $B(0, \delta)$, provided $\delta$ is small enough. By \eqref{eq.finale.0} we have $\ma^{\frac{4k}{p_\alpha-2} + 2k-n} = 1+\smallo(1)$ as $\alpha \to + \infty$. The estimates \eqref{weyl.es2}, \eqref{integrals.123}, \eqref{weyl.es3}, \eqref{weyl.es4} then show, since $p_\alpha \le 2^*_k$, that 
\begin{equation} \label{conclusion.1}
\begin{aligned}
&\(1+ \smallo(1) \)  \mathcal{P}_{k}(\delta;\widetilde{v}_{1,\alpha})  + \bigO \(\delta \mu_{1,\alpha}^{n-2k} \)\\
& \le \int_{B(0,\delta)}\(\frac{n-2k}{2}\widetilde{U}_{\alpha}+x^{i}\partial_{i}\widetilde{U}_{\alpha}\)(\Delta_{0}^{k}\,\widetilde{U}_{\alpha} - P_{\tga}\widetilde{U}_{\alpha})\,dx,
\end{aligned} 
\end{equation}
 where we have let 
$$  \tilde{U}_\alpha(x) = \frac{\mu_{1,\alpha}^{\frac{n-2k}{2}}}{\left(\mu_{1,\alpha}^{2}+\mathfrak{c}^{-1}_{n,k}~|x|^2\right)^{\frac{n-2k}{2}}} \quad \text{ for } x \in \R^n.$$
We claim that the following holds as $\alpha \to + \infty$:
\begin{equation}\label{conclusion.2}
\begin{aligned}
\int_{B(0,\delta)}\(\frac{n-2k}{2}\widetilde{U}_{\alpha}+x^{i}\partial_{i}\widetilde{U}_{\alpha}\)(\Delta_{0}^{k}\,\widetilde{U}_{\alpha} - P_{\tga}\widetilde{U}_{\alpha})\,dx  \le \bigO\( \delta \mu_{1,\alpha}^{n-2k} \),
\end{aligned} 
\end{equation}
where the constant in the $\bigO(\cdot)$ term is independent of $\alpha$ and $\delta$. First, if $2k+1 \le n \le 2k+3$, \eqref{conclusion.2} immediately follows from Proposition \ref{GJMS.exp2}. If now $2k+4 \le n \le 2k+5$, Proposition \ref{GJMS.exp2} and \eqref{weyl.vanishing} show that 
$$
\begin{aligned}
\int_{B(0,\delta)}& \(\frac{n-2k}{2}\widetilde{U}_{\alpha}+x^{i}\partial_{i}\widetilde{U}_{\alpha}\)(\Delta_{0}^{k}\,\widetilde{U}_{\alpha} - P_{\tga}\widetilde{U}_{\alpha})\,dx \\
& \le \bigO\( \delta \mu_{1,\alpha}^{n-2k} \) +C |\Weyl_g(x_{1,\alpha})|_g^2 \times  \left\{\begin{aligned}&\mu_{1,\alpha}^4 &&\text{if }n=2k+4\\&\mu_{1,\alpha}^5 &&\text{if }n=2k+5.\end{aligned}\right\} \\
&\le \bigO\( \delta \mu_{1,\alpha}^{n-2k} \) + \smallo(\mu_{1,\alpha}^{n-2k} )
\end{aligned} 
$$
as $\alpha \to + \infty$, which proves \eqref{conclusion.2}.  Combining  \eqref{conclusion.1} and \eqref{conclusion.2} we finally obtain 
\begin{equation} \label{conclusion.3}
\begin{aligned}
&\mu_{1,\alpha}^{- (n-2k)} \mathcal{P}_{k}(\delta;\widetilde{v}_{1,\alpha})  \le \bigO(\delta) 
\end{aligned} 
\end{equation}
where the constant in the $\bigO(\cdot)$ term is independent of $\alpha$ and $\delta$. Independently, and by \eqref{eq.finale.1}, we have 
$$ \mu_{1,\alpha}^{- (n-2k)} \mathcal{P}_{k}(\delta;\widetilde{v}_{1,\alpha}) \to  \mathcal{P}_{k}(\delta ; \Lambda) \quad \text{ as } \alpha \to + \infty,$$
where, by the conformal invariance \eqref{conf.inv.Pg} of $P_g$, we have
$$\Lambda(y) = G_{g_{x_1}} \big(x_1, \exp_{x_1}^{g_{x_1}} (y) \big) + \tilde{h}(y) \quad \text{ for } y \in B(0, \delta),  $$
where $\tilde{h}(0) \ge 0$. Combining the latter with \eqref{conclusion.3} then yields 
\begin{equation} \label{conclusion.4}
\limsup_{\delta \to 0}  \mathcal{P}_{k}(\delta ; \Lambda) \le 0. 
\end{equation}
The boundary integral $\mathcal{P}_{k}(\delta ; \Lambda) $ is computed by applying Proposition~\ref{masse.Green} below: we obtain that
\begin{equation} \label{conclusion.5}
 \lim_{\delta \to 0}  \mathcal{P}_{k}(\delta ; \Lambda) = c_{n,k} \big(A_{x_1} + \tilde{h}(0) \big) \ge c_{n,k} A_{x_1} >0, 
 \end{equation}
where $c_{n,k}$ is a positive numerical constant and where $A_{x_1}$ is the mass of the Green's function $G_{g_{x_1}}$ in conformal normal coordinates at $x_1$ (see Proposition \ref{expansion.Green.local} below). We used the positive mass assumption \eqref{positive:mass} for the strict inequality in \eqref{conclusion.5}. This contradicts \eqref{conclusion.4}.
Thus \eqref{blowup} cannot happen, and every sequence of solutions $(u_\alpha)_\alpha$ of \eqref{eq:one} is uniformly bounded in $L^\infty(M)$ as $\alpha \to + \infty$. Theorem~\ref{theo.main} thus follows from standard elliptic theory. 
\end{proof}

\begin{remark}
If we only assume \eqref{positivity} and do not assume that $G_g$ has positive mass at every point, Proposition \ref{prop.isolated.points} still applies and shows that any blowing-up sequence of solutions of \eqref{eq:one} is a priori bounded in $H^k(M)$ and blows up at distinct isolated and simple points. This is the analogue of Theorem 1.2 of \cite{LiXiong} when $2k+1\le n\le 2k+5$, for an arbitrary $k<\frac{n}{2}$. The positive mass assumption is indeed only used in the proof of the final step of Theorem \ref{theo.main}. \end{remark}

%

\smallskip

\appendix

\section{Pohozaev identity for $\Delta_0^{k}$}\label{sec.pohozaev}

We extend the Pohozaev identity in Proposition 2.2 of \cite{LiXiong} to all integers $k\geq1$. Recall that $\Delta_{0}:=-\sum\limits_{i=1}^{n}\partial^2_{ii}$ is the Euclidean Laplacian. We let 
$$\Delta_{0}^{k/2}u:=\left\{\begin{aligned}&\Delta_{0}^{\frac{k}{2}}u&&\text{if $k$ is even},\\&\nabla\Delta_{0}^{\frac{k-1}{2}}u&&\text{if $k$ is odd}.\end{aligned}\right.$$

\begin{proposition} \label{prop.poho.appendix}
Let $0\le s< r$ and $p\geq2$. Let $u\in C^{2k}(\overline{B(0,r)})$ and $f\in C^{1}(\overline{B(0,r)})$ and denote $\mathcal{E}(u):=\Delta_{0}^{k}u-f|u|^{p-2}u$. Then we have
\begin{equation} \label{poho.id0}
\begin{aligned}
&\mathcal{P}_{k}(r;u) - \mathcal{P}_{k}(s;u)=\int_{B(0,r) \backslash B(0,s)}\(\frac{n-2k}{2}u+x^{i}\partial_{i}u\)\mathcal{E}(u)\,dx \\
&~+\(\frac{n-2k}{2}-\frac{n}{p}\)\int_{B(0,r) \backslash B(0,s)}f\,|u|^{p}\,dx-\frac{1}{p}\int_{B(0,r) \backslash B(0,s)}x^{i}\partial_{i}f\,|u|^{p}\,dx\\
&~+\frac{r}{p}\int_{\partial B(0,r)}f|u|^{p}\,d\sigma-\frac{s}{p}\int_{\partial B(0,s)}f|u|^{p}\,d\sigma.
\end{aligned}
\end{equation}
Here $\mathcal{P}_{k}(r;u)$ denotes boundary terms whose expression is given by: 
\begin{equation} \label{poho.id01}
\begin{aligned} 
&\mathcal{P}_{k}(r;u)  = \mathcal{R}_{k}(r;u) \\
&+\sum_{i=0}^{[k/2]-1} \Bigg[\frac{n-2k}{2}\int_{\partial B(0,r)}\bigg(\partial_{\nu}(\Delta_{0}^{i}u)\,\Delta_{0}^{k-i-1}u\,-\Delta_{0}^{i}u~\partial_{\nu}(\Delta_{0}^{k-i-1}u)\bigg)\,d\sigma \\
&+\int_{\partial B(0,r)}\bigg(\partial_{\nu}(\Delta_{0}^{i}(x^{a}\partial_{a}u))\,\Delta_{0}^{k-i-1}u\, -\Delta_{0}^{i}(x^{a}\partial_{a}u)\,\partial_{\nu}(\Delta_{0}^{k-i-1}u)\bigg)\,d\sigma \Bigg]. \\
\end{aligned}
\end{equation}
In \eqref{poho.id01} we have let
\begin{equation*}
\mathcal{R}_{k}(r;u)=  \frac{r}{2}\int_{\partial B(0,r)}(\Delta_{0}^{k/2}u)^{2}d\sigma  \quad  \text{ if } k \text{ is even } \\
\end{equation*}
and 
\begin{equation*} 
\begin{aligned}
  \mathcal{R}_{k}(r;u) & =  \frac{r}{2}\int_{\partial B(0,r)}(\Delta_{0}^{\frac{k+1}{2}}u)(\Delta_{0}^{\frac{k-1}{2}}u)\,d\sigma \\ 
& + \frac{1}{2}\int_{\partial B(0,r)}\bigg(\Delta_{0}^{\frac{k-1}{2}}u\,\partial_{\nu}(x^{a}\partial_{a}(\Delta_{0}^{\frac{k-1}{2}}u))\,-(x^{a}\partial_{a}(\Delta_{0}^{\frac{k-1}{2}}u))\partial_{\nu}(\Delta_{0}^{\frac{k-1}{2}}u)\bigg)\,d\sigma \\
& \quad \quad\quad\quad\quad \quad\quad\quad\quad\quad\quad\quad \quad\quad\quad\quad\quad\quad\quad\quad\text{ if } k \text{ is odd. } 
\end{aligned}  
\end{equation*}
In the previous expressions, $[x]$ denotes the integer part of $x \in \mathbb{R}$. Note that the expression of $\mathcal{R}_{k}(r;u)$ depends on the parity of $k$. 

\end{proposition}
\begin{proof}
\noindent
Let $\Omega\subset\R^{n}$ denote a bounded smooth domain. We follow the arguments in Gazzola-Grunau-Sweers \cite[Chapter 7]{GazzolaGrunauSweers} but without imposing boundary conditions in $\partial \Omega$. We use the convention that repeated indices are summed over, so that $x^i \partial_i = \sum_{i=1}^n x^i \partial_i$. We will need two preliminary computations. First, repeated integration by parts show that for any $u,v \in C^{2k}(\overline{\Omega})$,
\begin{equation} \label{poho.k.1}
\begin{aligned}
& \int_{\Omega}u\Delta_{0}^{k}v~dx=\int_{\Omega}(\Delta_{0}^{k/2}u,\Delta_{0}^{k/2}v)\,dx -\underbrace{\int_{\partial \Omega}\Delta_{0}^{\[\frac{k}{2}\]}u~\partial_{\nu}\Delta_{0}^{\[\frac{k}{2}\]}v\,d\sigma}_{\text{if $k$ is odd, $0$ otherwise}} \\
&+\sum_{i=0}^{\[k/2\]-1}\int_{\partial\Omega}\bigg(\partial_{\nu}\Delta_{0}^{i}u\,\Delta_{0}^{k-1-i}v\,-\,\Delta_{0}^{i}u~\partial_{\nu}\Delta_{0}^{k-1-i}v\bigg)\,d\sigma.
\end{aligned}
\end{equation}
Here $\nu$ denotes the outer unit normal. Direct computations show that $\Delta_0(x^i \partial_i u ) = x^i \partial_i \Delta_0 u + 2 \Delta_0 u$. Iterating gives 
\begin{equation} \label{poho.k.0}
\Delta_{0}^{k}(x^{i}\partial_{i}u)=x^{i}\partial_{i}\Delta_{0}^{k}u+2k\Delta_{0}^{k}u. 
\end{equation}
Using \eqref{poho.k.1} we have
\begin{equation}
\begin{aligned}
& \int_{\Omega}\frac{n-2k}{2}u \Delta_0^k u \,dx  =  \frac{n-2k}{2}\int_{\Omega}|\Delta_{0}^{k/2}u|^{2}\,dx\\
& - \frac{n-2k}{2}\underbrace{\int_{\partial \Omega}\Delta_{0}^{\[\frac{k}{2}\]}u~\partial_{\nu}\Delta_{0}^{\[\frac{k}{2}\]}u\,d\sigma}_{\text{if $k$ is odd, $0$ otherwise}} \\
& + \frac{n-2k}{2} \sum_{i=0}^{[k/2]-1}\int_{\partial \Omega}\bigg(\partial_{\nu}(\Delta_{0}^{i}u)\Delta_{0}^{k-i-1}u\,-\Delta_{0}^{i}u\,\partial_{\nu}(\Delta_{0}^{k-i-1}u)\bigg)\,d\sigma .
\end{aligned} 
\end{equation}
Independently, direct computations give
\begin{align*}
&\int_{\Omega}\(x^{i}\partial_{i}u+\frac{n-2k}{2}u\)f|u|^{p-2}u~dx=\frac{n-2k}{2}\int_{\Omega}f|u|^{p}~dx+\frac{1}{p}\int_{\Omega}f~x^{i}\partial_{i}|u|^{p}\,dx\notag\\
&=\(\frac{n-2k}{2}-\frac{n}{p}\)\int_{\Omega}f|u|^{p}\,dx-\frac{1}{p}\int_{\Omega}x^{i}\partial_{i}f~|u|^{p}\,dx+\frac{1}{p}\int_{\partial\Omega}(x,\nu)f|u|^{p}\,d\sigma.
\end{align*}
Combining the latter computations then gives
\begin{equation} \label{base.du.poho}
\begin{aligned}
& \int_{\Omega}\(x^{i}\partial_{i}u + \frac{n-2k}{2}u\)\mathcal{E}(u)\,dx  = \int_{\Omega} x^{i}\partial_{i}u~\Delta_{0}^{k}u\,dx \\
&  -\(\frac{n-2k}{2}-\frac{n}{p}\)\int_{\Omega}f|u|^{p}\,dx+\frac{1}{p}\int_{\Omega}x^{i}\partial_{i}f~|u|^{p}\,dx-\frac{1}{p}\int_{\partial\Omega}(x,\nu)f|u|^{p}\,d\sigma \\
& + \frac{n-2k}{2}\int_{\Omega}|\Delta_{0}^{k/2}u|^{2}\,dx - \frac{n-2k}{2}\underbrace{\int_{\partial \Omega}\Delta_{0}^{\[\frac{k}{2}\]}u~\partial_{\nu}\Delta_{0}^{\[\frac{k}{2}\]}u\,d\sigma}_{\text{if $k$ is odd, $0$ otherwise}} \\
& + \frac{n-2k}{2} \sum_{i=0}^{[k/2]-1}\int_{\partial \Omega}\bigg(\partial_{\nu}(\Delta_{0}^{i}u)\Delta_{0}^{k-i-1}u\,-\Delta_{0}^{i}u\,\partial_{\nu}(\Delta_{0}^{k-i-1}u)\bigg)\,d\sigma .
\end{aligned} 
\end{equation}

\noindent 
{\bf{Case 1:} $k$ is even.} Using \eqref{poho.k.1} we get
\begin{align*}
&\int_{\Omega} x^{i}\partial_{i}u~\Delta_{0}^{k}u\,dx=\int_{\Omega}\Delta_{0}^{k/2}(x^{i}\partial_{i}u)\,\Delta_{0}^{k/2}u\,dx+\sum\limits_{i=0}^{k/2-1}\int_{\partial\Omega}\bigg(\partial_{\nu}\Delta_{0}^{i}(x^{\ell}\partial_{\ell}u)\,\Delta_{0}^{k-1-i}u\\
&\quad-\Delta_{0}^{i}(x^{\ell}\partial_{\ell}u)\,\partial_{\nu}\Delta_{0}^{k-1-i}u\bigg)\,d\sigma.\\
\end{align*}
The first integral is computed using \eqref{poho.k.0}. We get, integrating by parts,
\begin{align*}
& \int_{\Omega}\Delta_{0}^{k/2}(x^{i}\partial_{i}u)\,\Delta_{0}^{k/2}u\,dx\\
& =k\int_{\Omega}(\Delta_{0}^{k/2}u)^{2}\,dx+\int_{\Omega}x^{i}\partial_{i}(\Delta_{0}^{k/2}u)\,\Delta_{0}^{k/2}u\,dx \\
&=\frac{2k-n}{2}\int_{\Omega}(\Delta_{0}^{k/2}u)^{2}\,dx+\frac{1}{2}\int_{\partial\Omega}(x,\nu)(\Delta_{0}^{k/2}u)^{2}\,d\sigma.\\
\end{align*}
Combining the latter computations with \eqref{base.du.poho} and taking $\Omega=B(0,r) \backslash B(0,s)$, proves \eqref{poho.id01} when $k$ is even. 

\medskip

\noindent 
{\bf{Case 2:} $k$ is odd.} Similarly, using \eqref{poho.k.1}, we get
\begin{equation*}
\begin{aligned}
&\int_{\Omega} x^{i}\partial_{i}u\,\Delta_{0}^{k}u=\int_{\Omega}\Delta_{0}^{\frac{k-1}{2}}(x^{i}\partial_{i}u)\,\Delta_{0}^{\frac{k+1}{2}}u\,dx+\sum_{i=0}^{(k-1)/2-1}\int_{\partial\Omega}\bigg(\partial_{\nu}\Delta_{0}^{i}(x^{\ell}\partial_{\ell}u)\times\\
&\qquad\Delta_{0}^{k-1-i}u-\Delta_{0}^{i}(x^{\ell}\partial_{\ell}u)~\partial_{\nu}\Delta_{0}^{k-1-i}u\bigg)\,d\sigma.\\
\end{aligned}
\end{equation*}
Again, the first integral is computed using \eqref{poho.k.0}. We get, integrating by parts,
\begin{equation*}
\begin{aligned}
& \int_{\Omega}\Delta_{0}^{\frac{k-1}{2}}(x^{i}\partial_{i}u)\,\Delta_{0}^{\frac{k+1}{2}}u\,dx
=\int_{\Omega}x^{i}\partial_{i}(\Delta_{0}^{\frac{k-1}{2}}u)\,\Delta_{0}^{\frac{k+1}{2}}u~dx\\
&+(k-1)\int_{\Omega}|\Delta_{0}^{k/2}u|^{2}\,dx -(k-1)\int_{\partial\Omega}\Delta_{0}^{\frac{k-1}{2}}u\,\partial_{\nu}\Delta_{0}^{\frac{k-1}{2}}u\,d\sigma. 
\end{aligned}
\end{equation*}
We next calculate, integrating by parts, 
\begin{equation} \label{starstar}
\begin{aligned}
&\int_{\Omega}x^{i}\partial_{i}(\Delta_{0}^{\frac{k-1}{2}}u)\,\Delta_{0}^{\frac{k+1}{2}}u~dx\\
& =-n\int_{\Omega}\Delta_{0}^{\frac{k-1}{2}}u\,\Delta_{0}^{\frac{k+1}{2}}u\,dx-\int_{\Omega}x^{i}\partial_{i}(\Delta_{0}^{\frac{k+1}{2}}u)\,\Delta_{0}^{\frac{k-1}{2}}u\,dx\\\
&+\int_{\partial\Omega}(x,\nu)\,\Delta_{0}^{\frac{k-1}{2}}u\,\Delta_{0}^{\frac{k+1}{2}}u\,d\sigma\\
& = -n \int_{\Omega}|\Delta_{0}^{k/2}u|^{2}\,dx + n \int_{\partial \Omega} \Delta_0^{\frac{k-1}{2}}u \partial_\nu \Delta_0^{\frac{k-1}{2}}u d \sigma \\ 
& -\int_{\Omega}x^{i}\partial_{i}(\Delta_{0}^{\frac{k+1}{2}}u)\,\Delta_{0}^{\frac{k-1}{2}}u\,dx +\int_{\partial\Omega}(x,\nu)\,\Delta_{0}^{\frac{k-1}{2}}u\,\Delta_{0}^{\frac{k+1}{2}}u\,d\sigma.\\
\end{aligned}
\end{equation}
Recall that  $x^i \partial_i \Delta_0 h = \Delta_0(x^i \partial_i h) - 2 \Delta_0 h$ for any $h \in C^2(\overline{\Omega})$. Applying this to $h = \Delta_0^{\frac{k-1}{2}}u$ and integrating by parts gives
\begin{align*}
& \int_{\Omega}x^{i}\partial_{i}(\Delta_{0}^{\frac{k+1}{2}}u)\,\Delta_{0}^{\frac{k-1}{2}}u\,dx \\
& = \int_{\Omega}x^{i}\partial_{i}(\Delta_{0}^{\frac{k-1}{2}}u)\,\Delta_{0}^{\frac{k+1}{2}}u\,dx  + \int_{\partial\Omega}\Big[ (x^{i}\partial_{i}\Delta_{0}^{\frac{k-1}{2}}u)\partial_{\nu}\Delta_{0}^{\frac{k-1}{2}}u -\Delta_{0}^{\frac{k-1}{2}}u\,\partial_{\nu}(x^{i}\partial_{i}\Delta_{0}^{\frac{k-1}{2}}u)\Big]\,d\sigma \\
& - 2 \int_{\Omega}|\Delta_{0}^{k/2}u|^{2}\,dx + 2 \int_{\partial \Omega} \Delta_0^{\frac{k-1}{2}}u \partial_\nu \Delta_0^{\frac{k-1}{2}}u d \sigma . 
\end{align*}
Combining the latter with \eqref{starstar} shows that
\begin{align*}
&\int_{\Omega}x^{i}\partial_{i}(\Delta_{0}^{\frac{k-1}{2}}u)\,\Delta_{0}^{\frac{k+1}{2}}u\,dx=-\frac{(n-2)}{2}\int_{\Omega}|\Delta_{0}^{k/2}u|^{2}\,dx+\frac{(n-2)}{2}\int_{\partial\Omega}\Delta_{0}^{\frac{k-1}{2}}u\,\times\\
&~\,\partial_{\nu}\Delta_{0}^{\frac{k-1}{2}}u\,dx+\frac{1}{2}\int_{\partial\Omega}(x,\nu)\,\Delta_{0}^{\frac{k-1}{2}}u\,\Delta_{0}^{\frac{k+1}{2}}u\,d\sigma+\frac{1}{2}\int_{\partial\Omega}\bigg(\Delta_{0}^{\frac{k-1}{2}}u\,\partial_{\nu}(x^{i}\partial_{i}\Delta_{0}^{\frac{k-1}{2}}u)\\
&-(x^{i}\partial_{i}\Delta_{0}^{\frac{k-1}{2}}u)\,\partial_{\nu}\Delta_{0}^{\frac{k-1}{2}}u\bigg)\,d\sigma.
\end{align*}
We have thus shown that 
\begin{align*}
&\int_{\Omega} x^{i}\partial_{i}u~\Delta_{0}^{k}u\,dx=\frac{2k-n}{2}\int_{\Omega}|\Delta_{0}^{k/2}u|^{2}\,dx+\frac{n-2k}{2}\int_{\partial\Omega}\Delta_{0}^{\frac{k-1}{2}}u\,\partial_{\nu}\Delta_{0}^{\frac{k-1}{2}}u\,d\sigma\\
&+\frac{1}{2}\int_{\partial\Omega}(x,\nu)\,\Delta_{0}^{\frac{k-1}{2}}u~\Delta_{0}^{\frac{k+1}{2}}u\,d\sigma+\sum_{i=0}^{({k-1})/2-1}\int_{\partial\Omega}\bigg(\partial_{\nu}\Delta_{0}^{i}(x^{i}\partial_{i}u)\,\Delta_{0}^{k-1-i}u\\
&~-\Delta_{0}^{i}(x^{i}\partial_{i}u)\,\partial_{\nu}\Delta_{0}^{k-1-i}u\bigg)\,d\sigma+\frac{1}{2}\int_{\partial\Omega}\bigg(\Delta_{0}^{\frac{k-1}{2}}u\,\partial_{\nu}(x^{i}\partial_{i}\Delta_{0}^{\frac{k-1}{2}}u)\\
&~-(x^{i}\partial_{i}\Delta_{0}^{\frac{k-1}{2}}u)\,\partial_{\nu}\Delta_{0}^{\frac{k-1}{2}}u\bigg)\,d\sigma.
\end{align*}
Combining with \eqref{base.du.poho} and taking $\Omega=B(0,r) \backslash B(0,s)$, concludes the proof of \eqref{poho.id01} when $k$ is odd. 
\end{proof}

\noindent
The following lemma computes the contribution of the boundary terms in the Pohozaev identity \eqref{poho.id0} for a function having a polyharmonic singularity at the origin: 
\begin{lemma}\label{sing+har}
Let $u:B(0,2)\setminus\{0\}\to\R$ be given by $u(x)=\Lambda|x|^{2k-n}+\mathcal{H}(x)$ for some $\Lambda>0$ and $\mathcal{H}\in C^{2k}\(\overline{B(0,2)}\)$. Then 
\begin{align}
\lim\limits_{r\to0}\mathcal{P}_{k}(r;u)=&\Theta(n,k) \Lambda\,\mathcal{H}(0),
\end{align}
where we have let
\begin{equation} \label{def.Theta}
\Theta(n,k) = \omega_{n-1}2^{k-2}(k-1)!(n-2k)^{2}(n-2k-2)\cdots(n-4)(n-2) >0.
\end{equation}
\end{lemma}
\begin{proof}
Let $\mathcal{G}=\Lambda|x|^{2k-n}$ and write  $u=\mathcal{G}+\mathcal{H}$. The expression $\mathcal{P}_{k}(r;u)$ given by Proposition \ref{prop.poho.appendix} is a quadratic functional in $u$ and we will denote  by $\Phi_{k,r}(\cdot,\cdot)$ the associated symmetric bilinear form, which satisfies $\Phi_{k,r}(u,u) = \mathcal{P}_{k}(r;u)$. It is defined as follows: 
\begin{equation}\label{bilineaire}
\begin{aligned}
& \Phi_{k,r}(\mathcal{X},\mathcal{Y}) \\
&= \frac{n-2k}{2}\sum_{i=0}^{[k/2]-1} \Bigg[ \int_{\partial B(0,r)}\bigg(\partial_{\nu}(\Delta_{0}^{i}\mathcal{X})\Delta_{0}^{k-i-1}\mathcal{Y}\,-\Delta_{0}^{i}\mathcal{X}~\partial_{\nu}(\Delta_{0}^{k-i-1}\mathcal{Y})\bigg)\,d\sigma \\
& +\int_{\partial B(0,r)}\bigg(\partial_{\nu}(\Delta_{0}^{i}\mathcal{Y})\Delta_{0}^{k-i-1}\mathcal{X}\,-\Delta_{0}^{i}\mathcal{Y}~\partial_{\nu}(\Delta_{0}^{k-i-1}\mathcal{X})\bigg)\,d\sigma \Bigg] \\
&+\sum_{i=0}^{[k/2]-1} \Bigg[ \int_{\partial B(0,r)}\bigg(\partial_{\nu}(\Delta_{0}^{i}(x^{a}\partial_{a}\mathcal{X}))\,\Delta_{0}^{k-i-1}\mathcal{Y}-\Delta_{0}^{i}(x^{a}\partial_{a}\mathcal{X})\,\partial_{\nu}(\Delta_{0}^{k-i-1}\mathcal{Y})\bigg)\,d\sigma \\
& +\int_{\partial B(0,r)}\bigg(\partial_{\nu}(\Delta_{0}^{i}(x^{a}\partial_{a}\mathcal{Y}))\,\Delta_{0}^{k-i-1}\mathcal{X}\,-\Delta_{0}^{i}(x^{a}\partial_{a}\mathcal{Y})\,\partial_{\nu}(\Delta_{0}^{k-i-1}\mathcal{X})\bigg)\,d\sigma \Bigg] \\
& +  \Psi_{k,r}(\mathcal{X},\mathcal{Y}),
\end{aligned}
\end{equation}
where $ \Psi_{k,r}(\mathcal{X},\mathcal{Y})$ is defined as follows: 
\begin{itemize}
\item
If $k$ is even, then 
\begin{align} \label{bilineaire.pair}
&\Psi_{k,r}(\mathcal{X},\mathcal{Y})=r\int_{\partial B(0,r)}\Delta_{0}^{k/2}\mathcal{X}~\Delta_{0}^{k/2}\mathcal{Y}\,d\sigma
\end{align}
\item 
If $k$ is odd, then 
\begin{align} \label{bilineaire.impair}
&\Psi_{k,r}(\mathcal{X},\mathcal{Y})=\frac{r}{2}\int_{\partial B(0,r)}\Delta_{0}^{\frac{k+1}{2}}\mathcal{X}~\Delta_{0}^{\frac{k-1}{2}}\mathcal{Y}\,d\sigma+\frac{r}{2}\int_{\partial B(0,r)}\Delta_{0}^{\frac{k+1}{2}}\mathcal{Y}~\Delta_{0}^{\frac{k-1}{2}}\mathcal{X}\,d\sigma\notag\\
&~+\frac{1}{2}\int_{\partial B(0,r)}\bigg(\Delta_{0}^{\frac{k-1}{2}}\mathcal{X}~\partial_{\nu}(x^{a}\partial_{a}(\Delta_{0}^{\frac{k-1}{2}}\mathcal{Y}))\,-(x^{a}\partial_{a}(\Delta_{0}^{\frac{k-1}{2}}\mathcal{X}))\partial_{\nu}(\Delta_{0}^{\frac{k-1}{2}}\mathcal{Y})\bigg)\,d\sigma\notag\\
&~+\frac{1}{2}\int_{\partial B(0,r)}\bigg(\Delta_{0}^{\frac{k-1}{2}}\mathcal{Y}~\partial_{\nu}(x^{a}\partial_{a}(\Delta_{0}^{\frac{k-1}{2}}\mathcal{X}))\,-(x^{a}\partial_{a}(\Delta_{0}^{\frac{k-1}{2}}\mathcal{Y}))\partial_{\nu}(\Delta_{0}^{\frac{k-1}{2}}\mathcal{X})\bigg)\,d\sigma\notag\\
&~+ \frac{n-2k}{2} \int_{\partial B(0,r)}\bigg( \Delta_0^{\frac{k-1}{2}} \mathcal{X} \partial_\nu \Delta_0^{\frac{k-1}{2}} \mathcal{Y} + \Delta_0^{\frac{k-1}{2}} \mathcal{Y} \partial_\nu \Delta_0^{\frac{k-1}{2}} \mathcal{X} \bigg)\, d\sigma.  
\end{align}
\end{itemize}
By bilinearity we have for all $r>0$
\begin{align*}
\mathcal{P}_{k}(r;u)=\mathcal{P}_{k}(r;\mathcal{G})+2 \Phi_{k,r}(\mathcal{G},\mathcal{H})+\mathcal{P}_{k}(r;\mathcal{H}).
\end{align*}
Since  $\mathcal{H}\in C^{2k}\(\overline{B(0,2)}\)$, and by the definition of $\mathcal{P}_{k}(r,\cdot)$ it follows that
\begin{align*}
\lim \limits_{r\to0}|\mathcal{P}_{k}(r,\mathcal{H})|\lesssim\lim \limits_{r\to0}\(\|\mathcal{H}\|^{2}_{C^{2k}\(\overline{B(0,2)}\)}~r^{n-1}\)=0.
\end{align*}
We now estimate the bilinear term $\Phi_{k,r}(\mathcal{G},\mathcal{H})$. From \eqref{bilineaire}, \eqref{bilineaire.pair} and \eqref{bilineaire.impair} it is easily seen that 
the boundary terms appearing in $\Phi_{k,r}(\mathcal{G},\mathcal{H})$ involve derivatives of $\mathcal{G}$ of order at most $2k-1$. Since $\mathcal{H}\in C^{2k}\(\overline{B(0,2)}\)$, all the integrals involving derivatives of order at most $2k-2$ of $\mathcal{G}$ will be estimated as $\bigO(r)$ as $r \to 0$. Similarly, derivatives of $\mathcal{G}$ of order $2k-1$ multiplied by $x_i$ will be estimated as $\bigO(r)$. We therefore have, thanks to \eqref{bilineaire}, \eqref{bilineaire.pair} and \eqref{bilineaire.impair},
\begin{equation} \label{poho.singular.1}
\Phi_{k,r}(\mathcal{G},\mathcal{H})=-\frac{n-2k}{2}\int_{\partial B(0,r)}\mathcal{H}\,\partial_{\nu}(\Delta_{0}^{k-1}\mathcal{G})\,d\sigma+\bigO(r).
\end{equation}
By direct computations $\Delta_{0}\,r^{\alpha-n}=(\alpha-2)(n-\alpha)r^{\alpha-2-n}$ with $\alpha>0$. In particular $\Delta_{0}^{k-1}\,\mathcal{G}=2^{k-1}(k-1)!(n-2k)(n-2k-2)\cdots(n-4)\,\Lambda r^{2-n}$. Therefore, and by \eqref{poho.singular.1},
$$\Phi_{k,r}(\mathcal{G},\mathcal{H})=\Theta(n,k) \Lambda\mathcal{H}(0)+\bigO(r)$$
as $r \to 0$, where $\Theta(n,k)$ is given by \eqref{def.Theta}. We finally compute the term $\mathcal{P}_{k}(r;\mathcal{G})$. Since $\mathcal{G}$ is polyharmonic in $\R^{n}\setminus\{0\}$, applying the Pohozaev identity \eqref{poho.id0} in $B(0,r)\setminus B(0,\tilde{r})$ with $f\equiv 0$ gives that $r\mapsto\mathcal{P}_{k}(r;\mathcal{G})$ is constant. Moreover  $|\mathcal{P}_{k}(r;\mathcal{G})|\lesssim |r|^{2k-n}$, and then letting $r\to+\infty$ we obtain that $\mathcal{P}_{k}(r;\mathcal{G})\equiv0$ for all $r>0$. This concludes the proof. 
\end{proof}
\smallskip

\section{The Green's Function for the GJMS operator}\label{sec.green.expan}

In this section, we prove some properties of the Green's function $G_g$ of $P_g$. Under the assumption \eqref{positivity}, $G_g$ is well-defined. We recall that, for $\xi \in M$, $g_\xi := \Lambda_\xi^{\frac{4}{n-2k}} g$ denotes the conformal metric defined in \eqref{conf.1}. Using \eqref{conf.inv.Pg} it is easily seen that, for any $\xi \in M$ and for any $x \neq y $ in $M$, we have 
$$ G_{g_\xi}(x,y) = \Lambda_\xi(x)^{-1} \Lambda_\xi(y)^{-1} G_g(x,y). $$ 
In the following, if $f$ is a smooth function in $\R^n\backslash \{0\}$ and $p,q$ are nonnegative integers, we use the notation $f = \bigO^{(q)}(r^{p})$ to indicate that $f$ satisfies $|\nabla^\ell f(x)|_g \lesssim r^{p-\ell}$ for all $0 \le \ell \le q$ and $x \neq 0$. As before, we let $r = |x|$. We first state the following global result: 

\begin{proposition}
Assume that \eqref{positivity} holds. Then for any $x \neq y$ in $M$ and any $\xi \in M$, we have
\begin{equation}\label{bounds.Green}
\begin{aligned}
\frac{1}{C} d_g(x,y)^{2k-n} & \le G_{g}(x,y) \le C d_g(x,y)^{2k-n} \quad \text{ and } \\
\big| \nabla^{\ell} G_{g}(x,y) \big|_{g}& \le C_\ell d_g(x,y)^{2k-\ell-n}.\\
\end{aligned} 
\end{equation}
Furthermore we have 
\begin{equation} \label{expansion.Green}
\begin{aligned}
G_{g}\big(\exp_{\xi}^{g}x, \exp_{\xi}^{g}(y) \big) & = \frac{b_{n,k}}{|x-y|^{n-2k}} \Big( 1 + \bigO^{(2k-1)}(|x-y|) \Big) \\
\end{aligned} 
\end{equation}
where $b_{n,k}$ is as in \eqref{def.bnk}. Here $C$, $C_{\ell}$ for $\ell \ge 1$ and the constants in the $\bigO(\cdot )$ term are independent of $x,y$ and $\xi$. 
\end{proposition}

\begin{proof}
Let, for $x \in \R^n$, $G_0(r) = b_{n,k} r^{2k-n}$, so that $\Delta_0^k G_0 = \delta_0$. Let, for $x \neq y$ in $M$, $H_0(x,y) = b_{n,k} d_g(x,y)^{2k-n}$. It follows from \eqref{expansion.Pg} that 
$$ P_g H_0(x, \cdot) = \delta_x + \bigO( d_g(x,y)^{2-n} ). $$ 
The local expression \eqref{expansion.Green} as well as the upper bounds in \eqref{bounds.Green} now follow from the iterative construction of $G_g$, which can be found e.g. in \cite{Aub, RobDirichlet} for the second-order case. For the polyharmonic case, see Theorem C.1 in \cite{RobPoly1}. The lower-bound in \eqref{bounds.Green} follows from \eqref{expansion.Green} when $d_g(x,y)$ is small enough, and is a consequence of \eqref{positivity} when $d_g(x,y)$ is larger.  
\end{proof}
 
We now prove refined expansions of $G_g$ in conformal normal coordinates: 

\begin{proposition}\label{expansion.Green.local}
Let $\xi \in M$, $x \in \R^n$. 

\begin{itemize}
\item Assume that $2k+1 \le n \le 2k+3$ or that $(M,g)$ is locally conformally flat. There exists $A_\xi \in \R$ such that 
\begin{equation*} 
\begin{aligned}
G_{g_\xi}\big(\xi, \exp_{\xi}^{g_\xi}(x) \big) & = \frac{b_{n,k}}{|x|^{n-2k}} + A_\xi + \bigO^{(2k-1)}(|x|) 
\end{aligned} 
\end{equation*}
as $x \to 0$. 
\item Assume that $\Weyl_{g}(\xi) = 0$ and $2k+4 \le n \le 2k+5$ . There exists $A_\xi \in \R$  and there exist homogeneous polynomials $\psi^{(4)}, \psi^{(5)}$ of respective degrees $4$ and $5$ such that 
\begin{equation*} 
\begin{aligned}
\hspace{1.5cm} G_{g_\xi}\big(\xi, \exp_{\xi}^{g_\xi}(x) \big) & = \frac{b_{n,k}}{|x|^{n-2k}}\Big( 1 + \psi^{(4)}(x) + \psi^{(5)}(x) \Big) + A_\xi + \bigO^{(2k-1)}(|x|) 
\end{aligned} 
\end{equation*}
as $x \to 0$. In addition, we have $\int_{\partial B(0,1)} \psi^{(4)} d \sigma = \int_{\partial B(0,1)} \psi^{(5)} d \sigma = 0$. 
\end{itemize}
As before, the constants in the $\bigO(\cdot )$ terms, are positive constants independent of $x,y$ and $\xi$. 
\end{proposition}
If $\xi \in M$ is fixed, we call the constant term $A_\xi$ appearing in Proposition \ref{expansion.Green.local} \emph{the mass} of $G_{g_\xi}$ at $\xi$. Similar expansions were proven in \cite{LeeParker} ($k=1$), \cite{HangYang, GongKimWei, GurskyMalchiodi} ($k=2$) and \cite{ChenHou}($k=3$). For an arbitrary $1 \le k < \frac{n}{2}$, the investigation of the mass function $A_\xi$ when $2k+1 \le n \le 2k+3$ or $(M,g)$ is locally conformally flat was first carried on in \cite{Michel}. When $2k+1 \le n \le 2k+5$, we will say that $P_g$ has positive mass at every point if we have  $A_\xi >0$ for all $ \xi \in M$, that is 
\begin{equation} \label{positive:mass}
\begin{aligned}
\hspace{1.2cm}&\text{For } 2k+1 \le n \le 2k+3:\, A_\xi >0\text{ for all } \xi \in M.\\
& \text{For } 2k+4 \le n \le 2k+5: \, A_\xi >0 \text{ for all } \xi \in M \text{ with}  \Weyl_g(\xi)=0.
\end{aligned}
\end{equation}

\begin{proof}
Let $\xi \in M$ be fixed and let, for $x \in \R^n$, $G_0(r) = b_{n,k} r^{2k-n}$. In the following we use the notations of Proposition~\ref{GJMS.exp1} and we let $\mathfrak{g} = \big( \exp_{\xi}^{g_\xi} \big)^*g_{\xi}$. Let $\iota(M, g_\xi)$ be the injectivity radius of $(M,g_\xi)$. Observe that for a fixed $x \neq 0$,  
$$G_0(r) = b_{n,k}\mathfrak{c}_{n,k}^{- \frac{n-2k}{2}} \lim_{\mu \to 0} \mu^{- \frac{n-2k}{2}} \tilde{U}_{\mu}(x),$$
where $\mathfrak{c}_{n,k}$ is as in \eqref{bubble1} and $\tilde{U}_\mu$ is given by \eqref{bubble3.1}. We may thus compute $ \big(P_{\mathfrak{g}} - \Delta_0^k \big) G_0$ by formally choosing $\mu = 0$ in the right-hand side of \eqref{GJMS.exp1.eq}.

Assume first that $(M,g)$ is locally conformally flat. Then  $P_{g_\xi} = \Delta_0^k$, and $P_{\mathfrak{g}} \big( G_{g_\xi}\big(\xi, \exp_{\xi}^{g_\xi}(\cdot) \big) - G_0 \big) = 0$ in $B(0,  \iota(M, g_\xi))$. Standard elliptic theory then shows that $G_{g_\xi}\big(\xi, \exp_{\xi}^{g_\xi}(\cdot) \big) - G_0$ is smooth in a neighbourhood of the origin and Proposition \ref{expansion.Green.local} follows in this case.
 
Assume now that $2k+1 \le n \le 2k+3$. Then \eqref{GJMS.exp1.eq} shows that  for $r < \iota(M, g_\xi)$,
$$ P_{\mathfrak{g}} \big( G_{g_\xi}\big(\xi, \exp_{\xi}^{g_\xi}(\cdot) \big) - G_0 \big)(x) =  - \big(P_{\mathfrak{g}} - \Delta_0^k \big) G_0(x) =  \bigO( r^{4-n})$$
holds. Since $2k+1 \le n \le 2k+3$, standard elliptic theory shows that $H = G_{g_\xi}\big(\xi, \exp_{\xi}^{g_\xi}(\cdot) \big) - G_0$ is a H\"older-continuous function in $B(0,\iota(M, g_\xi))$. Local elliptic estimates then show that $H - H(0) = \bigO^{(2k-1)}(r)$ and Proposition \ref{expansion.Green.local} follows. This case was already investigated in \cite[Th\'eor\`eme 2.3]{Michel}. 

We now turn to the higher-dimensional case. We first recall some well-known results on homogeneous polynomials in $\R^n$. If $\ell \ge 0$ is an integer we let $P_\ell$ denote the set of homogeneous polynomials of degree $\ell$ in $\R^n$ and by $H_\ell$ the set of harmonic homogeneous polynomials of degree $\ell$ in $\R^n$. It is well-known (see \cite[Section 5]{LeeParker}) that for all $\ell \ge 0$ we have 
$$ P_\ell = \bigoplus_{p=0}^{[\frac{\ell}{2}]} r^{2p} H_{\ell-2p}, $$
that for $\ell \ge 2$, $r^2 \Delta_0 - \lambda: P_{\ell} \to P_{\ell}$ is invertible if and only if $\lambda \not \in \big\{ -  2p(n-2+2\ell-2p), 0 \le p \le [\frac{\ell}{2}]\big\}$, and that the corresponding eigenspaces are given by $r^{2p} H_{\ell-2p}$. In the following we denote by $g_0$ the round metric in $\mathbb{S}^{n-1}$ and we let $\Delta_{g_0} = - \text{div}_{g_0}(\nabla \cdot)$. If $\psi \in P_\ell$ we let $ \tilde{\psi} = \psi_{|\mathbb{S}^{n-1}}$. If $q$ is any real number it is easily seen that 
\begin{equation} \label{poly.homogenes.1}
\begin{aligned}
 r^2\Delta_0 (r^{q} \psi)  & = r^{\ell+q} \big( \Delta_{g_0} \tilde{\psi} - (q+ \ell)(q+\ell+n-2) \tilde{\psi}\big) \\
  & = r^{q} \big( \Delta_{g_0} \psi - (q+ \ell)(q+\ell+n-2) \psi\big) 
  \end{aligned} 
  \end{equation}
  in $\R^n \backslash \{0\}$. For $q=0$ the latter yields $r^2\Delta_0 \psi  = \big( \Delta_{g_0} \psi - \ell(\ell+n-2) \psi\big)$, and thus 
$$ \begin{aligned}
r^2\Delta_0 \big( r^{q} \psi \big) = r^{q} \Big( r^2 \Delta_0 - q(q+2\ell+n-2) \Big) \psi \quad \text{ in } \R^n \backslash \{0\}. 
\end{aligned} $$
If $2-n <q+\ell < 0$ the operator $r^2 \Delta_0 - q(q+2\ell+n-2)$ is invertible on $P_\ell$. As a consequence, for any $\ell \ge 0$ and  $T \in P_{\ell}$, there exists $\psi \in P_{\ell}$ such that $\Delta_0(r^{q} \psi) = r^{q-2} T$. Iterating \eqref{poly.homogenes.1} shows that 
\begin{equation} \label{induction.poly.homogenes}
 \begin{aligned} \Delta_0^k(r^{q} \psi ) & = r^{q+\ell-2k} \prod_{p=0}^{k-1}\Big( \Delta_{g_0} \tilde{\psi} - (q+\ell-2p)(q+\ell-2p+n-2) \tilde{\psi}\Big)  \\
\end{aligned} . 
\end{equation}
If $2k-n < q < 0$ the operator appearing in the previous line is invertible. As a consequence, for any $\ell \ge 0$ and $T \in P_\ell$, 
\begin{equation} \label{inversion.homogenes}
\text{ there exists }  \psi \in P_{\ell} \text{ such that } \Delta_0^k(r^{q} \psi) = r^{q-2k} T \quad \text{ in } \R^n \backslash \{0\}.
\end{equation}
Assume now that $2k+4 \le n\le2k+5$ and $\Weyl_{g}(\xi) = 0$. Expansion \eqref{GJMS.exp1.eq} together with \eqref{defPsiu} shows that 
\begin{equation} \label{GJMS.exp.variante}
\begin{aligned}
\big( P_{\mathfrak{g}} - \Delta_0^k \big) G_0& = c_2  (S_{g_\xi})_{;ab}(\xi) \frac{x^a x^b}{r^{n-2}} + c_3 \Delta_{g_\xi}( \Ricci_{ab}(\xi)) \frac{x^a x^b}{r^{n-2}} \\
& + \frac{R^{(5)}(x)}{r^{n}} + \bigO( r^{6-n})
\end{aligned} 
\end{equation}
for $x \neq 0$, where $R^{(5)} \in P_5$. We let in what follows 
$$ T_1(x) =  r^2 \big[  c_2  (S_{g_\xi})_{;ab}(\xi) + c_3 \Delta_{g_\xi}( \Ricci_{ab}(\xi)) \big] x^a x^b \quad \text{ and } \quad T_2(x) = \psi^{(5)}(x). $$
We apply \eqref{inversion.homogenes} with $q=2k-n$: this yields two homogeneous polynomials $\psi^{(4)} \in P_4$ and $\psi^{(5)} \in P_5$ such that 
$$ \Delta_0^k \big( r^{2k-n}\psi^{(4)} \big) = r^{-n}T_1 \quad \text{ and } \Delta_0^k \big( r^{2k-n} \psi^{(5)} \big) =  r^{-n}T_2. $$
Using \eqref{expansion.Pg} it is easily seen that we have
\begin{equation}\label{inversion.homogenes.2}
\begin{aligned}
 \big( P_{\mathfrak{g}} - \Delta_0^k \big)(r^{2k-n}\psi^{(4)} )  & = \bigO( r^{6-n}) \quad \text{ and } \\
  \big( P_{\mathfrak{g}} - \Delta_0^k \big) (r^{2k-n}\psi^{(5)} )  & = \bigO( r^{7-n}).
 \end{aligned}
 \end{equation}
 Combining \eqref{GJMS.exp.variante} with \eqref{inversion.homogenes.2} finally shows that 
 \begin{equation*} 
 \begin{aligned}
& P_{\mathfrak{g}} \Big( G_{g_\xi}\big(\xi, \exp_{\xi}^{g_\xi}(\cdot) \big) - G_0 - r^{2k-n}\psi^{(4)} - r^{2k-n}\psi^{(5)} \Big) 
 =  \bigO( r^{6-n}). 
 \end{aligned} 
\end{equation*}
When $n \in \{2k+4, 2k+5\}$ standard elliptic theory then shows that 
$$ H: = G_{g_\xi}\big(\xi, \exp_{\xi}^{g_\xi}(\cdot) \big) - G_0 - r^{2k-n}\psi^{(4)} - r^{2k-n}\psi^{(5)}  $$
is H\"older continuous in a neighbourhood of the origin. As before, local elliptic estimates then show that $H - H(0) = \bigO^{(2k-1)}(r)$, which proves the first part of Proposition \ref{expansion.Green.local}.

It remains to prove that $\int_{\mathbb{S}^{n-1}} \psi^{(4)} d \sigma = \int_{\mathbb{S}^{n-1}} \psi^{(5)} d \sigma =0$. First, since $G_{g_\xi}$ solves $P_{g_\xi} G_{g_\xi}(\xi,\cdot) = \delta_\xi$ and $G_0$ solves $\Delta_0^k G_0 = \delta_0$, we have 
\begin{equation} \label{DL.green.homogene.1}
 \begin{aligned}
 & \int_{\partial B(0,r)}\big( P_{\mathfrak{g}} - \Delta_0^k \big) G_{g_\xi}(\xi, \cdot) d \sigma \\
 &  = -  \int_{\partial B(0,r)}\Delta_0^k G_{g_\xi}(\xi, \cdot) d \sigma \\
 & = -  \int_{\partial B(0,r)} \Delta_0^k \big( r^{2k-n} \psi^{(4)} + r^{2k-n} \psi^{(5)} \big) d \sigma + \bigO (r^{n-2k} ) \\
 & = - C(n,k) r^3  \int_{\partial B(0,r)}\psi^{(4)} d \sigma - C(n,k) r^4 \int_{\partial B(0,r)}\psi^{(5)} d \sigma+ \bigO (r^{n-2k} ),
\end{aligned} \end{equation}
for a positive constant $C(n,k)$, where we used the well-known following property for homogeneous polynomials: if $\psi \in P_\ell$ then 
$$\int_{\partial B(0,1)} \psi d \sigma = \frac{1}{\ell (n+\ell-2)}\int_{\partial B(0,1)} \Delta_0 \psi d \sigma.$$
Independently, and using \eqref{GJMS.exp.variante} and \eqref{inversion.homogenes.2}, we have 
\begin{equation} \label{DL.green.homogene.2}
 \begin{aligned}
 & \int_{\partial B(0,r)}\big( P_{\mathfrak{g}} - \Delta_0^k \big) G_{g_\xi}(\xi, \cdot) d \sigma \\
 &  =  \int_{\partial B(0,r)}\big( P_{\mathfrak{g}} - \Delta_0^k \big) G_0 d \sigma + \bigO(r^{5})+ \bigO(r^{n-2k}) \\
 & = \bigO(r^{5}) + \bigO(r^{n-2k}) ,
\end{aligned} \end{equation}
where we used the oddness of $R^{(5)}$ and assumption $\Weyl_g(\xi) = 0$ which ensures that $\int_{\partial B(0,r)} (S_{g_\xi})_{;ab}(\xi)x_a x_b d \sigma = 0$ for every $r >0$. Combining 
\eqref{DL.green.homogene.1} and \eqref{DL.green.homogene.2} shows that 
$$ r^3  \int_{\partial B(0,r)}\psi^{(4)} d \sigma + r^4 \int_{\partial B(0,r)}\psi^{(5)} d \sigma = \bigO(r^{5}) + \bigO(r^{n-2k}). $$
If $n=2k+4$ this proves that $ \int_{\partial B(0,r)}\psi^{(4)} d \sigma = 0$, while if $n = 2k+5$ this proves that $ \int_{\partial B(0,r)}\psi^{(4)} d \sigma =  \int_{\partial B(0,r)}\psi^{(5)} d \sigma=0$. This concludes the proof of Proposition \ref{expansion.Green.local}. 
\end{proof}

\begin{remark}
Let $\psi \in H_2$ be a harmonic homogeneous polynomial of degree $2$. If $\ell \ge 0$ is any integer, direct computations show that $ \Delta_0 \big( r^{-\ell} \psi \big)= \ell(n-\ell+2) \psi$, and iterating the latter shows that 
\begin{equation} \label{Delta_k_T}
\Delta_0^{k} \big( r^{2k+2 - n} \psi  \big) = C  r^{2-n} \psi, 
\end{equation}
for some nonzero constant $C = C(n,k)$. If we assume that $\Weyl_{g_{\xi}}\(\xi\) = 0$, 
$$x \in \R^n \mapsto c_2 (S_{g_\xi})_{;ab}(\xi) x^a x^b + c_3 \Delta_{g_\xi}( \Ricci_{ab}(\xi)) x^a x^b $$
 is harmonic. Using \eqref{Delta_k_T} we may thus choose 
 $$\psi^{(4)}(x) = C' r^2 \Big[ c_2 (S_{g_\xi})_{;ab}(\xi) x^a x^b + c_3 \Delta_{g_\xi}( \Ricci_{ab}(\xi)) x^a x^b\Big] $$ 
 in \eqref{inversion.homogenes.2} for some constant $C'$ when $n \ge 2k+4$. 
\end{remark}

We conclude by estimating the Pohozaev boundary term involving $G_g$:

\begin{proposition}\label{masse.Green}
Let $\xi \in M$ be fixed. Assume that $2k+1 \le n \le 2k+3$ or that $2k+4 \le n \le 2k+5$ and $\Weyl_g(\xi) = 0$. Let, for $x \in \R^n$, $\mathcal{G}(x) = G_{g_\xi}\big(\xi, \exp_{\xi}^{g_\xi}(x) \big)$ and let $h \in C^{2k}(\overline{B(0,1)})$. Let $\mathcal{P}_k(r; \cdot)$ be given by \eqref{poho.id01}. We have 
$$\lim_{r \to 0} \mathcal{P}_k(r; \mathcal{G}+ h) = c_{n,k} \big(A_\xi + h(0) \big) $$
where $c_{n,k}$ is a positive constant depending only on $n$ and $k$ and $A_\xi$ is as in the statement of Proposition \ref{expansion.Green.local}.
\end{proposition}

\begin{proof}
We keep the notation of the proof of Proposition \ref{expansion.Green.local}. By Proposition \ref{expansion.Green.local} we may write, for $x$ small enough,
$$ \mathcal{G}(x)  + h(x)  = G_0(r) + A_\xi + h(0) +  R(x), $$
where, since $h \in C^{2k}(\overline{B(0,1)})$,
$$ R(x) = r^{2k-n} \psi^{(4)} + r^{2k-n} \psi^{(5)} + \bigO^{(2k-1)}(r), $$
and where $\psi^{(4)}$ and $\psi^{(5)}$ are homogeneous polynomials of degree $4$ and $5$ (with the convention that $\psi^{(4)} = 0$ and $\psi^{(5)} = 0$ when $2k+1 \le n \le 2k+3$). We let $\Phi_{k,r}(\cdot, \cdot)$ be the bilinear form associated to $\mathcal{P}_k(r;\cdot)$ and given by \eqref{bilineaire.pair} and \eqref{bilineaire.impair}. Let $\tilde{A}_\xi = A_\xi + h(0)$. By bilinearity, we have, for a fixed value of $r$,
$$\begin{aligned} 
 \mathcal{P}_k(r; \mathcal{G}) = \mathcal{P}_k(r; G_0(r) + \tilde{A}_\xi ) +  \mathcal{P}_k(r; R) + 2 \Phi_{k,r} (G_0(r) + \tilde{A}_\xi, R ).
 \end{aligned} $$
By definition of $R$ we have $|\nabla^{\ell}R(x)\big| \lesssim r^{4+2k-\ell-n}$ for $\ell \ge 0$. As a consequence, \eqref{poho.id01} shows that 
$$  \mathcal{P}_k(r; R)   = \bigO \big(r^{8+2k-n} \big) = \smallo(1) $$
as $r \to 0$, since $n \le 2k+5$. We estimate the bilinear term. First, straightforward computations using \eqref{bilineaire.pair} and \eqref{bilineaire.impair} show that 
$$  \Phi_{k,r} (G_0(r) + \tilde{A}_\xi, R ) =  \Phi_{k,r} (G_0(r) + \tilde{A}_\xi,  r^{2k-n} \psi^{(4)} + r^{2k-n} \psi^{(5)}  ) + \smallo(1) $$
as $r \to 0$. We now claim that the following holds: 
\begin{equation} \label{final.poho.masse}
 \Phi_{k,r} \big(G_0(r) + \tilde{A}_\xi,  r^{2k-n} \psi^{(4)} + r^{2k-n} \psi^{(5)} \big) = 0 
\end{equation}
for every $r >0$. We prove \eqref{final.poho.masse}. By linearity we just need to prove that $ \Phi_{k,r} \big(G_0(r) + \tilde{A}_\xi,  r^{2k-n} \psi^{(i)}\big) = 0$ for $i=4,5$. We prove it for $i=4$ since the proof for $i=5$ is identical and we let for simplicity $R_0 = r^{2k-n} \psi^{(4)}$. Observe first that the function $G_0 +\tilde{A}_\xi$ is radial. As a consequence, the functions $\Delta_0^{i}(G_0 +\tilde{A}_\xi)$, $\partial_\nu \Delta_0^{i}(G_0 +\tilde{A})_\xi$ and $x^a \partial_a (\Delta_0^{i}(G_0 +\tilde{A}_\xi))$ are also radial for any $i \ge 0$. A careful inspection of \eqref{bilineaire.pair} and \eqref{bilineaire.impair} shows that in order to prove \eqref{final.poho.masse} it is enough to prove that for any integer $i \ge 0$ we have
\begin{equation} \label{dernieres.integrales} \begin{aligned} \int_{\partial B(0,r)} \Delta_0^i R_0 d \sigma & =   \int_{\partial B(0,r)} \partial_\nu \Delta_0^i R_0 d \sigma =  \int_{\partial B(0,r)} \partial_\nu \Delta_0^i \big( x^a \partial_a R_0 \big) d \sigma \\
& =\int_{\partial B(0,r)} \Delta_0^i\big( x^a \partial_a  R_0) d \sigma = \int_{\partial B(0,r)}\partial_\nu \big( x^a \partial_a  \Delta^i R_0 \big) d \sigma \\
& =  \int_{\partial B(0,r)} x^a \partial_a \Delta^i  R_0 d \sigma = 0. 
\end{aligned}\end{equation}
These equalities will again follow from simple properties of homogeneous functions. For instance, if $i \ge 0$, using  \eqref{induction.poly.homogenes} we get that $\Delta_0^i R_0$ is a homogeneous function of order $4+2k-n-2i$ and therefore $x^a \partial_a ( \Delta_0^i R_0) = (4+2k-n-2i) \Delta_0^i R_0 $. As a consequence, 
$$ \begin{aligned}
 \int_{\partial B(0,r)} x^a \partial_a \Delta^i  R_0 d \sigma &  =(4+2k-n-2i)  \int_{\partial B(0,r)} \Delta_0^i R_0 d \sigma = 0,
 \end{aligned} $$ 
 where the last equality follows from \eqref{induction.poly.homogenes} and since $\int_{\partial B(0,r)} \psi^{(4)} d \sigma = 0$ by Proposition \ref{expansion.Green.local}. The other integrals in \eqref{dernieres.integrales} are computed in the same way since $\partial_\nu = x^a \partial_a$, and this proves \eqref{final.poho.masse}. With \eqref{final.poho.masse} we have thus proven that 
$$  \mathcal{P}_k(r; \mathcal{G}) =  \mathcal{P}_k(r; G_0(r) + \tilde{A}_\xi ) + \smallo(1)$$
as $r \to 0$, and Proposition \ref{masse.Green} follows from Lemma \ref{sing+har}.
\end{proof}
\smallskip

\section{A technical lemma}\label{sec.giraud}

We state and prove a frequently used integral lemma: 

\begin{lemma}[A global Giraud type lemma] \label{lemma.Giraud}
Let $1\leq p<n$ and $q<n$ be integers. Let $(\rho_\alpha)_\alpha$ be a sequence of positive numbers with $\rho_\alpha \to + \infty$ as $\alpha \to + \infty$. Let $\xi_{\alpha}\in B(0,\rho_{\alpha})$ be a sequence of points. We have
\begin{align}\label{giruad}
\int_{B(0,\rho_{\alpha})}(1+|z|)^{q-n}&|\xi_{\alpha}-z|^{p-n}\,dz\notag\\
&\lesssim\left\{\begin{aligned}&~\rho_{\alpha}^{\,p+q-n} &&\text{if }n<p+q,\\&\ln\(2+\rho_{\alpha}\)&&\text{if }n=p+q,\\&\(1+|\xi_{\alpha}|\)^{p+q-n}&&\text{if }n>p+q\text{ and }q>0,\\
&\(1+|\xi_{\alpha}|\)^{p-n}\ln\(2+|\xi_{\alpha}|\)&&\text{if } q=0,\\
&\(1+|\xi_{\alpha}|\)^{p-n}&&\text{if }q<0.\end{aligned}\right.
\end{align}
\end{lemma}

\begin{proof}
For $z \in B(0, \rho_\alpha)$ we let $\mathcal{K}(\xi_{\alpha},z):=(1+|z|)^{q-n}|\xi_{\alpha}-z|^{p-n}$. 

Assume first that $n \le p+q$. We decompose the integral in the three regions $B(\xi_{\alpha},|\xi_{\alpha}|/2)$, $B(\xi_{\alpha},3|\xi_{\alpha}|/2)\setminus B(\xi_{\alpha},|\xi_{\alpha}|/2)$ and $B(0,\rho_{\alpha})\setminus B(\xi_{\alpha},3|\xi_{\alpha}|/2)$. First, if $z \in B(\xi_{\alpha},|\xi_{\alpha}|/2)$ we have $|z-\xi_{\alpha}|\leq |\xi_{\alpha}|/2\leq|z|$. Therefore, 
\begin{align*}
&\int_{B(\xi_{\alpha},|\xi_{\alpha}|/2)}\mathcal{K}(\xi_{\alpha},z)\,dz\leq\int_{B(\xi_{\alpha},|\xi_{\alpha}|/2)}\(1+|\xi_{\alpha}-z|\)^{q-n}|\xi_{\alpha}-z|^{p-n}\,dz\\
&\lesssim\int_{0}^{|\xi_{\alpha}|/2}\frac{r^{p-1}}{(1+r)^{n-q}}\,dr\lesssim\left\{\begin{aligned}&|\xi_{\alpha}|^{p+q-n}&&\text{if }n<p+q,\\&\ln\(1+|\xi_{\alpha}|\)&&\text{if }n=p+q.\end{aligned}\right.
\end{align*}
If now $z \in B(\xi_{\alpha},3|\xi_{\alpha}|/2)\setminus B(\xi_{\alpha},|\xi_{\alpha}|/2)$ we have $|z|\leq 3|z-\xi_{\alpha}|$. Therefore 
\begin{align*}
&\int_{B(\xi_{\alpha},3|\xi_{\alpha}|/2)\setminus B(\xi_{\alpha},|\xi_{\alpha}|/2)}\mathcal{K}(\xi_{\alpha},z)\,dz\lesssim\int_{B(0,9|\xi_{\alpha}|/2)}\(1+|z|\)^{q-n}|z|^{p-n}\,dz\\
&\lesssim\int_{0}^{9|\xi_{\alpha}|/2}\frac{r^{p-1}}{(1+r)^{n-q}}dr\lesssim\left\{\begin{aligned}&|\xi_{\alpha}|^{p+q-n}&&\text{if }n<p+q,\\&\ln\(1+|\xi_{\alpha}|\)&&\text{if }n=p+q.\end{aligned}\right.
\end{align*}
If finally $z \in B(0,\rho_{\alpha})\setminus B(\xi_{\alpha},3|\xi_{\alpha}|/2)$ we have $|z-\xi_{\alpha}|\leq 3|z|\leq5 |z-\xi_{\alpha}|$ and therefore
\begin{align*}
&\int_{B(0,\rho_{\alpha})\setminus B(\xi_{\alpha},3|\xi_{\alpha}|/2)}\mathcal{K}(\xi_{\alpha},z)\,dz\lesssim \int_{B(0,\rho_{\alpha})\setminus B(0,|\xi_{\alpha}|/2)}\(1+|z|\)^{q-n}|z|^{p-n}\,dz\\
&\lesssim\int_{|\xi_{\alpha}|/2}^{\rho_{\alpha}}\frac{r^{p-1}}{(1+r)^{n-q}}\,dr\lesssim\left\{\begin{aligned}&\rho_{\alpha}^{~p+q-n}&&\text{if }n<p+q\\&\ln\(1+\rho_{\alpha}\)&&\text{if }n=p+q .\\
\end{aligned}\right.
\end{align*}

Assume now that $n>p+q$. We write this time
\begin{align*}
&\int_{B(0,\rho_{\alpha})}\mathcal{K}(\xi_{\alpha},z)\,dz\\
&=\(1+|\xi_{\alpha}|\)^{p+q-n}\int_{B\(0,\frac{\rho_{\alpha}}{1+|\xi_{\alpha}|}\)}\left|\frac{\xi_{\alpha}}{1+|\xi_{\alpha}|}-z\right|^{p-n}\(\frac{1}{1+|\xi_{\alpha}|}+|z|\)^{q-n}\,dz\notag\\
&\lesssim\(1+|\xi_{\alpha}|\)^{p+q-n}\int_{B\(\frac{\xi_{\alpha}}{1+|\xi_{\alpha}|},\frac{1}{2}\)}\left|\frac{\xi_{\alpha}}{1+|\xi_{\alpha}|}-z\right|^{p-n}\,dz\,+\(1+|\xi_{\alpha}|\)^{p+q-n}\notag\\
&\times\int_{B\(0,1/2\)}\(\frac{1}{1+|\xi_{\alpha}|}+ |z|\)^{q-n}\,dz\,+\(1+|\xi_{\alpha}|\)^{p+q-n}\times\notag\\
&\int_{B\(0,\frac{\rho_{\alpha}}{1+|\xi_{\alpha}|}\)\setminus B(0,1/2)\cup B\(\frac{\xi_{\alpha}}{1+|\xi_{\alpha}|},\frac{1}{2}\)}\left|\frac{\xi_{\alpha}}{1+|\xi_{\alpha}|}-z\right|^{p-n}\(\frac{1}{1+|\xi_{\alpha}|}+|z|\)^{q-n}\,dz\notag\\
&\lesssim\(1+|\xi_{\alpha}|\)^{p+q-n}\(~1+\int_{0}^{1/2}\frac{r^{n-1}}{\(\frac{1}{1+|\xi_{\alpha}|}+ r\)^{n-q}}\,dr+\int_{1/2}^{+\infty}\frac{1}{r^{n-p-q}}\,dr\)\notag\\
&\lesssim\(1+|\xi_{\alpha}|\)^{p+q-n}\(~1+\(\frac{1}{1+|\xi_{\alpha}|}\)^{q}\int_{0}^{\(1+|\xi_{\alpha}|\)/2}\frac{r^{n-1}}{\(1+ r\)^{n-q}}\,dr\)\notag\\
&\lesssim\(1+|\xi_{\alpha}|\)^{p+q-n}+\(1+|\xi_{\alpha}|\)^{p+q-n}\times\left\{\begin{aligned}&~1&&\text{if }q>0\\
&\ln\(2+|\xi_{\alpha}|\)&&\text{if } q=0\\
&\(1+|\xi_{\alpha}|\)^{-q}&&\text{if }q<0\end{aligned}\right.\\
&\lesssim\left\{\begin{aligned}&\(1+|\xi_{\alpha}|\)^{p+q-n}&&\text{if }q>0\\
&\(1+|\xi_{\alpha}|\)^{p-n}\ln\(2+|\xi_{\alpha}|\)&&\text{if } q=0\\
&\(1+|\xi_{\alpha}|\)^{p-n}&&\text{if }q<0.\end{aligned}\right.
\end{align*}
\end{proof}
\smallskip

\bibliographystyle{amsplain}
\bibliography{biblio}

\end{document}